\documentclass[11pt]{amsart}

\usepackage{mathrsfs}
\usepackage{amsfonts}
\usepackage{amssymb}
\usepackage{amsxtra}
\usepackage{dsfont}
\usepackage[dvipsnames]{xcolor}
\usepackage[english,polish]{babel}
\usepackage[shortlabels]{enumitem}

\usepackage{tikz}
\usetikzlibrary{arrows}

\usepackage[compress, sort]{cite}

\usepackage{newtxmath}
\usepackage{newtxtext}

\usepackage[margin=2.5cm, centering]{geometry}
\usepackage[colorlinks,citecolor=blue,urlcolor=blue,bookmarks=true]{hyperref}
\hypersetup{
pdfpagemode=UseNone,
pdfstartview=FitH,
pdfdisplaydoctitle=true,
pdfborder={0 0 0}, 
pdftitle={Orthogonal polynomials with periodically modulated recurrence coefficients in the Jordan block case II},
pdfauthor={Grzegorz Świderski and Bartosz Trojan},
pdflang=en-US
}

\newcommand{\CC}{\mathbb{C}}
\newcommand{\ZZ}{\mathbb{Z}}
\newcommand{\sS}{\mathbb{S}}
\newcommand{\NN}{\mathbb{N}}
\newcommand{\RR}{\mathbb{R}}

\newcommand{\calR}{\mathcal{R}}

\newcommand{\calC}{\mathcal{C}}

\newcommand{\calO}{\mathcal{O}}

\newcommand{\calA}{\mathcal{A}}

\newcommand{\calD}{\mathcal{D}}

\newcommand{\calE}{\mathcal{E}}
\newcommand{\frakB}{\mathfrak{B}}
\newcommand{\frakX}{\mathfrak{X}}
\newcommand{\frakA}{\mathfrak{A}}
\newcommand{\frakp}{\mathfrak{p}}
\newcommand{\frakf}{\mathfrak{f}}

\newcommand{\frakS}{\mathfrak{S}}
\newcommand{\frakU}{\mathfrak{U}}
\newcommand{\fraks}{\mathfrak{s}}
\newcommand{\frakr}{\mathfrak{r}}
\newcommand{\frakt}{\mathfrak{t}}
\newcommand{\fraku}{\mathfrak{u}}
\newcommand{\frakg}{\mathfrak{g}}

\newcommand{\scrD}{\mathscr{D}}

\newcommand{\vphi}{\varphi}
\newcommand{\vtheta}{\vartheta}

\newcommand{\Id}{\operatorname{Id}}

\newcommand{\sign}[1]{\operatorname{sign}({#1})}

\newcommand{\pl}[1]{\foreignlanguage{polish}{#1}}

\newcommand{\abs}[1]{\lvert {#1} \rvert}
\newcommand{\sprod}[2]{\langle {#1}, {#2} \rangle}

\newcommand{\tr}{\operatorname{tr}}

\newcommand{\sym}{\operatorname{sym}}

\newcommand{\GL}{\operatorname{GL}}
\newcommand{\SL}{\operatorname{SL}}
\newcommand{\discr}{\operatorname{discr}}

\newcommand{\Mat}{\operatorname{Mat}}

\newcommand{\sigmaEss}{\sigma_{\mathrm{ess}}}

\newcommand{\sigmaAC}{\sigma_{\mathrm{ac}}}
\newcommand{\sigmaS}{\sigma_{\mathrm{sing}}}

\newcommand{\cl}{\operatorname{cl}}

\newcommand{\ud}{{\: \rm d}}
\newcommand{\ue}{\textrm{e}}
\newcommand{\supp}{\operatornamewithlimits{supp}}

\newtheorem{theorem}{Theorem}[section]

\newtheorem{proposition}[theorem]{Proposition}
\newtheorem{lemma}[theorem]{Lemma}
\newtheorem{corollary}[theorem]{Corollary}
\newtheorem{claim}[theorem]{Claim}

\theoremstyle{plain}
\newcounter{thm}

\newtheorem{main_theorem}[thm]{Theorem}

\numberwithin{equation}{section}

\theoremstyle{definition}
\newtheorem{example}[theorem]{Example}
\newtheorem{remark}[theorem]{Remark}

\title[Orthogonal polynomials in the Jordan block case II]
{Orthogonal polynomials with periodically modulated recurrence coefficients in the Jordan block case II}

\author{Grzegorz \'{S}widerski}
\address{
	Grzegorz \'{S}widerski \\
	Mathematical Institute \\
	University of Wroc\l{}aw \\
	pl. Grunwaldzki 2 \\
	50-384 Wroc\l{}aw, Poland
}
\email{grzegorz.swiderski@math.uni.wroc.pl}

\author{Bartosz Trojan}
\address{
	\pl{
		Bartosz Trojan \\
		Institute of Mathematics \\
		Polish Academy of Sciences \\
        ul. \'{S}niadeckich 8 \\
        00-696 Warszawa, Poland}
}
\email{btrojan@impan.pl}


\begin{document}
\selectlanguage{english}

\begin{abstract}
	We study Jacobi matrices with $N$-periodically modulated recurrence coefficients when the sequence of $N$-step
	transfer matrices is convergent to a non-trivial Jordan block. In particular, we describe asymptotic behavior
	of their generalized eigenvectors, we prove convergence of $N$-shifted Tur\'an determinants as well as of
	the Christoffel--Darboux kernel on the diagonal. Finally, by means of subordinacy theory, we identify their absolutely
	continuous spectrum as well as their essential spectrum. By quantifying the speed of convergence of transfer matrices
	we were able to cover a large class of Jacobi matrices. In particular, those related to generators of birth--death processes.
\end{abstract}

\maketitle

\section{Introduction} 
\label{sec:1}
Consider two sequences $a = (a_n : n \in \NN_0)$ and $b = (b_n : n \in \NN_0)$ such that $a_n > 0$ and $b_n \in \RR$
for all $n \geq 0$. Let $A$ be the closure in $\ell^2(\NN_0)$ of the operator acting by the matrix
\[
	\calA =
	\begin{pmatrix}
		b_0 & a_0 & 0   & 0      &\ldots \\
		a_0 & b_1 & a_1 & 0       & \ldots \\
		0   & a_1 & b_2 & a_2   & \ldots \\
		0   & 0   & a_2 & b_3   &  \\
		\vdots & \vdots & \vdots  &  & \ddots
	\end{pmatrix}
\]
on finitely supported sequences. The operator $A$ is called \emph{Jacobi matrix} and its \emph{Jacobi parameters} are
the sequences $a$ and $b$. Recall that $\ell^2(\NN_0)$ is the Hilbert space of square summable complex-valued sequences
with the scalar product
\[
	\sprod{x}{y}_{\ell^2(\NN_0)} = \sum_{n=0}^\infty x_n \overline{y_n}.
\]
Its standard orthonormal basis will be denoted by $(\delta_n : n \in \NN_0)$. Namely, $\delta_n$ is the sequence having $1$
on the $n$th position and $0$ elsewhere.

Let us observe that the operator $A$ is always symmetric. However, if $A$ is unbounded, that is at least one of the sequences $a$ and $b$ is unbounded, it does not have to be self-adjoint. If it is self-adjoint, then one can define a Borel probability measure
$\mu$ as
\[
	\mu(\cdot) = \langle E_A(\cdot) \delta_0, \delta_0 \rangle_{\ell^2}
\]
where $E_A$ is the spectral resolution of the identity of $A$. Then the sequence of polynomials $(p_n : n \in \NN_0)$ satisfying
\[
	\begin{gathered}
		p_0(x) = 1, \qquad p_1(x) = \frac{x-b_0}{a_0}, \\
		a_{n-1} p_{n-1}(x) + b_n p_n(x) + a_n p_{n+1}(x)
		= x p_n(x), \qquad n \geq 1.
	\end{gathered}
\]
is an orthonormal basis in $L^2(\RR, \mu)$, that is the Hilbert space of square integrable complex-valued functions with
the scalar product
\[
	\langle f, g \rangle_{L^2(\RR, \mu)} =
	\int_{\RR} f(x) \overline{g(x)} \mu({\rm d} x).
\]
Moreover, the operator $U : \ell^2(\NN_0) \to L^2(\RR, \mu)$ defined on the basis vectors by
\[
	U \delta_n = p_n
\]
is unitary and satisfies
\[
	(U A U^{-1} f)(x) = x f(x)
\]
for every $f \in L^2(\RR, \mu)$ such that $x f \in L^2(\RR, \mu)$, see \cite[Section 6]{Schmudgen2017} for more details. 
It follows that the spectral properties of $A$ are intimately related to the properties of $\mu$. For example, $\sigmaEss(A)$ 
is the set of accumulation points of $\supp(\mu)$. Furthermore, if
\[
	\mu = \mu_{\mathrm{ac}} + \mu_{\mathrm{sing}}
\]
is the Lebesgue decomposition of $\mu$ into the absolutely continuous and the singular parts with respect to the Lebesgue
measure, then $\sigmaAC(A) = \supp(\mu_{\mathrm{ac}})$ and $\sigmaS(A) = \supp(\mu_{\mathrm{sing}})$.

Jacobi matrices are thoroughly studied. In the bounded case,  let us only refer to the recent monograph \cite{Simon2010Book} and to the references therein. For unbounded case, see e.g.
\cite{SwiderskiTrojan2019, ChristoffelI, ChristoffelII, Discrete, HintonLewis1978, Naboko2019, JanasNaboko2002, Sahbani2008,
DombrowskiPedersen2002a} and the references therein.
In this article we consider \emph{unbounded} Jacobi matrices only. 

An interesting class of unbounded Jacobi matrices is related to the so-called \emph{birth--death processes} (see, e.g. 
\cite{Schoutens2000}), that is stationary Markov processes with the discrete state space $\NN_0$. According to \cite{Karlin1957}
generators of birth--death processes correspond to the Jacobi parameters
\[
	a_n = \sqrt{\lambda_n \mu_{n+1}}, \qquad
	b_n = -\lambda_n - \mu_n
\]
where positive sequences $(\lambda_n : n \in \NN_0)$ and $(\mu_n : n \in \NN_0)$ are called the \emph{birth and death rates}, 
respectively. The simplest case is when $\lambda_n = \mu_{n+1}$, which we call \emph{symmetric}. In particular, we can consider the following example.
\begin{example} \label{ex:2} 
Let $\kappa \in (1,2)$ and set
\[
	a_n = (n+1)^\kappa, \qquad
	b_n = -(n+1)^\kappa -n^\kappa.
\]
Then $\lambda_n = (n+1)^\kappa$, $\mu_n = n^\kappa$.
\end{example}

Another interesting class of unbounded Jacobi matrices has been recently studied in \cite{Yafaev2020a}.
\begin{example} \label{ex:1} 
For $\kappa \in (1, \infty)$ and $f,g>-1$ we set
\[
	a_n = (n+1)^{\kappa} \Big( 1 + \frac{f}{n+1} \Big), \qquad
	b_n = -2 (n+1)^{\kappa} \Big( 1 + \frac{g}{n+1} \Big).
\]
\end{example}
In particular, in \cite{Yafaev2020a}, spectral properties of $A$ has been described if 
$\kappa \in (\frac{3}{2}, \infty)$ and $\kappa + 2 g - 2 f \neq 0$. 

Let us observe that in both examples the Jacobi parameters satisfy
\begin{equation} \label{eq:127}
	\lim_{n \to \infty} a_n = \infty, \quad
	\lim_{n \to \infty} \frac{a_{n-1}}{a_n} = 1, \quad
	\lim_{n \to \infty} \frac{b_n}{a_n} = -2.
\end{equation}
The aim of this article is to study spectral properties of $A$ as well as the asymptotic behavior of the associated orthogonal 
polynomials $(p_n : n \geq 0)$ for a large subclass of Jacobi parameters satisfying \eqref{eq:127} containing sequences 
from Examples \ref{ex:2} and \ref{ex:1} as special cases.
In fact, in this article we will go beyond \eqref{eq:127} by allowing the sequences $(\frac{a_{n-1}}{a_n})$ and 
$(\frac{b_n}{a_n})$ to be asymptotically periodic. To be more precise, given $N$ a positive integer, we say that Jacobi 
parameters $(a_n : n \in \NN_0)$ and $(b_n : n \in \NN_0)$ are \emph{$N$-periodically modulated} if there are two $N$-periodic
sequences $(\alpha_n : n \in \ZZ)$ and $(\beta_n : n \in \ZZ)$ of positive and real numbers, respectively, such that
\begin{enumerate}[(a)]
	\item
	$\begin{aligned}[b]
	\lim_{n \to \infty} a_n = \infty
	\end{aligned},$
	\item
	$\begin{aligned}[b]
	\lim_{n \to \infty} \bigg| \frac{\alpha_{n-1}}{\alpha_n} - \frac{a_{n-1}}{a_n} \bigg| = 0
	\end{aligned},$
	\item
	$\begin{aligned}[b]
	\lim_{n \to \infty} \bigg| \frac{\beta_n}{\alpha_n} - \frac{b_n}{a_n} \bigg| = 0
	\end{aligned}.$
\end{enumerate}
It turns out that spectral properties of $N$-periodically modulated
Jacobi matrices depend on the matrix $\frakX_0(0)$ where for any $n \geq 0$ we have set
\[
	\frakX_n(x) = \frakB_{N+n-1}(x) \frakB_{N+n-2}(x) \cdots \frakB_n(x)
\]
where
\[
	\frakB_j(x) =
	\begin{pmatrix}
		0 & 1 \\
		-\frac{\alpha_{j-1}}{\alpha_j} & \frac{x - \beta_j}{\alpha_j}
	\end{pmatrix}.
\]
More specifically, we can distinguish four cases:
\begin{enumerate}[I.]
\item \label{perMod:I}
if $|\tr \frakX_0(0)|<2$, then under some regularity assumptions on Jacobi parameters one has that $\sigma(A) = \RR$,
and it is purely absolutely continuous, see e.g. 
\cite{JanasNaboko2002, PeriodicI, PeriodicII, SwiderskiTrojan2019, JanasNaboko2001}; 

\item \label{perMod:II} 
if $|\tr \frakX_0(0)|=2$, then we have two subcases: 
\begin{enumerate}[a)]
\item  \label{perMod:IIa}
if $\frakX_0(0)$ is diagonalizable then under some regularity assumptions on Jacobi parameters there is a compact interval 
$I \subset \RR$ such that $A$ is purely absolutely continuous on $\RR \setminus I$, and it is purely discrete in the interior of 
$I$, see e.g. \cite{Dombrowski2004, Dombrowski2009, DombrowskiJanasMoszynskiEtAl2004, DombrowskiPedersen2002a, 
DombrowskiPedersen2002, JanasMoszynski2003, JanasNabokoStolz2004, Sahbani2016, Janas2012, PeriodicII, PeriodicIII,
Discrete};

\item \label{perMod:IIb}
if $\frakX_0(0)$ is \emph{not} diagonalizable then the only situation which was known is the case when either the essential spectrum of $A$ is empty or it is a half-line, see e.g. \cite{Damanik2007, Janas2001, Janas2009, Naboko2009, Naboko2010, Simonov2007, DombrowskiPedersen1995,
Naboko2019, Yafaev2020a, Pchelintseva2008, Motyka2014, Motyka2015, jordan, Dombrowski1997};
\end{enumerate}

\item \label{perMod:III}
if $|\tr \frakX_0(0)|>2$, then under some regularity assumptions on Jacobi parameters the essential spectrum of $A$ is empty, see e.g. \cite{Discrete, JanasNaboko2002, HintonLewis1978, Szwarc2002, NabokoJanas2003}; 
\end{enumerate}

Observe that in case~I the absolutely continuous spectrum fills the whole real line, whereas in the case~III it is empty. This phenomenon was originally observed in \cite{JanasNaboko2002} and it was called \emph{spectral phase transition of the first type}. Notice that the case~II corresponds to the \emph{point} where the actual phase transition occurs. In fact, in \cite[Section 5]{JanasNaboko2002} the task of analyzing the case~II was formulated as a very interesting open problem, whose analysis required finding new tools. Nowadays, the case~II.a is quite well-understood, see \cite{ChristoffelII, Discrete}. Therefore, in this article we are exclusively interested in the case II.b, which for $N = 1$ and $\alpha_n \equiv 1, \beta_n \equiv -2$ covers \eqref{eq:127}. 

Let us introduce an
auxiliary positive sequence $\gamma = (\gamma_n : n \in \NN_0)$ tending to infinity. In Examples \ref{ex:2} and \ref{ex:1} we take
$\gamma_n = a_n$ and $\gamma_n = n+1$, respectively. We say that $N$-periodically modulated Jacobi parameters $(a_n), (b_n)$ 
are $\gamma$-\emph{tempered} if
\[
	\bigg( \sqrt{\gamma_n} \Big( \frac{\alpha_{n-1}}{\alpha_n} - \frac{a_{n-1}}{a_n} \Big) : n \in \NN \bigg),
	\bigg( \sqrt{\gamma_n} \Big( \frac{\beta_n}{\alpha_n} - \frac{b_n}{a_n} \Big) : n \in \NN\bigg),
	\bigg( \frac{\gamma_n}{a_n} : n \in \NN \bigg) \in \calD_1^N.
\]
Let us recall that a sequence $(x_n : n \in \NN)$ belongs to $\calD_1^N$ if
\[
	\sum_{n=1}^\infty |x_{n+N} - x_n| < \infty. 
\]
About the sequence $\gamma$ we assume that
\begin{equation}
	\label{eq:87}
	\bigg( \sqrt{\gamma_n} \Big( \sqrt{\frac{\alpha_{n-1}}{\alpha_n}} - \sqrt{\frac{\gamma_{n-1}}{\gamma_n}} \Big) 
	: n \in \NN \bigg), 
	\bigg(\frac{1}{\sqrt{\gamma_n}} : n \in \NN \bigg)
	\in \calD_1^N,
\end{equation}
and
\begin{equation}
	\label{eq:89}
	\lim_{n \to \infty} \big( \sqrt{\gamma_{n+N}} - \sqrt{\gamma_n} \big) = 0.
\end{equation}
Moreover, we impose that
\begin{equation}
	\label{eq:88}
		\bigg(\gamma_n \big(1 - \varepsilon \big[\frakX_n(0)\big]_{11}\big)
		\Big(\frac{\alpha_{n-1}}{\alpha_n} - \frac{a_{n-1}}{a_n} \Big)
		- \gamma_n
		\varepsilon \big[\frakX_n(0) \big]_{21}
		\Big(\frac{\beta_n}{\alpha_n} - \frac{b_n}{a_n}\Big)
		: n \in \NN
		\bigg) \in \calD_1^N,
\end{equation}
where $\varepsilon = \sign{\tr \frakX_0(0)}$. To formulate the main results of this paper, we need further definitions. 
For $x \in \CC$ and $n \in \NN_0$ we define the \emph{transfer matrix} by
\[
	B_n(x) =
	\begin{pmatrix}
		0 & 1 \\
		-\frac{a_{n-1}}{a_n} & \frac{x-b_n}{a_n}
	\end{pmatrix}.
\]
We use the convention that $a_{-1} := 1$. Moreover, for a matrix 
\[
	Y = 
	\begin{pmatrix}
	y_{11} & y_{12} \\
	y_{21} & y_{22}
	\end{pmatrix}
\]
we set $[Y]_{ij} = y_{ij}$. The discriminant of $Y$ is defined as $\discr Y = (\tr Y)^2 - 4 \det Y$.

The first main result of this article identifies the absolutely continuous and the essential spectrum of the studied class
of Jacobi matrices.
\begin{main_theorem}
	\label{thm:A}
	Let $N$ be a positive integer. Let $(\gamma_n)$ be a sequence of positive numbers tending to infinity 
	and satisfying \eqref{eq:87} and \eqref{eq:89}. Let $(a_n)$ and $(b_n)$ be $\gamma$-tempered $N$-periodically modulated 
	Jacobi parameters such that $\frakX_0(0)$ is a non-trivial parabolic element. Suppose that \eqref{eq:88} holds true with
	$\varepsilon = \sign{\tr \frakX_0(0)}$. Set
	\begin{equation} \label{eq:130}
		X_n(x) = B_{n+N-1}(x) B_{n+N-2}(x) \ldots B_{n+1}(x) B_n(x).
	\end{equation}
	Then the limit
	\begin{equation} \label{eq:134}
		\tau(x) = \frac{1}{4} \lim_{n \to \infty} \frac{\gamma_{n+N-1}}{\alpha_{n+N-1}} \discr X_n(x), 
		\qquad x \in \RR,
	\end{equation}
	exists and defines a polynomial of degree at most one. Let
	\[
		\Lambda_- = \tau^{-1} \big( (-\infty, 0) \big) \quad \text{and} \quad
		\Lambda_+ = \tau^{-1} \big( (0, \infty) \big).
	\]
	If $\Lambda_- \cup \Lambda_+ \neq \emptyset$ and $A$ is self-adjoint 
	then\footnote{By $\cl(X)$ we denote the closure of the set $X$.}
	\[
		\sigmaS(A) \cap \Lambda_- = \emptyset \quad \text{and} \quad
		\sigmaAC(A) = \sigmaEss(A) = \cl(\Lambda_-).
	\]
\end{main_theorem}

Let us emphasize that Theorem \ref{thm:A} excludes the case $\Lambda_- = \Lambda_+ = \emptyset$, that is $\tau \equiv 0$. 
Moreover, Theorem~\ref{thm:A} implies that the operator $A$ is \emph{not} semi-bounded if $\Lambda_+ = \emptyset$ and 
$\Lambda_- \neq \emptyset$, because $\Lambda_- = \RR = \sigma(A)$. However, it is unclear under what hypotheses
the operator $A$ is semi-bounded when $\Lambda_+ \neq \emptyset$. Recall that in the case III a characterization of
semi-boundedness of the operator $A$ was given in \cite{NabokoJanas2003}.

The condition \eqref{eq:88} might look rather restrictive. However, it is always satisfied by Jacobi parameters studied in 
\cite{jordan} as well as for generators of symmetric birth--death processes (cf. Remark~\ref{rem:4}). Hence,  
Theorem~\ref{thm:A} can be applied to Jacobi parameters described in Example~\ref{ex:2} where for $\gamma_n = a_n$ we
get $\tau(x) = x$. Moreover, if $N=1$ the condition \eqref{eq:88} reduces to
\[
	\bigg( \gamma_n \Big( 1 + \frac{a_{n-1}}{a_n} + \varepsilon \frac{b_n}{a_n} \Big) \bigg) \in \calD_1^1.
\]
Therefore, Theorem~\ref{thm:A} can be applied to Jacobi parameters given in Example~\ref{ex:1} where for $\gamma_n = n+1$ 
we obtain $\tau(x) \equiv -\kappa - 2 g + 2 f$.

The proof of Theorem~\ref{thm:A} uses the \emph{theory of subordinacy}. It was first developed in \cite{Gilbert1987} for 
one-dimensional Schr\"{o}dinger operators on the real half-line, and later adapted to other classes of operators, see e.g. 
the survey \cite{Gilbert2005} for more details. In particular, the extension to Jacobi matrices has been accomplished in 
\cite{Khan1992}. The theory of subordinacy links asymptotic behavior of generalized eigenvectors to spectral properties 
of Jacobi matrices. Let us recall that a sequence $(u_n : n \in \NN_0)$ is a \emph{generalized eigenvector} associated to
$x \in \CC$, and corresponding to $\eta \in \RR^2 \setminus \{0\}$, if the sequence of vectors
\begin{align*}
	\vec{u}_0 &= \eta, \\
	\vec{u}_n &= 
	\begin{pmatrix}
		u_{n-1} \\
		u_n
	\end{pmatrix}, \quad n \geq 1,
\end{align*}
satisfies
\begin{equation} \label{eq:131}
	\vec{u}_{n+1} = B_n(x) \vec{u}_n, \quad n \geq 0.
\end{equation}
We often write $(u_n(\eta, x) : n \in \NN_0)$ to indicate the dependence on
the parameters. In particular, the sequence of orthogonal polynomials $(p_n(x) : n \in \NN_0)$ is the generalized eigenvector
associated to $\eta = (0,1)^t$ and $x \in \CC$. Motivated by \cite[Section~8]{Simon2008} it will be convenient to define \emph{(generalized) Christoffel--Darboux kernel} by
\[
	K_n(x, y; \eta) = \sum_{j=0}^n u_j(\eta, x) u_j(\eta, y), \quad 
	x, y \in \RR,\ \eta \in \RR^2 \setminus \{0\}.
\]
Suppose that $A$ is self-adjoint. According to \cite[Theorem 3]{Khan1992}, if for some compact interval with non-empty interior 
$K \subset \RR$,
\begin{equation} \label{eq:129}
	\liminf_{n \to \infty} \frac{K_n(x,x;\eta)}{K_n(x,x;\tilde{\eta})} < \infty \quad
	\text{for any } x \in K \text{ and } \eta, \tilde{\eta} \in \sS^1 ,\footnote{By $\sS^1$ we denote the unit sphere in $\RR^2$.}
\end{equation}
then the measure $\mu$ is absolutely continuous on $K$, and $K \subset \supp(\mu)$. Consequently, $A$ is absolutely continuous on $K$, and $K \subset \sigmaAC(A)$. This theory became a standard approach to spectral analysis of Jacobi matrices. It has also been observed that by imposing some uniformity conditions to \eqref{eq:129} more detailed information on the density of $\mu$ can be obtained, see the references in \cite[Section 4]{Gilbert2005}. In the present article we shall show that for any compact interval $K \subset \Lambda_-$ the following stronger version of \eqref{eq:129} holds true
\begin{equation}
	\label{eq:84}
	\sup_{n \in \NN_0} 
	\sup_{x \in K} 
	\sup_{\eta, \tilde{\eta} \in \sS^1} 
	\frac{K_n(x, x; \eta)}{K_n(x, x; \tilde{\eta})} < \infty.
\end{equation}
In view of \cite{Clark1993} (see also \cite{Moszynski2022a} for a different proof in a more general setup) the condition~\eqref{eq:84} implies existence of positive
constants $c_1, c_2$ such that the density of $\mu$, $\mu'$, satisfies
\begin{equation} \label{eq:128}
	c_1 < \mu'(x) < c_2
\end{equation}
for almost all $x \in K$, with respect to the Lebesgue measure. Finally, in \cite{Silva2007}, the following consequence of subordinacy theory has been established: if $A$ is self-adjoint and for some $K \subset \RR$ there
is a function $\eta : K \to \RR^2 \setminus \{0\}$ such that
\begin{equation} \label{eq:85}
	\sum_{n=0}^\infty \sup_{x \in K} \big| u_n \big( \eta(x), x \big) \big|^2 < \infty,
\end{equation}
then $K \cap \sigmaEss(A) = \emptyset$. In Theorem~\ref{thm:3}, with a help of a recently obtained variant of discrete Levinson's
type theorem (see \cite{Discrete}), we show that \eqref{eq:85} holds for every compact interval $K \subset \Lambda_+$. 
The fact that \eqref{eq:84} holds for every compact interval $K \subset \Lambda_-$ is a consequence of the following theorem.
\begin{main_theorem}
	\label{thm:B}
	Let $N$ be a positive integer. Let $(\gamma_n)$ be a sequence of positive numbers tending to infinity
	and satisfying \eqref{eq:87} and \eqref{eq:89}. Let $(a_n)$ and $(b_n)$ be $\gamma$-tempered $N$-periodically modulated 
	Jacobi parameters such that $\frakX_0(0)$ is a non-trivial parabolic element.
	Suppose that \eqref{eq:88} holds true with $\varepsilon = \sign{\tr \frakX_0(0)}$. Set
	\[
		\rho_n = \sum_{j=0}^n \frac{\sqrt{\alpha_j \gamma_j}}{a_j}.
	\]
	If $\Lambda_- \neq \emptyset$, then $A$ is self-adjoint if and only if $\rho_n \to \infty$. If it is the case, then the limit
	\begin{equation}
		\label{eq:125}
		\lim_{n \to \infty} \frac{1}{\rho_n} K_n(x, x; \eta)
	\end{equation}
	exists locally uniformly with respect to $(x, \eta) \in \Lambda_- \times \sS^1$, and defines a~continuous positive function.
\end{main_theorem}

\begin{example} 
	\label{ex:4}
	Let $\kappa \in (1, \tfrac{3}{2}]$ and $f, g > -1$ be such that
	\[
		\kappa + 2 g - 2 f < 0.
	\]
	We set
	\[
		a_n = (n+1)^{\kappa} \bigg(1 + \frac{f}{n+1}\bigg), \qquad
		b_n = 2 (n+1)^\kappa \bigg(1 + \frac{g}{n+1}\bigg).
	\]
	Since $\kappa > 1$, the Carleman condition is not satisfied, that is
	\[
		\sum_{n=0}^\infty \frac{1}{a_n} < \infty.
	\]
	As it is easy to check, we can apply Theorems \ref{thm:A} and \ref{thm:B} to the above Jacobi parameters,
	which leads to the conclusion that the corresponding Jacobi operator $A$ is self-adjoint and $\sigmaAC(A) = \RR$.
\end{example}
Example~\ref{ex:4} is inspired by examples given by Kostyuchenko--Mirzoev in \cite{Kostyuchenko1999} who provided
Jacobi parameters giving rise to self-adjoint Jacobi operators violating the Carleman's condition.
Later the original Kostyuchenko--Mirzoev class was somewhat extended and it was proven that one usually has 
$\sigmaEss(A) = \emptyset$, see e.g. \cite[Section 2.2]{JanasMoszynski2003} and \cite[Section 6.2]{Discrete}. 
To the best of our knowledge Jacobi parameters described in Example~\ref{ex:4} provide the first instances of Jacobi
operators violating the Carleman's condition such that $\sigmaAC(A) = \RR$. In contrast, a construction of self-adjoint 
Jacobi matrices with $\sigmaAC(A) = [0, \infty)$ violating Carleman's condition is well-known, see e.g. 
\cite{Dombrowski1997}.

To prove Theorem~\ref{thm:B}, we first determine asymptotic behavior of generalized eigenvectors. Then we
apply a non-trivial averaging procedure to it. The asymptotic formula is given in the following theorem.
\begin{main_theorem}
	\label{thm:C}
	Let $N$ be a positive integer. Let $(\gamma_n)$ be a sequence of positive numbers tending to infinity
	and satisfying \eqref{eq:87} and \eqref{eq:89}. Let $(a_n)$ and $(b_n)$ be $\gamma$-tempered $N$-periodically modulated 
	Jacobi parameters such that $\frakX_0(0)$ is a non-trivial parabolic element. Suppose that \eqref{eq:88} holds true with 
	$\varepsilon = \sign{\tr \frakX_0(0)}$. If $\Lambda_- \neq \emptyset$, then for each $i \in \{0, 1, \ldots, N-1 \}$ and 
	every compact interval $K \subset \Lambda_-$, there are a~continuous function $\varphi_i : \sS^1 \times K \to \CC$ and 
	$j_0 \geq 1$ such that
	\[
		\lim_{j \to \infty} \sup_{(\eta,x) \in \sS^1 \times K}
		\bigg|
		\sqrt{\frac{a_{jN+i-1}}{\sqrt{\gamma_{jN+i-1}}}} u_{jN+i}(\eta, x) 
		- 
		|\varphi_i(\eta, x)| \sin \bigg( \sum_{k=j_0}^{j-1} \theta_{k;i}(x) + \arg \varphi_i(\eta,x) \bigg) \bigg| = 0
	\]
	where $\theta_{k;i} : K \to \RR$ are some explicit continuous functions. 
	Moreover, $\varphi_i(\eta,x) = 0$ for some (and then for all) $(\eta,x) \in \sS^1 \times K$ if and only if $[\frakX_i(0)]_{21} = 0$.
\end{main_theorem}

The proof of Theorem~\ref{thm:C} is based on uniform diagonalization of transfer matrices which has been already used in
\cite{jordan}. However, in the current setup we were not able to relate $|\varphi_i(\eta,x)|$ to the density of $\mu$. 
Hence, in order to prove that $\varphi_i(\eta, x) \neq 0$ provided $[\frakX_i(0)]_{21} \neq 0$, we needed an additional
argument based on a consequence of the following theorem (see Corollary~\ref{cor:4} for details) which studies convergence of
\emph{generalized $N$-shifted Tur\'{a}n determinants}. The latter are defined as
\begin{align*}
	\scrD_n(\eta, x) 
	&=
	\det
	\begin{pmatrix}
		u_{n+N-1}(\eta,x) & u_{n-1}(\eta,x) \\
		u_{n+N}(\eta,x) & u_n(\eta,x)
	\end{pmatrix} \\
	&=
	u_{n}(\eta,x) u_{n+N-1}(\eta,x) - u_{n-1}(\eta,x) u_{n+N}(\eta,x)
\end{align*}
where $(u_n(\eta,x) : n \in \NN_0)$ is the generalized eigenvector associated to $x \in \RR$, and corresponding to 
$\eta \in \RR^2 \setminus \{0\}$. 
The (classical) shifted Tur\'{a}n determinants correspond to $\eta = (0,1)^t$. They were defined for the first time in 
\cite{Turan1950} for $N=1$, and then generalized in \cite{GeronimoVanAssche1991} to $N \geq 1$. In \cite{Turan1950} they were 
instrumental in studying the zeros of the Legendre polynomials where it was observed that they are non-negative on the support 
of their orthogonality measure, see also \cite{Karlin1960} for later developments. As it was shown in 
\cite[Theorem 7.34]{Nevai1979} and \cite[Theorem 6]{GeronimoVanAssche1991}, if $\supp(\mu)$ is compact, the asymptotic behavior
of shifted Tur\'an determinants is usually closely related to the density of $\mu$, see \cite{Nevai1983, Nevai1987} and
the survey \cite{Nevai1992}. The extension of the above phenomena to measures with unbounded support has been accomplished in
\cite{PeriodicII, PeriodicIII, SwiderskiTrojan2019, jordan}. For these reasons the following theorem is an important result on its own.
\begin{main_theorem}
	\label{thm:D}
	Let $N$ be a positive integer. Let $(\gamma_n)$ be a sequence of positive numbers tending to infinity
	and satisfying \eqref{eq:87} and \eqref{eq:89}. Let $(a_n)$ and $(b_n)$ be $\gamma$-tempered $N$-periodically modulated 
	Jacobi parameters such that $\frakX_0(0)$ is a non-trivial parabolic element.
	Suppose that \eqref{eq:88} holds true with $\varepsilon = \sign{\tr \frakX_0(0)}$. If $\Lambda_- \neq \emptyset$, 
	then for each $i \in \{0, 1, \ldots, N-1 \}$ the limit
	\begin{equation} 
		\label{eq:126}
		\lim_{\substack{n \to \infty \\ n \equiv i \bmod N}} 
		a_{n+N+1} \sqrt{\gamma_{n+N-1}} \big| \scrD_n(\eta,x) \big|
	\end{equation}
	exists locally uniformly with respect to $(x, \eta) \in \Lambda_- \times \sS^1$ and defines a~continuous positive function.
\end{main_theorem}

Let us remark that the first order asymptotics of generalized eigenvectors provided by Theorem~\ref{thm:C} is \emph{insufficient} to 
prove~\eqref{eq:126}. It is an open problem whether, similarly to \cite{SwiderskiTrojan2019, ChristoffelI, ChristoffelII, jordan}, one can relate the value of \eqref{eq:126} to the density of the measure $\mu$. We hope to return to this problem in the future.

In this article, we also consider $\ell^1$-type perturbations of Jacobi parameters $a$, $b$ satisfying hypotheses of
Theorem~\ref{thm:A}. Namely, in Section~\ref{sec:10}, we study Jacobi parameters $\tilde{a}, \tilde{b}$ of the form
\[
	\tilde{a}_n = a_n (1 + \xi_n), \qquad
	\tilde{b}_n = b_n (1+\zeta_n),
\]
where $(\sqrt{\gamma_n} \xi_n), (\sqrt{\gamma_n} \zeta_n) \in \ell^1$. We show that for sequences $\tilde{a}$ and $\tilde{b}$
the analogues of Theorems \ref{thm:A}--\ref{thm:C} hold true. In particular, we can treat the following Jacobi parameters 
\begin{example}
	\label{ex:3} 
	For $\kappa \in (1, \infty)$ and $f,g \in \RR$ we set
	\[
		\tilde{a}_n = (n+1)^{\kappa} \Big( 1 + \frac{f}{n+1} + \xi_n \Big), \qquad
		\tilde{b}_n = 2 (n+1)^{\kappa} \Big( 1 + \frac{g}{n+1} + \zeta_n \Big),
	\]
	where $(\sqrt{n} \xi_n), (\sqrt{n} \zeta_n) \in \ell^1$ and $\kappa + 2 g - 2f \neq 0$.
\end{example}
Jacobi parameters considered in Example~\ref{ex:3} under the additional restrictions $\kappa \in (\tfrac{3}{2}, \infty)$ and 
$\xi_n, \zeta_n = \calO(n^{-2})$, have recently been studied in \cite{Yafaev2020a}.

Before we close the introduction, let us mention some of the approaches used in the literature for analysis of the case II.b. In \cite{Dombrowski1997} it was observed that a certain class of Jacobi matrices related to birth--death processes can be studied by considering the restriction to a subspace of $\ell^2$ of the square of Jacobi matrices belonging to the case I with 
$b_n \equiv 0$. This method is particularly effective in describing $\sigmaAC(A)$. Next, in \cite{Naboko2009}, asymptotics of generalized eigenvectors was studied by the reduction to the analysis of a discrete variant of Ricatti equation, whereas in \cite{Pchelintseva2008, Simonov2007} the analysis was possible by applying Birkhoff--Adams theorem. Further,
in \cite{Naboko2019} by adaptation of Kooman method (see \cite{Kooman2007}) and the approach of \cite{AptekarevGeronimo2016} it was possible to obtain asymptotic behavior of generalized eigenvectors for $x \in \CC \setminus \{0\}$ as well as continuity of the  density of the measure $\mu$. A very important class of methods is motivated by the technique introduced by Harris and Lutz
in \cite{Harris1975}. In these methods for a given $i \in \{0,1, \ldots, N-1\}$ one consider the "change of variables"
\begin{equation} \label{eq:133}
	\vec{u}_{nN+i} = Z_n \vec{v}_n, \quad n \geq 0
\end{equation}
for some invertible matrices $Z_n$. Then by \eqref{eq:131} and \eqref{eq:130} the sequence $(\vec{v}_n)$ satisfies the equation
\begin{equation} \label{eq:132}
	\vec{v}_{n+1} = Z_{n+1}^{-1} X_{nN+i} Z_n \vec{v}_n, \quad n \geq 0.
\end{equation}
The matrices $Z_n$ are chosen in a way that one can apply to the system~\eqref{eq:132} Levinson's theorem. Then thanks to the
relation \eqref{eq:133}, the asymptotics of $(\vec{v}_n)$ easily leads to the asymptotics of $(\vec{u}_{nN+i} : n \geq n_0)$.
The success of this approach depends on the properties of the matrices $Z_n$. In \cite{Janas2001} the construction of these 
matrices were motivated by a formal WKB method in which, by means of an \emph{ansatz}, one guesses the form of the solution. 
This approach was later extended in \cite{Damanik2007, Janas2009, Motyka2014, Naboko2010}. It should be emphasized that the 
resulting matrices $Z_n$ were complex-valued, oscillating and unbounded.

In this work, we start by extending techniques which were successful in the prequel~\cite{jordan}. Namely, we construct matrices
$Z_n$ such that the system \eqref{eq:132} satisfies hypotheses of a uniform discrete Levinson's theorem so it belongs to Harris--Lutz paradigm. However, our matrices are
very simple and explicit (see \eqref{eq:3}), real and convergent (obviously to a singular matrix). These features lead to 
greater applicability of our approach than in the previous works. Since Jacobi parameters considered in this paper are more 
"singular" than in \cite{jordan}, we were forced to use a more general and delicate change of variables, so that we can exploit 
the condition \eqref{eq:88} to "smooth them out". Using our change of variables, the spectral properties of $A$ on $\Lambda_+$
can be derived analogously to \cite{jordan}. On $\Lambda_-$ the situation is much more involved. Namely, in \cite{jordan}, 
in order to prove that $\mu$ is absolutely continuous on every compact $K \subset \Lambda_-$, we used an explicit sequence
of probability measures $(\mu_n : n \in \NN)$ which converges weakly to $\mu$, and such that the sequence of their densities
converges uniformly on $K$ to a~continuous positive function. In the present paper this approach does not work anymore. 
To get around of this issue we apply the subordinacy theory. This requires to analyze the asymptotic behavior of Christoffel--Darboux
kernel which was possible thanks to the asymptotics obtained in Theorem~\ref{thm:C}. All of this reduces the problem to study
averages of highly oscillatory sums. For this reason we develop Lemma~\ref{lem:1}, which might be of independent interest.

The method of asymptotic analysis of generalized eigenvectors is similar to \cite{jordan}. However, in the present situation we 
had to find another argument showing positivity of the function $|\varphi_i|$. Previously, by using the convergence of densities
of the sequence $(\mu_n : n \in \NN)$, we were able to explicitly compute the value of $|\varphi_i|$ in terms of $\mu'$. In the present work we
use  certain algebraic properties of $\varphi_i$ together with Theorem~\ref{thm:D}, see Claim~\ref{clm:4} for details. 
Let us emphasize that the method of subordinacy gives the bound \eqref{eq:128} only, which is weaker than the continuity of 
$\mu'$. The drawback of the current approach compared to \cite{jordan} is that we do not get a constructive method to 
approximate the density of $\mu$. In the forthcoming article~\cite{Zeros}, by linking the asymptotic behavior of zeros of 
the polynomials $(p_n : n \in \NN_0)$ with the value of \eqref{eq:125}, we managed to prove that, under certain additional 
hypotheses, the density of the measure $\mu$ for Jacobi matrices satisfying Theorem~\ref{thm:B} is a~continuous positive
function on $\Lambda_-$. 

The article is organized as follows: In Section~\ref{sec:2} we fix notation and we formulate basic facts. Section~\ref{sec:3}
is devoted to our change of variables. In Section~\ref{sec:4} we study spectral properties of $A$ on $\Lambda_+$. Next,
in Section~\ref{sec:5} we describe uniform diagonalization of transfer matrices on $\Lambda_-$, which is used in the rest of
the article. The proof of Theorem~\ref{thm:D} is presented in Section~\ref{sec:6}. Next, in Section~\ref{sec:7}, we prove
Theorem~\ref{thm:C}. Section~\ref{sec:8} is devoted to the proof of Theorem~\ref{thm:B}. In Section~\ref{sec:9} we study
the self-adjointness of $A$. The extensions of Theorems \ref{thm:A}--\ref{thm:C} to $\ell^1$-type perturbations is achieved
in Section~\ref{sec:10}. Finally, in Section~\ref{sec:11}, we present more concrete classes of sequences to
illustrate results of this article.

\subsection*{Notation}
By $\NN$ we denote the set of positive integers and $\NN_0 = \NN \cup \{0\}$. Throughout the whole article, we write 
$A \lesssim B$ if there is an absolute constant $c>0$ such that $A \le cB$. We write $A \asymp B$ if $A \lesssim B$ and
$B \lesssim A$. Moreover, $c$ stands for a positive constant whose value may vary from occurrence to occurrence. For any
compact set $K$, by $o_K(1)$ we denote the class of functions $f_n : K \rightarrow \RR$ such that
$\lim_{n \to \infty} f_n = 0$ uniformly on $K$. 

\subsection*{Acknowledgment}
The first author was supported by long term structural funding -- Methusalem grant of the Flemish Government. This work was completed while the first author was a postdoctoral fellow at KU Leuven. The authors would like to thank referees for their very valuable suggestions.

\section{Preliminaries} \label{sec:2}
In this section we fix the notation which is used in the rest of the article.

\subsection{Stolz class}
In this section we define a proper class of slowly oscillating sequences which is motivated by \cite{Stolz1994}, see also \cite[Section 2]{SwiderskiTrojan2019}. Let $V$ be a normed space. We say that a sequence $(x_n : n \in \NN)$ of vectors from $V$ belongs to $\calD_r(V)$ for certain $r \in \NN$, if it is bounded and for each $j \in \{1,\ldots,r\}$,
\[
	\sum_{n=1}^\infty \| \Delta^j x_n \|^{\tfrac{r}{j}} < \infty
\]
where
\begin{align*}
	\Delta^0 x_n &= x_n, \\
	\Delta^j x_n &= \Delta^{j-1} x_{n+1} - \Delta^{j-1} x_n, \quad j \geq 1.
\end{align*}
If $V$ is the real line with Euclidean norm we abbreviate $\calD_{r} = \calD_{r}(V)$. Given a compact set
$K \subset \CC$ and a normed vector space $R$, we denote by $\calD_{r}(K, R)$ the case when $V$ is the space of all
continuous mappings from $K$ to $R$ equipped with the supremum norm. 
Let us recall that $\calD_r(V)$ is an algebra provided $V$ is a normed algebra.
Let $N$ be a positive integer. We say that a sequence $(x_n : n \in \NN)$ belongs to $\calD_r^N (V)$, if
for any $i \in \{0, 1, \ldots, N-1 \}$,
\[
	(x_{nN+i} : n \in \NN) \in \calD_r(V).
\]
Again, $\calD_r^N(V)$ is an algebra provided $V$ is a normed algebra. In what follows we shall use $\calD_1^N(V)$ only.

\subsection{Finite matrices}
By $\Mat(2, \CC)$ and $\Mat(2, \RR)$ we denote the space of $2 \times 2$ matrices with complex and real entries, respectively, equipped with the spectral norm. Next, $\GL(2, \RR)$ and $\SL(2, \RR)$ consist of all matrices from $\Mat(2, \RR)$ which are invertible and of determinant equal $1$, respectively. A matrix $X \in \SL(2, \RR)$ is a \emph{non-trivial parabolic} if it is not a multiple of the identity and
$|\tr X| = 2$.

Let $X \in \Mat(2, \CC)$. By $X^t$ we denote the transpose of the matrix $X$. Let us recall that symmetrization and the discriminant are defined as
\[
	\sym(X) = \frac{1}{2} X + \frac{1}{2} X^*, \quad\text{and}\quad
	\discr(X) = (\tr X)^2 - 4 \det X,
\]
respectively. Here $X^*$ denotes the Hermitian transpose of the matrix $X$. 

By $\{ e_1, e_2 \}$ we denote the standard orthonormal basis of $\CC^2$, i.e.
\[
	e_1 =
	\begin{pmatrix}
		1 \\
		0
	\end{pmatrix} \quad \text{and} \quad
	e_2 =
	\begin{pmatrix}
		0 \\
		1
	\end{pmatrix}.
\]
Lastly, for a sequence of square matrices $(C_n : n_0 \leq n \leq n_1)$ we set
\[
	\prod_{k = n_0}^{n_1} C_k = C_{n_1} C_{n_1-1} \cdots C_{n_0}.
\]

\subsection{Generalized eigenvectors}
A sequence $(u_n : n \in \NN_0)$ is a \emph{generalized eigenvector} associated to
$x \in \CC$ and corresponding to $\eta \in \RR^2 \setminus \{0\}$, if the sequence of vectors
\begin{align*}
	\vec{u}_0 &= \eta, \\
	\vec{u}_n &= 
	\begin{pmatrix}
		u_{n-1} \\
		u_n
	\end{pmatrix}, \quad n \geq 1,
\end{align*}
satisfies
\begin{equation} 
	\label{eq:108a}
	\vec{u}_{n+1} = B_n(x) \vec{u}_n, \quad n \geq 0,
\end{equation}
where $B_n$ is the \emph{transfer matrix} defined as
\begin{equation} 
	\label{eq:108}
	\begin{aligned}
	B_0(x) &= 
	\begin{pmatrix}
		0 & 1 \\
		-\frac{1}{a_0} & \frac{x-b_0}{a_0}
	\end{pmatrix} \\
	B_n(x) &= 
	\begin{pmatrix}
		0 & 1 \\
		-\frac{a_{n-1}}{a_n} & \frac{x - b_n}{a_n}
	\end{pmatrix}
	,
	\quad n \geq 1.
	\end{aligned}
\end{equation}
To indicate the dependence on the parameters, we write $(u_n(\eta, x) : n \in \NN_0)$.
In particular, the sequence of orthogonal polynomials $(p_n(x) : n \in \NN_0)$ is the generalized eigenvector associated 
to $\eta = e_2$ and $x \in \CC$. 

\subsection{Periodic Jacobi parameters}
By $(\alpha_n : n \in \ZZ)$ and $(\beta_n : n \in \ZZ)$ we denote
$N$-periodic sequences of real and positive numbers, respectively. For each $k \geq 0$, let us define polynomials
$(\mathfrak{p}^{[k]}_n : n \in \NN_0)$ by relations
\[
	\begin{gathered}
		\mathfrak{p}_0^{[k]}(x) = 1, \qquad \mathfrak{p}_1^{[k]}(x) = \frac{x-\beta_k}{\alpha_k}, \\
		\alpha_{n+k-1} \mathfrak{p}^{[k]}_{n-1}(x) + \beta_{n+k} \mathfrak{p}^{[k]}_n(x) 
		+ \alpha_{n+k} \mathfrak{p}^{[k]}_{n+1}(x)
		= x \mathfrak{p}^{[k]}_n(x), \qquad n \geq 1.
	\end{gathered}
\]
Let
\[
	\frakB_n(x) = 
	\begin{pmatrix}
		0 & 1 \\
		-\frac{\alpha_{n-1}}{\alpha_n} & \frac{x - \beta_n}{\alpha_n}
	\end{pmatrix},
	\qquad\text{and}\qquad
	\frakX_n(x) = \prod_{j = n}^{N+n-1} \mathfrak{B}_j(x), \qquad n \in \ZZ.
\]
By $\frakA$ we denote the Jacobi matrix corresponding to 
\begin{equation*}
	\begin{pmatrix}
		\beta_0 & \alpha_0 & 0   & 0      &\ldots \\
		\alpha_0 & \beta_1 & \alpha_1 & 0       & \ldots \\
		0   & \alpha_1 & \beta_2 & \alpha_2     & \ldots \\
		0   & 0   & \alpha_2 & \beta_3   &  \\
		\vdots & \vdots & \vdots  &  & \ddots
	\end{pmatrix}.
\end{equation*}

\subsection{Tempered periodic modulations}
Let $N$ be a positive integer. We say that Jacobi parameters $(a_n : n \in \NN_0)$ and $(b_n : n \in \NN_0)$
are $N$-periodically modulated if there are two $N$-periodic sequences $(\alpha_n : n \in \ZZ)$ and
$(\beta_n : n \in \ZZ)$ of positive and real numbers, respectively, such that
\begin{enumerate}[(a)]
	\item
	\label{eq:1a}
	$\begin{aligned}[b]
	\lim_{n \to \infty} a_n = \infty
	\end{aligned},$
	\item
	\label{eq:1b}
	$\begin{aligned}[b]
	\lim_{n \to \infty} \bigg| \frac{\alpha_{n-1}}{\alpha_n} - \frac{a_{n-1}}{a_n} \bigg| = 0
	\end{aligned},$
	\item
	\label{eq:1c}
	$\begin{aligned}[b]
	\lim_{n \to \infty} \bigg| \frac{\beta_n}{\alpha_n} - \frac{b_n}{a_n} \bigg| = 0
	\end{aligned}.$
\end{enumerate}
In this article we are mostly interested in \emph{tempered} $N$-periodically modulated Jacobi parameters, i.e.
we assume that there is a sequence of positive numbers $(\gamma_n : n \in \NN_0)$ tending to infinity and satisfying
\begin{equation}
	\label{eq:90}
	\bigg( \sqrt{\gamma_n} \Big( \sqrt{\frac{\alpha_{n-1}}{\alpha_n}} - \sqrt{\frac{\gamma_{n-1}}{\gamma_n}} \Big) :
	n \in \NN \bigg), 
	\bigg(\frac{1}{\sqrt{\gamma_n}} : n \in \NN\bigg)
	\in \calD_1^N,
\end{equation}
and
\begin{equation}
	\label{eq:91}
	\lim_{n \to \infty} \big( \sqrt{\gamma_{n+N}} - \sqrt{\gamma_n} \big) = 0,
\end{equation}
such that
\begin{equation} 
	\label{eq:20a}
	\bigg(\sqrt{\gamma_n} \Big(\frac{\alpha_{n-1}}{\alpha_n} - \frac{a_{n-1}}{a_n} \Big) : n \in \NN \bigg),
	\bigg(\sqrt{\gamma_n} \Big(\frac{\beta_n}{\alpha_n} - \frac{b_n}{a_n} \Big) : n \in \NN\bigg),  
	\bigg(\frac{\gamma_n}{a_n} : n \in \NN\bigg)
	\in \calD_1^N.
\end{equation}
In view of \eqref{eq:20a}, there are two $N$-periodic sequence $(\fraks_n : n \in \ZZ)$ and
$(\frakr_n : n \in \ZZ)$ such that
\begin{equation}
	\label{eq:20e}
	\lim_{n \to \infty} \bigg| \sqrt{\alpha_n \gamma_n} \Big(\frac{\alpha_{n-1}}{\alpha_n} - \frac{a_{n-1}}{a_n} \Big) 
	-  \fraks_n \bigg| = 0,
	\qquad\text{and}\qquad
	\lim_{n \to \infty}	\bigg| \sqrt{\alpha_n \gamma_n} \Big(\frac{\beta_n}{\alpha_n} - \frac{b_n}{a_n} \Big) 
	- \frakr_n \bigg| = 0.
\end{equation}
From \eqref{eq:90} it stems that
\begin{equation} \label{eq:142}
	\lim_{n \to \infty} 
	\bigg|\frac{\alpha_{n-1}}{\alpha_n} - \frac{\gamma_{n-1}}{\gamma_n} \bigg| = 0.
\end{equation}
Hence, there is $\frakt \geq 0$, such that
\begin{equation}
	\label{eq:20d}
	\lim_{n \to \infty} \frac{\gamma_n}{a_n} = \frakt.
\end{equation}
Let us observe that, if $\frakt > 0$ then with no lose of generality we can assume that $\frakt = 1$ and $\gamma_n \equiv a_n$.
Therefore, in what follows we shall assume that $\frakt \in \{0, 1\}$.

Let us define the $N$-step transfer matrix by
\[
	X_n = B_{n+N-1} B_{n+N-2} \cdots B_{n+1} B_n.
\]
Observe that for each $i \in \{0, 1, \ldots, N-1\}$,
\[
	\lim_{j \to \infty} B_{jN+i}(x) = \frakB_i(0)
\]
and
\[
	\lim_{j \to \infty} X_{jN+i}(x) = \frakX_i(0)
\]
locally uniformly with respect to $x \in \CC$. In the whole article we assume that the matrix $\frakX_0(0)$ is a
non-trivial parabolic element of $\SL(2, \RR)$. Let $T_0$ be a matrix so that
\begin{equation} \label{eq:140}
	\frakX_0(0) = \varepsilon 
	T_0 \begin{pmatrix}
		0 & 1 \\
		-1 & 2
	\end{pmatrix}
	T_0^{-1}
\end{equation}
where
\begin{equation} 
	\label{eq:61a}
	\varepsilon = \sign{\tr \frakX_0(0)}.
\end{equation}
Since
\begin{equation} \label{eq:144}
	\frakX_i(0) = 
	\frakB_{i-1}(0) \cdots \frakB_0(0) \frakX_0(0) \frakB_0^{-1} (0) \cdots \frakB_{i-1}^{-1}(0),
\end{equation}
by taking
\begin{equation} \label{eq:141}
	T_i = \frakB_{i-1}(0) \cdots \frakB_0(0) T_0,
\end{equation}
we obtain
\[
	\frakX_i(0) = \varepsilon 
	T_i 
	\begin{pmatrix}
		0 & 1 \\
		-1 & 2
	\end{pmatrix}
	T_i^{-1}.
\]
Hence,
\[
	\frakX_i(0) = \frac{\varepsilon}{\det T_i}
	\begin{pmatrix}
		\det T_i - ([T_i]_{11} + [T_i]_{12})([T_i]_{21} + [T_i]_{22}) & ([T_i]_{11}+[T_i]_{12})^2 \\
		-([T_i]_{21} + [T_i]_{22})^2 & \det T_i + ([T_i]_{11} + [T_i]_{12})([T_i]_{21} + [T_i]_{22})
	\end{pmatrix}.
\]
In particular,
\begin{equation}
	\label{eq:33}
	\frac{([T_i]_{11} + [T_i]_{12})([T_i]_{21} + [T_i]_{22})}{\det T_i} = 1 - \varepsilon [\frakX_i(0)]_{11}	
\end{equation}
and
\begin{equation}
	\label{eq:34}
	\frac{([T_i]_{21} + [T_i]_{22})^2}{\det T_i} = -\varepsilon [\frakX_i(0)]_{21}.
\end{equation}
We often assume that 
\begin{equation}
	\label{eq:92}
		\bigg(\gamma_n \big(1 - \varepsilon \big[\frakX_n(0)\big]_{11}\big)
		\Big(\frac{\alpha_{n-1}}{\alpha_n} - \frac{a_{n-1}}{a_n} \Big)
		- \gamma_n
		\varepsilon \big[\frakX_n(0) \big]_{21}
		\Big(\frac{\beta_n}{\alpha_n} - \frac{b_n}{a_n}\Big)
		: n \in \NN
		\bigg) \in \calD_1^N.
\end{equation}
Therefore, there is $N$-periodic sequence $(\fraku_n : n \in \NN_0)$ such that
\begin{equation}
	\label{eq:63}
	\lim_{n \to \infty}
	\bigg|\gamma_n \big(1 - \varepsilon \big[\frakX_n(0)\big]_{11}\big)
    \Big(\frac{\alpha_{n-1}}{\alpha_n} - \frac{a_{n-1}}{a_n} \Big)
	- 
	\gamma_n \varepsilon \big[\frakX_n(0) \big]_{21}
	\Big(\frac{\beta_n}{\alpha_n} - \frac{b_n}{a_n} \Big)
	-
	\fraku_n
	\bigg|
	=0.
\end{equation}
Let us define 
\begin{equation}
	\label{eq:32}
	\tau(x) = \frac{1}{4} \frakS^2 - \upsilon(x)
\end{equation}
where
\begin{equation}
	\label{eq:43}
	\upsilon(x) = \frakt \varepsilon x \sum_{i' = 0}^{N-1} \frac{[\frakX_{i'}(0)]_{21}}{\alpha_{i'-1}}
	-\frakU,
\end{equation}
and
\begin{equation}
	\label{eq:4}
	\frakU = \sum_{i' = 0}^{N-1} \frac{\fraku_{i'}}{\alpha_{i'-1}},
    \qquad\text{and}\qquad
	\frakS = \sum_{i' = 0}^{N-1} \frac{\fraks_{i'}}{\alpha_{i'-1}}.
\end{equation}
In view of \cite[Proposition 2.1]{jordan},
\begin{equation} \label{eq:143a}
	\sum_{i' = 0}^{N-1} 
	\frac{[\frakX_{i'}(0)]_{21}}{\alpha_{i'-1}} = 
	-\tr \frakX_0'(0),
\end{equation}
thus
\begin{equation} \label{eq:4a}
	\tau(x) = 
	\frac{1}{4} \frakS^2
	+ 
	\frakU + \frakt \varepsilon (\tr \frakX_0'(0)) x.
\end{equation}
The following proposition answers the question when $\frakX_0(0)$ is a non-trivial parabolic element of $\SL(2, \RR)$.
\begin{proposition} \label{prop:2}
Suppose that $|\tr \frakX_0(0)| = 2$. Then $\frakX_0(0)$ is a non-trivial parabolic element of $\SL(2, \RR)$ if and only if $\tr \frakX_0'(0) \neq 0$.
\end{proposition}
\begin{proof}
	The matrix $\frakX_0(0)$ is a trivial parabolic element if and only if $\frakX_{0}(0) = \varepsilon \Id$ where $\varepsilon$ 
	is defined in \eqref{eq:61a}. Then by \eqref{eq:144} we get $\frakX_i(0) \equiv \varepsilon \Id$ for all 
	$i \in \{0, 1, \ldots, N-1\}$. Consequently, by \eqref{eq:143a}, $\tr \frakX_0'(0) = 0$. On the other hand, if
	$\frakX_0(0) \neq \varepsilon \Id$, then thanks to \cite[Proposition 3]{PeriodicIII} at least one of the numbers 
	$[\frakX_0(0)]_{21}$ and $[\frakX_1(0)]_{21}$ is non-zero. In view of \cite[Proposition 2.2]{jordan} we have
	\[
		\sum_{i' = 0}^{N-1} 
		\frac{|[\frakX_{i'}(0)]_{21}|}{\alpha_{i'-1}} =  
		|\tr \frakX_0'(0)|,
	\]
	thus $\tr \frakX_0'(0) \neq 0$.
\end{proof}

If $\frakt \neq 0$, then thanks to Proposition~\ref{prop:2} we have $\tr \frakX'(0) \neq 0$, so we can set
\begin{equation}
	\label{eq:39}
	x_0 = -\frac{\frakU + \tfrac{1}{4} \frakS^2}{\varepsilon \tr \frakX_0'(0)},
\end{equation}
and $\Lambda = \RR \setminus \{x_0\}$. Otherwise, we shall assume that $\tau \nequiv 0$ and we set $\Lambda = \RR$.
\begin{proposition}
	\label{prop:1}
	If \eqref{eq:142} is satisfied,
	then
	\begin{equation}
		\label{eq:50}
		\frakS = \lim_{n \to \infty} \sqrt{\frac{\gamma_n}{\alpha_n}} \Big( 1 - \frac{a_n}{a_{n+N}} \Big).
	\end{equation}
	In particular, $\frakS \geq 0$.
\end{proposition}
\begin{proof}
	Let us first observe that that
	\begin{align*}
		1 - \frac{a_{n}}{a_{n+N}} 
		=
		\frac{\alpha_{n}}{\alpha_{n+N}} - 
		\frac{a_{n}}{a_{n+N}}
		&=
		\sum_{j=0}^{N-1} 
		\Big( \prod_{\ell=j+1}^{N-1} \frac{\alpha_{n+\ell}}{\alpha_{n+\ell+1}} \Big)
		\Big( \frac{\alpha_{n+j}}{\alpha_{n+j+1}} - \frac{a_{n+j}}{a_{n+j+1}} \Big)
		\Big( \prod_{\ell=0}^{j-1} \frac{a_{n+\ell}}{a_{n+\ell+1}} \Big) \\
		&=
		\sum_{j=0}^{N-1} 
		\frac{\alpha_{n+j+1}}{\alpha_{n+N}}
		\frac{a_n}{a_{n+j}}
		\Big( \frac{\alpha_{n+j}}{\alpha_{n+j+1}} - \frac{a_{n+j}}{a_{n+j+1}} \Big).
	\end{align*}
	Thus
	\[
		\sqrt{\frac{\gamma_n}{\alpha_n}} \Big( 1 - \frac{a_{n}}{a_{n+N}} \Big) = 
		\sum_{j=0}^{N-1} 
		\frac{\alpha_{n+j+1}}{\alpha_{n+N}}
		\frac{a_n}{a_{n+j}}
		\frac{1}{\sqrt{\alpha_n}} 
		\sqrt{\frac{\gamma_n}{\alpha_{n+j+1} \gamma_{n+j+1}}} 
		\sqrt{\alpha_{n+j+1} \gamma_{n+j+1}} 
		\Big( \frac{\alpha_{n+j}}{\alpha_{n+j+1}} - \frac{a_{n+j}}{a_{n+j+1}} \Big).
	\]
	Hence, by the hypothesis~\ref{eq:1b}, \eqref{eq:142} and \eqref{eq:20e},
	\begin{align*}
		\lim_{n \to \infty}
		\sqrt{\frac{\gamma_n}{\alpha_n}} \Big( 1 - \frac{a_{n}}{a_{n+N}} \Big) &=
		\lim_{n \to \infty}
		\sum_{j=0}^{N-1} 
		\frac{\alpha_{n+j+1}}{\alpha_{n+N}}
		\frac{\alpha_n}{\alpha_{n+j}}
		\frac{1}{\sqrt{\alpha_n}} 
		\sqrt{\frac{\alpha_n}{\alpha_{n+j+1} \alpha_{n+j+1}}} 
		\fraks_{n+j+1} \\
		&=
		\lim_{n \to \infty}
		\sum_{j=0}^{N-1} \frac{\fraks_{n+j+1}}{\alpha_{n+j}} = \frakS
	\end{align*}
	and \eqref{eq:50} follows. To see the last statement, we assume, contrary to our claim, that $\frakS < 0$.
	Then there is $n_0 > 0$ such that for $n \geq n_0$,
	\[
		a_n - a_{n+N} > 0.
	\]
	Therefore, for each $n \geq n_0$,
	\[
		a_n \leq \max_{0 \leq i \leq N-1} a_{n_0+i} < \infty,
	\]
	which contradicts the hypothesis~\ref{eq:1a}.
\end{proof}

\section{The shifted conjugation} \label{sec:3}
In this section we freely use the notation introduced in Section~\ref{sec:2}.
Fix $i \in \{0, 1, \ldots, N-1\}$ and set
\begin{equation}
	\label{eq:3}
	Z_j = T_i 
	\begin{pmatrix}
		1 & 1 \\
		e^{\vtheta_j} & e^{-\vtheta_j}
	\end{pmatrix}
\end{equation}
where $T_i$ has been defined in \eqref{eq:140} and \eqref{eq:141}, and
\begin{equation}
	\label{eq:21}
	\vtheta_j(x) = \sqrt{\frac{\alpha_{i-1} \abs{\tau(x)}}{\gamma_{(j+1)N+i-1}}}.
\end{equation}
The proof of the following theorem is a generalization of \cite[Theorem 3.2]{jordan}. 
\begin{theorem}
	\label{thm:2}
	Let $N$ be a positive integer and $i \in \{0, 1, \ldots N-1\}$. Let $(\gamma_n : n \in \NN)$ be a sequence of positive 
	numbers tending to infinity and satisfying \eqref{eq:87} and \eqref{eq:89}. Let $(a_n : n \in \NN_0)$ and 
	$(b_n : n \in \NN_0)$ be $\gamma$-tempered $N$-periodically modulated Jacobi parameters such that $\frakX_0(0)$ is 
	a non-trivial parabolic element. Suppose that \eqref{eq:88} holds true with $\varepsilon = \sign{\tr \frakX_0(0)}$.
	Then for any compact interval $K \subset \Lambda$,
	\begin{equation} 
		\label{eq:15}
		Z_j^{-1} Z_{j+1} = \Id + \sqrt{\frac{\alpha_{i-1}}{\gamma_{(j+1)N+i-1}}} Q_j
	\end{equation}
	where $(Q_j)$ is a sequence from $\calD_1\big(K, \Mat(2, \RR)\big)$ convergent uniformly on $K$ to the zero matrix.
\end{theorem}
\begin{proof}
	In the proof we denote by $(\delta_j)$ a generic sequence from $\calD_1$ tending to zero which may
	change from line to line. By a straightforward computation we obtain
	\begin{align*}
		Z_j^{-1} Z_{j+1} 
		&=
		\frac{1}{\det Z_j} 
		\begin{pmatrix}
			\ue^{-\vartheta_j} & -1\\
			-\ue^{\vartheta_j} & 1
		\end{pmatrix}
		\begin{pmatrix}
			1 & 1\\
			\ue^{\vartheta_{j+1}} & \ue^{-\vartheta_{j+1}}
		\end{pmatrix} \\
		&=
		\frac{1}{\ue^{-\vartheta_{j}} - \ue^{\vartheta_{j}}}
		\begin{pmatrix}
			f_j & g_j \\
			\tilde{g}_j & \tilde{f}_j
		\end{pmatrix}
	\end{align*}
	where
	\begin{alignat*}{3}
		&f_j = \ue^{-\vartheta_j} - \ue^{\vartheta_{j+1}}, &\qquad
		&g_j = \ue^{-\vartheta_j} - \ue^{-\vartheta_{j+1}} \\
		&\tilde{g}_j = - \ue^{\vartheta_j} +  \ue^{\vartheta_{j+1}} &
		&\tilde{f}_j = - \ue^{\vartheta_j} +  \ue^{-\vartheta_{j+1}}.
	\end{alignat*}
	Observe that
	\begin{equation} \label{eq:145a}
		\gamma_{n} \bigg(
		\frac{1}{\sqrt{\gamma_{n}}}
		-
		\frac{1}{\sqrt{\gamma_{n+N}}}
		\bigg)
		=
		\big(\sqrt{\gamma_{n+N}} - \sqrt{\gamma_n}\big) \sqrt{\frac{\gamma_n}{\gamma_{n+N}}}.
	\end{equation}
	Notice that
	\[
		\sqrt{\frac{\gamma_{n+N}}{\alpha_{n+N}}} - 
		\sqrt{\frac{\gamma_{n}}{\alpha_{n}}} =
		\sum_{j=0}^{N-1} 
		\bigg( 
		\sqrt{\frac{\gamma_{n+j+1}}{\alpha_{n+j+1}}} -
		\sqrt{\frac{\gamma_{n+j}}{\alpha_{n+j}}}
		\bigg).
	\]
	Hence, by $N$-periodicity of $(\alpha_n)$,
	\[
		\sqrt{\gamma_{n+N}} - \sqrt{\gamma_{n}} =
		\sum_{j=0}^{N-1} 
		\sqrt{\frac{\alpha_n}{\alpha_{n+j}}}
		\sqrt{\gamma_{n+j+1}}
		\bigg(
		\sqrt{\frac{\alpha_{n+j}}{\alpha_{n+j+1}}} - 
		\sqrt{\frac{\gamma_{n+j}}{\gamma_{n+j+1}}}
		\bigg).
	\]
	Now, by \eqref{eq:90} we obtain 
	\begin{equation} \label{eq:145b}
		\big( \sqrt{\gamma_{n+N}} - \sqrt{\gamma_{n}} : n \in \NN \big) 
		\in \calD_1^N.
	\end{equation}
	Similarly, $N$-periodicity of $(\alpha_n)$ and \eqref{eq:90} leads to
	\[
		\bigg( \sqrt{\frac{\gamma_{n-1}}{\gamma_n}} : n \in \NN \bigg) \in \calD_1^N.
	\]
	For fixed $j \in \NN$, we have
	\[
		\sqrt{\frac{\gamma_{n}}{\gamma_{n+j}}} =
		\sqrt{\frac{\gamma_{n}}{\gamma_{n+1}}}
		\sqrt{\frac{\gamma_{n+1}}{\gamma_{n+2}}} \ldots
		\sqrt{\frac{\gamma_{n+j-1}}{\gamma_{n+j}}},
	\]
	thus
	\begin{equation} 
		\label{eq:145c}
		\bigg( \sqrt{\frac{\gamma_{n}}{\gamma_{n+j}}} : n \in \NN \bigg) \in \calD_1^N.
	\end{equation}
	Hence, by \eqref{eq:145a}--\eqref{eq:145c}
	\begin{align*}
		\frac{1}{\sqrt{\gamma_{(j+2)N+i-1}}}
		=
		\frac{1}{\sqrt{\gamma_{(j+1)N+i-1}}}
		+
		\frac{1}{\gamma_{jN}} \delta_j,
	\end{align*}
	and so
	\begin{equation}
		\label{eq:23}
		\vartheta_{j+1} = 
		\vartheta_j 
		+ \frac{1}{\gamma_{jN}} \delta_j.
	\end{equation}
	Next, we compute
	\[
		\ue^{\vartheta_{j+1}}
		= 1 + \vartheta_{j+1} + \frac{1}{2} \vartheta_{j+1}^2 + \frac{1}{\gamma_{jN}} \delta_j
	\]
	and
	\[
		\ue^{-\vartheta_j} = 1 - \vartheta_j + \frac{1}{2} \vartheta_j^2 + \frac{1}{\gamma_{jN}} \delta_j.
	\]
	Hence,
	\begin{align*}
		f_j 
		&=1 - \vartheta_j + \frac{1}{2} \vartheta_j^2 
		-\bigg(1 + \vartheta_{j+1} + \frac{1}{2} \vartheta_{j+1}^2\bigg)
		+ \frac{1}{\gamma_{jN}} \delta_j\\
		&=
		-2 \vartheta_j + \frac{1}{\gamma_{jN}} \delta_j.
	\end{align*}
	Since $\frac{x}{\sinh (x)}$ is an even $\calC^2(\RR)$ function, we have
	\[
		\frac{\vartheta_{j}}{\sinh( \vartheta_{j} )} = 1 + \frac{1}{\sqrt{\gamma_{jN}}} \delta_j.
	\]
	Therefore,
	\begin{align*}
		\frac{1}{\ue^{-\vartheta_{j}} - \ue^{\vartheta_{j}}} f_j
		&=
		\frac{f_j}{-2\vartheta_{j}} \frac{\vartheta_j}{\sinh(\vartheta_{j})} \\
		&=
		\bigg(1 + \frac{1}{\sqrt{\gamma_{jN}}} \delta_j \bigg)
		\bigg(1 + \frac{1}{\sqrt{\gamma_{jN}}}\delta_j \bigg) \\
		&=
		1 + \frac{1}{\sqrt{\gamma_{jN}}} \delta_j.
	\end{align*}
	Analogously, we treat $g_j$. Namely, we write
	\begin{align*}
		g_j &= 
		1 - \vartheta_j + \frac{1}{2} \vartheta_j^2 
		-
		\bigg(1 - \vartheta_{j+1} + \frac{1}{2}\vartheta_{j+1}^2\bigg)
		+
		\frac{1}{\gamma_{jN}} \delta_j \\
		&=
		\frac{1}{\gamma_{jN}} \delta_j.
	\end{align*}
	Hence,
	\begin{align*}
		\frac{1}{\ue^{-\vartheta_{j}} - \ue^{\vartheta_{j}}} g_j
		&=
		\frac{1}{\sqrt{\gamma_{jN}}} \delta_j. 
	\end{align*}
	Similarly, we can find that
	\begin{alignat*}{2}
		\frac{1}{\ue^{-\vartheta_{j}} - \ue^{\vartheta_{j}}} \tilde{f}_j
		&=
		1 + \frac{1}{\sqrt{\gamma_{jN}}} \delta_j,\\
		\frac{1}{\ue^{-\vartheta_{j}} - \ue^{\vartheta_{j}}} \tilde{g}_j
		&=
		\frac{1}{\sqrt{\gamma_{jN}}} \delta_j.
	\end{alignat*}
	Consequently,
	\[
		Z_{j}^{-1} Z_{j+1} = \Id + \sqrt{\frac{\alpha_{i-1}}{\gamma_{(j+1)N+i-1}}} Q_j
	\]
	where $(Q_j)$ is a sequence from $\calD_1\big(K, \Mat(2, \RR)\big)$ for any compact interval
	$K \subset \Lambda$ convergent to the zero matrix proving the formula \eqref{eq:15}. 
\end{proof}

\begin{theorem} 
	\label{thm:1}
	Let $N$ be a positive integer and $i \in \{0, 1, \ldots N-1\}$. Let $(\gamma_n : n \in \NN)$ be a sequence of positive 
	numbers tending to infinity and satisfying \eqref{eq:87} and \eqref{eq:89}. Let $(a_n : n \in \NN_0)$ and 
	$(b_n : n \in \NN_0)$ be $\gamma$-tempered $N$-periodically modulated Jacobi parameters such that $\frakX_0(0)$ is 
	a non-trivial parabolic element. Suppose that \eqref{eq:88} holds true with $\varepsilon = \sign{\tr \frakX_0(0)}$. 
	Then for any compact interval $K \subset \Lambda$,
	\[
		Z_{j+1}^{-1} X_{jN+i} Z_{j} = \varepsilon \bigg( \Id 
		+ 
		\sqrt{\frac{\alpha_{i-1}}{\gamma_{(j+1)N+i-1}}}
		R_j \bigg)
	\]
	where $(R_j)$ is a sequence from $\calD_1\big(K, \Mat(2, \RR)\big)$ convergent uniformly on $K$ to
	\begin{equation}
		\label{eq:26}
		\calR_i(x) = \frac{\sqrt{\abs{\tau(x)}}}{2}
		\begin{pmatrix}
			1 & -1\\
			1 & -1
		\end{pmatrix}
		-
		\frac{\upsilon(x)}{2 \sqrt{\abs{\tau(x)}}}
		\begin{pmatrix}
			1 & 1\\
			-1 & -1
		\end{pmatrix}
		-
		\frac{\frakS}{2}
		\begin{pmatrix}
			1 & -1\\
			-1 & 1
		\end{pmatrix}.
	\end{equation}
	In particular, $\discr \calR_i = 4 \tau(x)$.
\end{theorem}
\begin{proof}
	In the following argument, we denote by $(\delta_j)$ and $(\calE_j)$ generic sequences tending to zero
	from $\calD_1$ and $\calD_1\big(K, \Mat(2, \RR)\big)$, respectively, which may change from line to line.

	Observe that by \eqref{eq:87}
	\[
		\lim_{j \to \infty} \frac{\gamma_{jN+i'-1}}{\gamma_{jN+i'}} = \frac{\alpha_{i'-1}}{\alpha_{i'}},
	\]
	and
	\[
		\bigg(\frac{\gamma_{jN+i'-1}}{\gamma_{jN+i'}} : j \in \NN \bigg) \in \calD_1.
	\]
	Hence,
	\begin{align*}
		\frac{1}{\gamma_{jN+i'}} 
		&= 
		\frac{1}{\gamma_{jN+i'-1}} \frac{\gamma_{jN+i'-1}}{\gamma_{jN+i'}} \\
		&=
		\frac{1}{\gamma_{jN+i'-1}} \frac{\alpha_{i'-1}}{\alpha_{i'}} 
		-
		\frac{1}{\gamma_{jN+i'-1}} \bigg(
		\frac{\gamma_{jN+i'-1}}{\gamma_{jN+i'}} - \frac{\alpha_{i'-1}}{\alpha_{i'}}\bigg)\\
		&=
		\frac{\alpha_{i'-1}}{\alpha_{i'}} \frac{1}{\gamma_{jN+i'-1}} + \frac{1}{\gamma_{jN}} \delta_j.
	\end{align*}
	Consequently,
	\begin{equation}
		\label{eq:17}
		\frac{\alpha_{i'}}{\gamma_{jN+i'}} = 
		\frac{\alpha_{i-1}}{\gamma_{(j+1)N+i-1}} + \frac{1}{\gamma_{jN}} \delta_j.
	\end{equation}
	Next, for $x \in K$, we have
	\begin{align}
		\nonumber
		\frac{x}{a_{jN+i'}} 
		&= \frac{x}{\gamma_{jN+i'}} \frac{\gamma_{jN+i'}}{a_{jN+i'}} \\
		\nonumber
		&=
		\frac{\frakt x}{\gamma_{jN+i'}} 
		+ \frac{x}{\gamma_{jN+i'}} \bigg(\frac{\gamma_{jN+i'}}{a_{jN+i'}} - \frakt\bigg)\\
		\label{eq:12}
		&=
		\frac{\frakt x}{\gamma_{jN+i'}}
		+ 
		\frac{1}{\gamma_{jN}}
		\delta_j.
	\end{align}
	Let
	\[
		\xi_n = \frac{\alpha_{n-1}}{\alpha_n} - \frac{a_{n-1}}{a_n}
		\quad\text{and}\quad
		\zeta_n = \frac{\beta_n}{\alpha_n} - \frac{b_n}{a_n}.
	\]
	Using \eqref{eq:12} we can write
	\begin{align*}
		B_{jN+i'}
		&=
		\begin{pmatrix}
			0 & 1 \\
			-\frac{\alpha_{i'-1}}{\alpha_{i'}} + \xi_{jN+i'} + \frac{1}{\gamma_{jN+i'}} \delta_j
			& 
			-\frac{\beta_{i'}}{\alpha{i'}}
			+\frac{x \frakt}{\gamma_{jN+i'}}
			+\zeta_{jN+i'} 
			+ \frac{1}{\gamma_{jN+i'}} \delta_j
		\end{pmatrix} \\
		&=
		\begin{pmatrix}
			0 & 1 \\
			-\frac{\alpha_{i'-1}}{\alpha_{i'}} & - \frac{\beta_{i'}}{\alpha_{i'}}
		\end{pmatrix}
		+
		\begin{pmatrix}
			0 & 0 \\
			\xi_{jN+i'} & \frac{x \frakt}{\gamma_{jN+i'}} + \zeta_{jN+i'}
		\end{pmatrix}
		+
		\frac{1}{\gamma_{jN+i'}} \calE_j \\
		&=
		\begin{pmatrix}
			0 & 1 \\
			-\frac{\alpha_{i'-1}}{\alpha_{i'}} & - \frac{\beta_{i'}}{\alpha_{i'}}
		\end{pmatrix}
		\left\{
		\Id
		+
		\frac{\alpha_{i'}}{\alpha_{i'-1}}
		\begin{pmatrix}
			-\frac{\beta_{i'}}{\alpha_{i'}} & -1 \\
			\frac{\alpha_{i'-1}}{\alpha_{i'}} & 0
		\end{pmatrix}
		\begin{pmatrix}
			0 & 0 \\
			\xi_{jN+i'} & \frac{x \frakt}{\gamma_{jN+i'}} + \zeta_{jN+i'}
		\end{pmatrix}
		+
		\frac{1}{\gamma_{jN+i'}} \calE_j \right\} \\
		&=
		\begin{pmatrix}
			0 & 1 \\
			-\frac{\alpha_{i'-1}}{\alpha_{i'}} & - \frac{\beta_{i'}}{\alpha_{i'}}
		\end{pmatrix}
		\left\{
		\Id -
		\frac{\alpha_{i'}}{\alpha_{i'-1}}
		\begin{pmatrix}
			\xi_{jN+i'} & \zeta_{jN+i'} + \frac{x \frakt}{\gamma_{jN+i'}}\\
			0 & 0
		\end{pmatrix}
		+
		\frac{1}{\gamma_{jN}} \calE_j
		\right\}.
	\end{align*}
	Hence, 
	\begin{align*}
		X_{jN+i}
		&=
		B_{jN+i+N-1} \cdots B_{jN+i+1} B_{jN+i} \\
		&=
		\frakX_{i}(0)
		\Bigg\{
		\Id
		-
		\sum_{i' = i}^{N+i-1}
		\frac{\alpha_{i'}}{\alpha_{i'-1}}
		\Big(
		\frakB_{i'-1}(0) \cdots  \frakB_i(0)\Big)^{-1}
		\begin{pmatrix}
			\xi_{jN+i'} & \zeta_{jN+i'} + \frac{x \frakt}{\gamma_{jN+i'}}\\
			0 & 0
		\end{pmatrix}
		\Big(\frakB_{i'-1}(0) \cdots \frakB_i(0)\Big) \\
		&\phantom{\frakX_i(0)\Bigg\{\Id}+
		\frac{1}{\gamma_{jN}} \calE_j
		\Bigg\},
	\end{align*}
	and so
	\begin{align*}
		&Z_{j+1}^{-1} X_{jN+i} Z_j \\
		&\qquad=
		Z_{j+1}^{-1} \mathfrak{X}_{i}(0) Z_j
		\Bigg\{
		\Id
		-
		\sum_{i' = i}^{N+i-1} 
		\frac{\alpha_{i'}}{\alpha_{i'-1}}
		\begin{pmatrix}
			1 & 1 \\
			\ue^{\vartheta_j} & \ue^{-\vartheta_j}
		\end{pmatrix}^{-1}
		T_{i'}^{-1}
		\begin{pmatrix}
			\xi_{jN+i'}& \zeta_{jN+i'} + \frac{x \frakt}{\gamma_{jN+i'}}\\
			0 & 0
		\end{pmatrix}
		T_{i'}
		\begin{pmatrix}
			1 & 1 \\
			\ue^{\vartheta_j} & \ue^{-\vartheta_j}
		\end{pmatrix} \\
		&\qquad\phantom{=Z_{j+1}^{-1} \mathfrak{X}_{i}(0) Z_{j} \Bigg\{}+
		\frac{1}{\sqrt{\gamma_{jN}}} \calE_j
		\Bigg\}.
	\end{align*}
	To find the asymptotic of the first factor, we write
	\[
		Z_{j+1}^{-1} \mathfrak{X}_{i}(0) Z_j
		=
		\frac{\varepsilon}{\ue^{-\vartheta_{j+1}} - \ue^{\vartheta_{j+1}}}
		\begin{pmatrix}
			f_j & g_j \\
			\tilde{g}_j & \tilde{f}_j
		\end{pmatrix}
	\]
	where
	\begin{alignat*}{3}
		f_j &= \ue^{{\vartheta_{j}} - \vartheta_{j+1}} + 1 - 2\ue^{\vartheta_{j}}, &\qquad
		g_j &= \ue^{{-\vartheta_{j}} - \vartheta_{j+1}} + 1 - 2\ue^{-\vartheta_{j}}, \\
		\tilde{g}_j &= -\ue^{\vartheta_{j}+\vartheta_{j+1}} -1 +2\ue^{\vartheta_{j}}, &\qquad
		\tilde{f}_j &= -\ue^{-\vartheta_{j}+\vartheta_{j+1}} -1 + 2 \ue^{-\vartheta_{j}}.
	\end{alignat*}
	Since by \eqref{eq:23}
	\[
		\ue^{{\vartheta_{j}} - \vartheta_{j+1}} = 1 + \frac{1}{\gamma_{jN}} \delta_j
		\qquad\text{and}\qquad
		\ue^{\vartheta_{j}} = 1 + \vartheta_j + \frac{1}{2} \vartheta_j^2 + \frac{1}{\gamma_{jN}} \delta_j,
	\]
	we get
	\begin{align*}
		f_j 
		&= 1 + 1 - 2 \bigg(1 + \vartheta_{j} + \frac{1}{2}\vartheta_{j}^2 \bigg)+ \frac{1}{\gamma_{jN}} \delta_j \\
		&= -2 \vartheta_{j} - \vartheta_{j}^2 + \frac{1}{\gamma_{jN}} \delta_j.
	\end{align*}
	Moreover,
	\begin{align*}
		\frac{\vartheta_j}{\vartheta_{j+1}} = 1 + \frac{1}{\sqrt{\gamma_{jN}}} \delta_j.
	\end{align*}
	Thus
	\begin{align}
		\nonumber
		\frac{1}{\ue^{-\vartheta_{j+1}} - \ue^{\vartheta_{j+1}}} f_j
		&=
		\frac{f_j}{-2\vartheta_{j}} \frac{\vartheta_j}{\vartheta_{j+1}} \frac{\vartheta_{j+1}}{\sinh \vartheta_{j+1}} \\
		\nonumber
		&=
		\bigg(1 + \frac{1}{2} \vartheta_j + \frac{1}{\sqrt{\gamma_{jN}}} \delta_j \bigg)
		\bigg(1 + \frac{1}{\sqrt{\gamma_{jN}}} \delta_j \bigg)
		\bigg(1 + \frac{1}{\sqrt{\gamma_{jN}}} \delta_j\bigg) \\
		\label{eq:14}
		&=
		1 + \frac{1}{2} \vartheta_{j} 
		+ \frac{1}{\sqrt{\gamma_{jN}}} \delta_j.
	\end{align}
	Analogously, we can find that
	\[
		\tilde{f}_j = -2 \vartheta_{j} + \vartheta_{j}^2 +\frac{1}{\gamma_{jN}} \delta_j,
	\]
	and
	\begin{equation}
		\label{eq:18}
		\frac{1}{\ue^{-\vartheta_{j+1}} - \ue^{\vartheta_{j+1}}} \tilde{f}_{j} 
		= 1 -\frac{1}{2} \vartheta_{j} 
		 +
		\frac{1}{\sqrt{\gamma_{jN}}} \delta_j.
	\end{equation}
	Next, we write
	\begin{align*}
		g_j &= 1 - \vartheta_{j} - \vartheta_{j+1} 
		+ \frac{1}{2} (\vartheta_{j}+\vartheta_{j+1})^2 + 1 - 2 \bigg(1 -\vartheta_{j} + \frac{1}{2} \vartheta_{j}^2\bigg) 
		+ \frac{1}{\gamma_{jN}} \delta_j \\
		&= \vartheta_{j}^2 + \frac{1}{\gamma_{jN}} \delta_j,
	\end{align*}
	thus
	\begin{equation}
		\label{eq:19}
		\frac{1}{\ue^{-\vartheta_{j+1}} - \ue^{\vartheta_{j+1}}} g_j 
		= -\frac{1}{2} \vartheta_{j} 
		+ \frac{1}{\sqrt{\gamma_{jN}}} \delta_j.
	\end{equation}
	Similarly, we get
	\[
		\tilde{g}_j = -\vartheta_{j}^2 +
		+
		\frac{1}{\gamma_{jN}} \delta_j,
	\]
	and so
	\begin{equation}
		\label{eq:27}
		\frac{1}{\ue^{-\vartheta_{j+1}} - \ue^{\vartheta_{j+1}}} \tilde{g}_j = 
		 \frac{1}{2} \vartheta_{j} + 
		 +\frac{1}{\sqrt{\gamma_{jN}}} \delta_j.
	\end{equation}
	Consequently, by \eqref{eq:14}--\eqref{eq:27} we obtain
	\begin{equation}
		\label{eq:36}
		Z_{j+1}^{-1} \mathfrak{X}_{i}(0) Z_{j}
		=
		\varepsilon 
		\left\{
		\Id + \frac{1}{2} \vartheta_{j} 
		\begin{pmatrix}
			1 & -1 \\
			1 & -1
		\end{pmatrix} 
		+ \frac{1}{\sqrt{\gamma_{jN}}} \calE_j
		\right\}.
	\end{equation}
	Next, we observe that
	\[
		\ue^{\vartheta_{j}} = 1 + \delta_j, \qquad
		\frac{\vartheta_{j}}{\sinh \vartheta_{j}} = 1 + \delta_j,
	\]
	thus in view of \eqref{eq:20e}, for each $i' \in \{0, 1, \ldots, N-1\}$ we have
	\begin{align*}
		&
		\begin{pmatrix}
			1 & 1 \\
			\ue^{\vartheta_j} & \ue^{-\vartheta_j}
		\end{pmatrix}^{-1}
		T_{i'}^{-1}
		\begin{pmatrix}
			0 & \frac{x \frakt}{\gamma_{jN+i'}} \\
			0 & 0
		\end{pmatrix}
		T_{i'}
		\begin{pmatrix}
			1 & 1 \\
			\ue^{\vartheta_j} & \ue^{-\vartheta_j}
		\end{pmatrix}  \\
		&\qquad=
		-
		\frac{1}{2\vartheta_j}
		\frac{x \frakt}{\gamma_{jN+i'}}
		\begin{pmatrix}
			1 & -1\\
			-1 & 1
		\end{pmatrix}
		T_{i'}^{-1}
		\begin{pmatrix}
			0 & 1\\
			0 & 0
		\end{pmatrix}
		T_{i'}
		\begin{pmatrix}
			1 & 1 \\
			1 & 1
		\end{pmatrix} 
		+
		\frac{1}{\sqrt{\gamma_{jN}}} \calE_j \\
		&\qquad=
		-\frac{1}{2\vartheta_j}
		\frac{x \frakt}{\gamma_{jN+i'}}
		\frac{([T_{i'}]_{21} + [T_{i'}]_{22})^2}{\det T_{i'}}
		\begin{pmatrix}
			1 & 1\\
			-1 & -1
		\end{pmatrix}
		+
		\frac{1}{\sqrt{\gamma_{jN}}} \calE_j \\
		&\qquad=
		\frac{1}{2\vartheta_j}
		\frac{x \frakt}{\gamma_{jN+i'}}
		\big(\varepsilon [\frakX_{i'}(0)]_{21}\big)
		\begin{pmatrix}
			1 & 1\\
			-1 & -1
		\end{pmatrix}
		+
		\frac{1}{\sqrt{\gamma_{jN}}} \calE_j.
	\end{align*}
	We write	
	\begin{align*}
		&
		\begin{pmatrix}
			1 & 1 \\
			\ue^{\vartheta_j} & \ue^{-\vartheta_j}
		\end{pmatrix}^{-1}
		T_{i'}^{-1}
		\begin{pmatrix}
			\xi_{jN+i'} & \zeta_{jN+i'} \\
			0 & 0
		\end{pmatrix}
		T_{i'}
		\begin{pmatrix}
			1 & 1 \\
			\ue^{\vartheta_j} & \ue^{-\vartheta_j}
		\end{pmatrix}  \\
		&\qquad=
		\frac{1}{\ue^{-\vartheta_j} - \ue^{\vartheta_j}}
		\begin{pmatrix}
			1  & -1 \\
			-1 & 1
		\end{pmatrix}
		T_{i'}^{-1}
		\begin{pmatrix}
			\xi_{jN+i'} & \zeta_{jN+i'}\\
			0 & 0
		\end{pmatrix}
		T_{i'}
		\begin{pmatrix}
			1 & 1 \\
			1 & 1 
		\end{pmatrix}  \\
		&\qquad\phantom{=}+
		\frac{1}{\ue^{-\vartheta_j} - \ue^{\vartheta_j}}
		\begin{pmatrix}
			\ue^{-\vartheta_j}-1 & 0 \\
			1-\ue^{\vartheta_j} & 0
		\end{pmatrix}
		T_{i'}^{-1}
		\begin{pmatrix}
			\xi_{jN+i'} & \zeta_{jN+i'} \\
			0 & 0
		\end{pmatrix}
		T_{i'}
		\begin{pmatrix}
			1 & 1 \\
			1 & 1
		\end{pmatrix}  \\
		&\qquad\phantom{=}+
		\frac{1}{\ue^{-\vartheta_j} - \ue^{\vartheta_j}}
		\begin{pmatrix}
			1 & -1 \\
			-1 & 1
		\end{pmatrix}
		T_{i'}^{-1}
		\begin{pmatrix}
			\xi_{jN+i'} & \zeta_{jN+i'} \\
			0 & 0
		\end{pmatrix}
		T_{i'}
		\begin{pmatrix}
			0 & 0 \\
			\ue^{\vartheta_j}-1 & \ue^{-\vartheta_j} -1
		\end{pmatrix}
		+ \frac{1}{\sqrt{\gamma_{jN}}} \calE_j.
	\end{align*}
	We observe that
	\begin{align*}
		&\frac{1}{\ue^{-\vartheta_j} - \ue^{\vartheta_j}}
		\begin{pmatrix}
			1  & -1 \\
			-1 & 1
		\end{pmatrix}
		T_{i'}^{-1}
		\begin{pmatrix}
			\xi_{jN+i'} & 0\\
			0 & 0
		\end{pmatrix}
		T_{i'}
		\begin{pmatrix}
			1 & 1 \\
			1 & 1 
		\end{pmatrix} \\
		&\qquad=
		\frac{\xi_{jN+i'}}{\ue^{-\vartheta_j} - \ue^{\vartheta_j}}
		\frac{([T_{i'}]_{11} + [T_{i'}]_{12})([T_{i'}]_{21} + [T_{i'}]_{22})}{\det T_{i'}}
		\begin{pmatrix}
			1 & 1 \\
			-1 & -1
		\end{pmatrix} \\
		&\qquad=
		\frac{\xi_{jN+i'}}{\ue^{-\vartheta_j} - \ue^{\vartheta_j}}
		\big(1 - \varepsilon [\frakX_{i'}(0)]_{11}\big)
		\begin{pmatrix}	
			1 & 1 \\
			-1 & -1
		\end{pmatrix},
	\end{align*}
	and
	\begin{align*}
		&
		\frac{1}{\ue^{-\vartheta_j} - \ue^{\vartheta_j}}
		\begin{pmatrix}
			1  & -1 \\
			-1 & 1
		\end{pmatrix}
		T_{i'}^{-1}
		\begin{pmatrix}
			0 & \zeta_{jN+i'} \\
			0 & 0
		\end{pmatrix}
		T_{i'}
		\begin{pmatrix}
			1 & 1 \\
			1 & 1 
		\end{pmatrix}\\
		&\qquad=
		\frac{\zeta_{jN+i'}}{\ue^{-\vartheta_j} - \ue^{\vartheta_j}}
		\frac{([T_{i'}]_{21} + [T_{i'}]_{22})^2}{\det T_{i'}}
		\begin{pmatrix}
			1 & 1 \\
			-1 & -1
		\end{pmatrix} \\
		&\qquad=
		-
		\frac{\zeta_{jN+i'}}{\ue^{-\vartheta_j} - \ue^{\vartheta_j}}
		\big(\varepsilon [\frakX_{i'}(0)]_{21}\big)
		\begin{pmatrix}
			1 & 1 \\
			-1 & -1
		\end{pmatrix}.
	\end{align*}
	Hence, by \eqref{eq:88} and \eqref{eq:63},
	\[
		\frac{1}{\ue^{-\vartheta_j} - \ue^{\vartheta_j}}
		\begin{pmatrix}
			1  & -1 \\
			-1 & 1
		\end{pmatrix}
		T_{i'}^{-1}
		\begin{pmatrix}
			\xi_{jN+i'} & \zeta_{jN+i'} \\
			0 & 0
		\end{pmatrix}
		T_{i'}
		\begin{pmatrix}
			1 & 1 \\
			1 & 1 
		\end{pmatrix}
		=
		-\frac{1}{2\vartheta_j} \frac{\fraku_{i'}}{\gamma_{jN+i'}} 
		\begin{pmatrix}
			1 & 1 \\
			-1 & -1
		\end{pmatrix}
		+
		\frac{1}{\sqrt{\gamma_{jN}}} \calE_j.
	\]
	Next, let us notice that
	\[
		\bigg(\frac{e^{\vartheta_j} - 1}{e^{\vartheta_j} - e^{-\vartheta_j}} : j \in \NN\bigg) \in \calD_1, 
	\]
	and
	\[
		\lim_{j \to \infty} \frac{e^{\vartheta_j} - 1}{e^{\vartheta_j} - e^{-\vartheta_j}} = \frac{1}{2}.
	\]
	Therefore, by \eqref{eq:20e}, we get
	\begin{align*}
		&
		\frac{1}{\ue^{-\vartheta_j} - \ue^{\vartheta_j}}
		\begin{pmatrix}
			\ue^{-\vartheta_j}-1 & 0 \\
			1-\ue^{\vartheta_j} & 0
		\end{pmatrix}
		T_{i'}^{-1}
		\begin{pmatrix}
			\xi_{jN+i'} & 0 \\
			0 & 0
		\end{pmatrix}
		T_{i'}
		\begin{pmatrix}
			1 & 1 \\
			1 & 1
		\end{pmatrix}\\ 
		&\qquad=
		\frac{1}{2}
		\begin{pmatrix}
			1 & 0 \\
			1 & 0
		\end{pmatrix}
		T_{i'}^{-1}
		\begin{pmatrix}
			\frac{\fraks_{i'}}{\sqrt{\alpha_{i'} \gamma_{jN+i'}}} & 0 \\
			0 & 0
		\end{pmatrix}
		T_{i'}
		\begin{pmatrix}
			1 & 1 \\
			1 & 1
		\end{pmatrix} 
		+
		\frac{1}{\sqrt{\gamma_{jN}}} \calE \\
		&\qquad=
		\frac{\fraks_{i'}}{2 \sqrt{ \alpha_{i'} \gamma_{jN+i'}}}
		\frac{([T_{i'}]_{11} + [T_{i'}]_{12}) [T_{i'}]_{22}}{\det T_{i'}}
		\begin{pmatrix}
			1 & 1 \\
			1 & 1
		\end{pmatrix}
		+
		\frac{1}{\sqrt{\gamma_{jN}}} \calE_j,
	\end{align*}
	and
	\begin{align*}
		&
		\frac{1}{\ue^{-\vartheta_j} - \ue^{\vartheta_j}}
		\begin{pmatrix}
			1 & -1 \\
			-1 & 1
		\end{pmatrix}
		T_{i'}^{-1}
		\begin{pmatrix}
			\xi_{jN+i'} & 0 \\
			0 & 0
		\end{pmatrix}
		T_{i'}
		\begin{pmatrix}
			0 & 0 \\
			e^{\vartheta_j}-1 & e^{-\vartheta_j} - 1
		\end{pmatrix} \\
		&\qquad=
		\frac{1}{2}
		\begin{pmatrix}
			1 & -1\\
			-1 & 1
		\end{pmatrix}
		T_{i'}^{-1}
		\begin{pmatrix}
			\frac{\fraks_{i'}}{\sqrt{ \alpha_{i'} \gamma_{jN+i'}}} & 0 \\
			0 & 0
		\end{pmatrix}
		T_{i'}
		\begin{pmatrix}
			0 & 0 \\
			-1 & 1
		\end{pmatrix} 
		+
		\frac{1}{\sqrt{\gamma_{jN}}} \calE_j \\
		&\qquad=
		\frac{\fraks_{i'}}{2 \sqrt{ \alpha_{i'} \gamma_{jN+i'}}}
		\frac{([T_{i'}]_{21} + [T_{i'}]_{22}) [T_{i'}]_{12}}{\det T_{i'}}
		\begin{pmatrix}
			-1 & 1 \\
			1 & -1
		\end{pmatrix}
		+
		\frac{1}{\sqrt{\gamma_{jN}}} \calE_j.
	\end{align*}
	Analogously,
	\begin{align}
		\nonumber
		&
		\frac{1}{\ue^{-\vartheta_j} - \ue^{\vartheta_j}}
		\begin{pmatrix}
			\ue^{-\vartheta_j}-1 & 0 \\
			1-\ue^{\vartheta_j} & 0
		\end{pmatrix}
		T_{i'}^{-1}
		\begin{pmatrix}
			0 & \zeta_{jN+i'} \\
			0 & 0
		\end{pmatrix}
		T_{i'}
		\begin{pmatrix}
			1 & 1 \\
			1 & 1
		\end{pmatrix}\\
		\nonumber
		&\qquad=
		\frac{1}{2}
		\begin{pmatrix}
			1 & 0 \\
			1 & 0
		\end{pmatrix}
		T_{i'}^{-1}
		\begin{pmatrix}
			0 & \frac{\frakr_{i'}}{\sqrt{ \alpha_{i'} \gamma_{jN+i'}}} \\
			0 & 0
		\end{pmatrix}
		T_{i'}
		\begin{pmatrix}
			1 & 1 \\
			1 & 1
		\end{pmatrix} 
		+
		\frac{1}{\sqrt{\gamma_{jN}}} \calE_j \\
		\label{eq:24}
		&\qquad=
		\frac{\frakr_{i'}}{2 \sqrt{\alpha_{i'} \gamma_{jN+i'}}}
		\frac{([T_{i'}]_{21} + [T_{i'}]_{22}) [T_{i'}]_{22}}{\det T_{i'}}
		\begin{pmatrix}
			1 & 1 \\
			1 & 1
		\end{pmatrix}
		+
		\frac{1}{\sqrt{\gamma_{jN}}} \calE_j,
	\end{align}
	and
	\begin{align}
		\nonumber
		&
		\frac{1}{\ue^{-\vartheta_j} - \ue^{\vartheta_j}}
		\begin{pmatrix}
			1 & -1 \\
			-1 & 1
		\end{pmatrix}
		T_{i'}^{-1}
		\begin{pmatrix}
			0 & \zeta_{jN+i'} \\
			0 & 0
		\end{pmatrix}
		T_{i'}
		\begin{pmatrix}
			0 & 0 \\
			e^{\vartheta_j}-1 & e^{-\vartheta_j} - 1
		\end{pmatrix} \\
		\nonumber
		&\qquad=
		\frac{1}{2}
		\begin{pmatrix}
			1 & -1\\
			-1 & 1
		\end{pmatrix}
		T_{i'}^{-1}
		\begin{pmatrix}
			0 & \frac{\frakr_{i'}}{\sqrt{ \alpha_{i'} \gamma_{jN+i'}}} \\
			0 & 0
		\end{pmatrix}
		T_{i'}
		\begin{pmatrix}
			0 & 0 \\
			-1 & 1
		\end{pmatrix} 
		+
		\frac{1}{\sqrt{\gamma_{jN}}} \calE_j \\
		\label{eq:25}
		&\qquad=
		\frac{\frakr_{i'}}{2 \sqrt{ \alpha_{i'} \gamma_{jN+i'}}}
		\frac{([T_{i'}]_{21} + [T_{i'}]_{22}) [T_{i'}]_{22}}{\det T_{i'}}
		\begin{pmatrix}
			-1 & 1 \\
			1 & -1
		\end{pmatrix}
		+
		\frac{1}{\sqrt{\gamma_{jN}}} \calE_j.
	\end{align}
	In view of \eqref{eq:33}, \eqref{eq:34} and \eqref{eq:63}, we have
	\[
		\fraks_{i'} ([T_{i'}]_{11} + [T_{i'}]_{12}) = -\frakr_{i'} ([T_{i'}]_{21} + [T_{i'}]_{22}),
	\]
	thus by \eqref{eq:24} and \eqref{eq:25} we get
	\begin{align*}
		&
		\frac{1}{\ue^{-\vartheta_j} - \ue^{\vartheta_j}}
		\begin{pmatrix}
			\ue^{-\vartheta_j}-1 & 0 \\
			1-\ue^{\vartheta_j} & 0
		\end{pmatrix}
		T_{i'}^{-1}
		\begin{pmatrix}
			0 & \zeta_{jN+i'} \\
			0 & 0
		\end{pmatrix}
		T_{i'}
		\begin{pmatrix}
			1 & 1 \\
			1 & 1
		\end{pmatrix}\\ 
		&\qquad=
		\frac{\fraks_{i'}}{2 \sqrt{\alpha_{i'} \gamma_{jN+i'}}}
		\frac{([T_{i'}]_{11} + [T_{i'}]_{12}) [T_{i'}]_{22}}{\det T_{i'}}
		\begin{pmatrix}
			-1 & -1 \\
			-1 & -1
		\end{pmatrix}
		+
		\frac{1}{\sqrt{\gamma_{jN}}} \calE_j,
	\end{align*}
	and
	\begin{align*}
		&
		\frac{1}{\ue^{-\vartheta_j} - \ue^{\vartheta_j}}
		\begin{pmatrix}
			1 & -1 \\
			-1 & 1
		\end{pmatrix}
		T_{i'}^{-1}
		\begin{pmatrix}
			0 & \zeta_{jN+i'} \\
			0 & 0
		\end{pmatrix}
		T_{i'}
		\begin{pmatrix}
			0 & 0 \\
			e^{\vartheta_j}-1 & e^{-\vartheta_j} - 1
		\end{pmatrix} \\
		&\qquad=
		\frac{\fraks_{i'}}{2 \sqrt{ \alpha_{i'} \gamma_{jN+i'}}}
		\frac{([T_{i'}]_{11} + [T_{i'}]_{12}) [T_{i'}]_{22}}{\det T_{i'}}
		\begin{pmatrix}
			1 & -1 \\
			-1 & 1
		\end{pmatrix}
		+
		\frac{1}{\sqrt{\gamma_{jN}}} \calE_j.
	\end{align*}
	Hence,
	\begin{align*}
		&\frac{1}{\ue^{-\vartheta_j} - \ue^{\vartheta_j}}
		\begin{pmatrix}
			\ue^{-\vartheta_j}-1 & 0 \\
			1-\ue^{\vartheta_j} & 0
		\end{pmatrix}
		T_{i'}^{-1}
		\begin{pmatrix}
			\xi_{jN+i'} & \zeta_{jN+i'} \\
			0 & 0
		\end{pmatrix}
		T_{i'}
		\begin{pmatrix}
			1 & 1 \\
			1 & 1
		\end{pmatrix}  \\
		&\qquad+
		\frac{1}{\ue^{-\vartheta_j} - \ue^{\vartheta_j}}
		\begin{pmatrix}
			1 & -1 \\
			-1 & 1
		\end{pmatrix}
		T_{i'}^{-1}
		\begin{pmatrix}
			\xi_{jN+i'} & \zeta_{jN+i'} \\
			0 & 0
		\end{pmatrix}
		T_{i'}
		\begin{pmatrix}
			0 & 0 \\
			\ue^{\vartheta_j}-1 & \ue^{-\vartheta_j} -1
		\end{pmatrix}\\
		&=
		\frac{\fraks_{i'}}{2 \sqrt{\alpha_{i'} \gamma_{jN+i'}}} 
		\begin{pmatrix}
			1 & -1 \\
			-1 & 1
		\end{pmatrix}
		+
		\frac{1}{\sqrt{\gamma_{jN}}} \calE_j. 
	\end{align*}
	Summarizing, we obtain
	\begin{align*}
		&
		\begin{pmatrix}
			1 & 1 \\
			\ue^{\vartheta_j} & \ue^{-\vartheta_j}
		\end{pmatrix}^{-1}
		T_{i'}^{-1}
		\begin{pmatrix}
			\xi_{jN+i'} & \zeta_{jN+i'}+\frac{x \frakt}{\gamma_{jN+i'}} \\
			0 & 0
		\end{pmatrix}
		T_{i'}
		\begin{pmatrix}
			1 & 1 \\
			\ue^{\vartheta_j} & \ue^{-\vartheta_j}
		\end{pmatrix}  \\
		&\qquad=
		\frac{1}{2\vartheta_j}
		\frac{\frakt x \big(\varepsilon [\frakX_{i'}(0)]_{21} \big) - \fraku_{i'}}{\gamma_{jN+i'}}
		\begin{pmatrix}
			1 & 1\\
			-1 & -1
		\end{pmatrix}
		+
		\frac{\fraks_{i'}}{2 \sqrt{\alpha_{i'} \gamma_{jN+i'}}}
		\begin{pmatrix}
			1 & -1\\
			-1 & 1
		\end{pmatrix}
		+
		\frac{1}{\sqrt{\gamma_{jN}}} \calE_j.
	\end{align*}
	By \eqref{eq:17} and \eqref{eq:43}, we have
	\begin{align*}
		\sqrt{\frac{\gamma_{(j+1)N+i-1}}{\alpha_{i-1}}}
		\sum_{i' = i}^{N+i-1}
		\frac{\alpha_{i'}}{\alpha_{i'-1}} 
		\frac{x \frakt \big(\varepsilon [\frakX_{i'}(0)]_{21} \big)-\fraku_{i'}}{\gamma_{jN+i'}} 
		&=
		\sqrt{\frac{\alpha_{i-1}}{\gamma_{(j+1)N+i-1}}} \upsilon
		+
		\frac{1}{\sqrt{\gamma_{jN}}} \delta_j,
	\end{align*}
	and
	\[
		\sum_{i' = i}^{N+i-1}
		\frac{\alpha_i'}{\alpha_{i'-1}} \frac{\fraks_{i'}}{\sqrt{\alpha_{i'} \gamma_{jN+i'}}}
		=
		\sqrt{\frac{\alpha_{i-1}}{\gamma_{(j+1)N+i-1}}}
		\frakS + \frac{1}{\sqrt{\gamma_{jN}}} \delta_j.
	\]
	Therefore,
	\begin{align*}
		&
		\sum_{i' = i}^{N+i-1}
		\frac{\alpha_{i'}}{\alpha_{i'-1}}
		\begin{pmatrix}
			1 & 1 \\
			\ue^{\vartheta_j} & \ue^{-\vartheta_j}
		\end{pmatrix}^{-1}
		T_{i'}^{-1}
		\begin{pmatrix}
			\xi_{jN+i'} & \zeta_{jN+i'}-\frac{x \frakt}{\gamma_{jN+i'}} \\
			0 & 0
		\end{pmatrix}
		T_{i'}
		\begin{pmatrix}
			1 & 1 \\
			\ue^{\vartheta_j} & \ue^{-\vartheta_j}
		\end{pmatrix} \\
		&\qquad\qquad=
		\frac{\vartheta_j}{2} \frac{\upsilon}{\abs{\tau}}
		\begin{pmatrix}
			1 & 1 \\
			-1 & -1
		\end{pmatrix}
		+
		\frac{\sqrt{\alpha_{i-1}} \frakS}{2\sqrt{\gamma_{(j+1)N+i-1}}}
		\begin{pmatrix}
			1 & -1\\
			-1 & 1
		\end{pmatrix}
		+
		\frac{1}{\sqrt{\gamma_{jN}}} \calE_j.
	\end{align*}
	Finally, by \eqref{eq:36} we get
	\begin{align*}
		Z_{j+1}^{-1} X_{jN+i} Z_j
		=
		\varepsilon
		\left\{
		\Id
		+
		\sqrt{\frac{\alpha_{i-1}}{\gamma_{(j+1)N+i-1}}}
		\calR_i
		+
		\frac{1}{\sqrt{\gamma_{jN}}} \calE_j
		\right\} 
	\end{align*}
	where
	\[
		\calR_i(x) = \frac{\sqrt{\abs{\tau(x)}}}{2}
		\begin{pmatrix}
			1 & -1\\
			1 & -1
		\end{pmatrix}
		-
		\frac{\upsilon(x)}{2 \sqrt{\abs{\tau(x)}}}
		\begin{pmatrix}
			1 & 1\\
			-1 & -1
		\end{pmatrix}
		-
		\frac{\frakS}{2}
		\begin{pmatrix}
			1 & -1\\
			-1 & 1
		\end{pmatrix}
	\]
	which finishes the proof.
\end{proof}

\begin{corollary}
	\label{cor:1}
	Suppose that the hypotheses of Theorem \ref{thm:1} are satisfied. Then
	\begin{align*}
		\lim_{j \to \infty} \gamma_{(j+1)N+i-1} \discr\big(X_{jN+i} \big) 
		= 4 \tau \alpha_{i-1}
	\end{align*}
	locally uniformly on $\Lambda$ where $\tau$ is defined in \eqref{eq:32}.
\end{corollary}
\begin{proof}
	We write
	\[
		Z_j^{-1} X_{jN+i} Z_j = \big( Z_j^{-1} Z_{j+1} \big) \big( Z_{j+1}^{-1} X_{jN+i} Z_j \big),
	\]
	thus by Theorems \ref{thm:2} and \ref{thm:1}, we obtain
	\begin{align*}
		\varepsilon Z_j^{-1} X_{jN+i} Z_j 
		&= 
		\bigg( \Id + \sqrt{\frac{\alpha_{i-1}}{\gamma_{(j+1)N+i-1}}} Q_j \bigg) 
		\bigg( \Id + \sqrt{\frac{\alpha_{i-1}}{\gamma_{(j+1)N+i-1}}} R_j \bigg)  \\
		&= 
		\Id + \sqrt{\frac{\alpha_{i-1}}{\gamma_{(j+1)N+i-1}}} (Q_j + R_j) + \frac{\alpha_{i-1}}{\gamma_{(j+1)N+i-1}} 
		Q_j R_j.
	\end{align*}
	Hence,
	\[
		\frac{\gamma_{(j+1)N+i-1}}{\alpha_{i-1}} 
		\discr (X_{jN+i}) = \discr \bigg( R_j + Q_j + \sqrt{\frac{\alpha_{i-1}}{\gamma_{(j+1)N+i-1}}} Q_j R_j \bigg),
	\]
	and consequently,
	\[
		\lim_{j \to \infty} 
		\frac{\gamma_{(j+1)N+i-1}}{\alpha_{i-1}}
		\discr (X_{jN+i}) 
		= \discr (\calR_i).
	\]
	Since $\discr \calR_i = 4 \tau$, the conclusion follows.
\end{proof}

\section{Essential spectrum} \label{sec:4}
In this section we want to understand spectral properties of Jacobi operators corresponding to tempered
$N$-periodically modulated sequences. We set
\[
	\Lambda_- = \tau^{-1}\big((-\infty, 0)\big),
	\qquad\text{and}\qquad
	\Lambda_+ = \tau^{-1}\big((0, +\infty)\big).
\]
Observe that, if $\frakt = 0$ then at least one of the sets $\Lambda_-$ or $\Lambda_+$ is empty. Similar to the proof of 
\cite[Theorem 4.1]{jordan} we use the shifted conjugation (Theorems \ref{thm:2} and \ref{thm:1})
together with a variant of a Levison's theorem for discrete systems developed in \cite[Theorem 4.4]{Discrete}.

\begin{theorem}
	\label{thm:3}
	Let $N$ be a positive integer. Let $(\gamma_n : n \in \NN)$ be a sequence of positive
    numbers tending to infinity and satisfying \eqref{eq:87} and \eqref{eq:89}. Let $(a_n : n \in \NN_0)$ and
	$(b_n : n \in \NN_0)$ be $\gamma$-tempered $N$-periodically modulated Jacobi parameters such that $\frakX_0(0)$ is
	a non-trivial parabolic element. Suppose that \eqref{eq:88} holds true with $\varepsilon = \sign{\tr \frakX_0(0)}$.
	If the operator $A$ is self-adjoint then
	\[
		\sigmaEss(A) \cap \Lambda_+ = \emptyset.
	\]
\end{theorem}
\begin{proof}
	Let us fix a compact interval $K \subset \Lambda_+$ and $i \in \{0, 1, \ldots, N-1\}$. We set
	\[
		Y_j = Z_{j+1}^{-1} X_{jN+i} Z_j
	\]
	where $Z_j$ is the matrix defined in \eqref{eq:3}. By Theorem \ref{thm:1}, we have
	\begin{equation}
		\label{eq:37}
		Y_j = \varepsilon\bigg(\Id + \sqrt{\frac{\alpha_{i-1}}{\gamma_{(j+1)N+i-1}}} R_j\bigg)
	\end{equation}
	where $(R_j : j \in \NN)$ is a sequence from $\calD_1\big(K, \Mat(2, \RR) \big)$ convergent uniformly on $K$
	to the matrix $\calR_i$ given by the formula \eqref{eq:26}. By Corollary \ref{cor:1}, there are $j_0 \geq 0$ and 
	$\delta > 0$, so that for all $j \geq j_0$ and $x \in K$,
	\begin{equation}
		\label{eq:38}
		\discr R_j(x) \geq \delta.
	\end{equation}
	In particular, the matrix $R_j(x)$ has two eigenvalues
	\[
		\xi_j^+(x) = \frac{\tr R_j(x) + \sqrt{\discr R_j(x)}}{2},
		\quad\text{and}\quad
		\xi_j^-(x) = \frac{\tr R_j(x) - \sqrt{\discr R_j(x)}}{2}.
	\]
	In view of \eqref{eq:37}, the matrix $Y_j(x)$ has eigenvalues
	\[
		\lambda_j^+ = \varepsilon\bigg(1 + \sqrt{\frac{\alpha_{i-1}}{\gamma_{(j+1)N+i-1}}} \xi_j^+\bigg)
		\quad\text{and}\quad
		\lambda_j^- = \varepsilon\bigg(1 + \sqrt{\frac{\alpha_{i-1}}{\gamma_{(j+1)N+i-1}}} \xi_j^-\bigg).
	\]
	By possible increasing $j_0$ we can guarantee that
	\begin{equation}
		\label{eq:71}
		|\lambda_j^-(x)| \leq |\lambda_j^+(x)|
	\end{equation}
	for all $j \geq j_0$ and $x \in K$.

	Now, by the Stolz--Ces\`{a}ro theorem (see e.g. \cite[Section 3.1.7]{Muresan2009}), \eqref{eq:89} implies that
	\begin{equation}
		\label{eq:62}
		\lim_{j \to \infty} \frac{\sqrt{\gamma_{jN+i}}}{j}
		=
		\lim_{j \to \infty} \big(\sqrt{\gamma_{(j+1)N+i}} - \sqrt{\gamma_{jN+i}}\big) = 0,
	\end{equation}
	and so
	\[
		\sum_{j = 0}^\infty \frac{1}{\sqrt{\gamma_{jN+i}}} = \infty.
	\]
	Therefore, by \eqref{eq:38}, we can apply \cite[Theorem 4.4]{Discrete} to the system
	\begin{equation}
		\label{eq:60}
		\Psi_{j+1} = Y_j \Psi_j.
	\end{equation}
	Consequently, there are $(\Psi_j^- : j \geq j_0)$ and $(\Psi_j^+ : j \geq j_0)$ such that
	\begin{equation}
		\label{eq:16}
		\lim_{j \to \infty} \sup_{x \in K}{
		\bigg\|\frac{\Psi_j^+(x)}{\prod_{k = j_0}^{j-1} \lambda^+_k(x)} - v^+(x) \bigg\|} 
		=
		\lim_{j \to \infty} \sup_{x \in K}{
		\bigg\|\frac{\Psi_j^-(x)}{\prod_{k = j_0}^{j-1} \lambda^-_k(x)} - v^-(x) \bigg\|} = 0
	\end{equation}
	where $v^-(x)$ and $v^+(x)$ are continuous eigenvectors of $\calR_i(x)$ corresponding to
	\[
		\xi^+(x) = \frac{\tr \calR_i(x) + \sqrt{\discr \calR_i(x)}}{2},
		\quad\text{and}\quad
		\xi^-(x) = \frac{\tr \calR_i(x) - \sqrt{\discr \calR_i(x)}}{2}.
	\]
	Since $\tau(x) > 0$, by means of \eqref{eq:26} one can verify that
	\begin{equation}
		\label{eq:61}
		v_1^+(x)+v_2^+(x) \neq 0 \quad\text{and}\quad
		v_1^-(x)+v_2^-(x) \neq 0.
	\end{equation}
	Indeed, otherwise $e_1 - e_2$ would be an eigenvector of $\calR_i(x)$, but
	\begin{align*}
		\calR_i(x) (e_1 - e_2) 
		&= \big(\sqrt{\abs{\tau(x)}} - \frakS\big)e_1 + \big(\sqrt{\abs{\tau(x)}} + \frakS\big) e_2 \\
		&= \big(\sqrt{\abs{\tau(x)}} - \frakS\big)(e_1-e_2) + 2 \sqrt{\abs{\tau(x)}} e_2,
	\end{align*}
	thus $\tau(x) = 0$, which is impossible. 

	Now, by \eqref{eq:60} the sequences $\Phi_j^\pm = Z_j \Psi_j^\pm$ satisfy
	\[
		\Phi_{j+1} = X_{jN+i} \Phi_j, \qquad j \geq j_0.
	\]
	We set
	\[
		\phi_1^\pm = B_1^{-1} \cdots B_{j_0}^{-1} \Phi_{j_0}^\pm
	\]
	and
	\[
		\phi_{n+1}^\pm = B_n \phi_n^\pm, \qquad n > 1.
	\]
	Then for $jN+i' > j_0N+i$ with $i' \in \{0, 1, \ldots, N-1\}$, we get
	\[
		\phi_{jN+i'}^\pm =
		\begin{cases}
			B_{jN+i'}^{-1} B_{jN+i'+1}^{-1} \cdots B_{jN+i-1}^{-1} \Phi_j^\pm
			&\text{if } i' \in \{0, 1, \ldots, i-1\}, \\
			\Phi_j^\pm & \text{if } i ' = i,\\
			B_{jN+i'-1} B_{jN+i'-2} \cdots B_{jN+i} \Phi_j^\pm &
			\text{if } i' \in \{i+1, \ldots, N-1\}.
		\end{cases}
	\]
	Since for $i' \in \{0, 1, \ldots, i-1\}$,
	\[
		\lim_{j \to \infty} B_{jN+i'}^{-1} B_{jN+i'+1}^{-1} \cdots B_{jN+i-1}^{-1}
		=
		\frakB_{i'}^{-1}(0) \frakB_{i'+1}^{-1}(0) \cdots \frakB_{i-1}^{-1}(0),
	\]
	and
	\[
		\lim_{j \to \infty} Z_j v^\pm = (v_1^\pm+v_2^\pm) T_i (e_1 + e_2),
	\]
	we obtain
	\begin{equation}
		\label{eq:40}
		\lim_{j \to \infty}
		\sup_K{
		\bigg\|
		\frac{\phi^\pm_{jN+i'}}{\prod_{k = j_0}^{j-1} \lambda^\pm_k} - 
		(v_1^\pm+v_2^\pm) T_{i'} (e_1 + e_2)
		\bigg\|
		} =
		0.
	\end{equation}
	Analogously, we can show that \eqref{eq:40} holds true also for $i' \in \{i+1, \ldots, N-1\}$.

	Since $(\phi_n^\pm : j \in \NN)$ satisfies \eqref{eq:87}, the sequence $(u_n^\pm(x) : n \in \NN_0)$ defined as
	\[
		u_n^\pm(x) = 
		\begin{cases}
			\langle \phi_1^\pm(x), e_1 \rangle & \text{if } n = 0, \\
			\langle \phi_n^\pm(x), e_2 \rangle & \text{if } n \geq 1,
		\end{cases}
	\]
	is a generalized eigenvector associated to $x \in K$, provided that $(u_0^\pm, u_1^\pm) \neq 0$ on $K$. Suppose on the
	contrary that there is $x \in K$ such that $\phi_1^\pm(x) = 0$. Hence, $\phi_n^\pm(x) = 0$ for all $n \in \NN$, thus
	by \eqref{eq:61} and \eqref{eq:40} we must have $T_0(e_1 + e_2) = 0$ which is impossible since $T_0$ is invertible.
	
	Consequently, $(u_n^+(x) : n \in \NN_0)$ and $(u_n^-(x) : n \in \NN_0)$ are two generalized eigenvectors associated with
	$x \in K$ with different asymptotic behavior, thus they are linearly independent.

	Now, let us suppose that $A$ is self-adjoint. By the proof of \cite[Theorem 5.3]{Silva2007}, if
	\begin{equation}
		\label{eq:58}
		\sum_{n = 0}^\infty \sup_{x \in K}{\big|u^-_n(x)\big|^2} < \infty
	\end{equation}
	then $K \cap \sigmaEss(A) = \emptyset$, and since $K$ is any compact subset of $\Lambda_+$ this implies that
	$\sigmaEss(A) \cap \Lambda_+ = \emptyset$. Hence, it is enough to show \eqref{eq:58}. Let us observe that 
	by \eqref{eq:40}, for each $i' \in \{0, 1, \ldots, N-1\}$, $j > j_0$, and $x \in K$,
	\begin{equation}
		\label{eq:41}
		|u_{jN+i'}^-(x)|
		\leq
		c 
		\prod_{k = j_0}^{j-1} |\lambda_k^-(x)|.
	\end{equation}
	Since $(R_j : j \in \NN)$ converges to $\calR_i$ uniformly on $K$, and
	\[
		\lim_{n \to \infty} \gamma_n = \infty,
	\]
	there is $j_1 \geq j_0$, such that for $j > j_1$,
	\[
		\sqrt{\frac{\alpha_{i-1}}{\gamma_{(j+1)N+i-1}}} 
		\Big( |\tr R_j(x)| + \sqrt{\discr R_j(x)}\Big)\leq 1.
	\]
	Therefore, for $j \geq j_1$,
	\[
		|\lambda_j^-(x)|
		=
		1 +
		\sqrt{\frac{\alpha_{i-1}}{\gamma_{(j+1)N+i-1}}} 
		\frac{ \tr R_j(x) - \sqrt{\discr R_j(x)} }{2}.
	\]
	By \eqref{eq:26} and Proposition \ref{prop:1}, $\tr \calR_i = -\frakS \leq 0$, thus \eqref{eq:62} implies that
	\[
		\lim_{j \to \infty}
		j \sqrt{\frac{\alpha_{i-1}}{\gamma_{(j+1)N+i-1}}} \frac{\tr R_j(x) - \sqrt{\discr R_j(x)}} {2} 
		=
		-\infty.
	\]
	In particular, there is $j_2 \geq j_1$ such that for all $j > j_2$,
	\[
		\sup_{x \in K} {|\lambda_j^-(x)|} \leq 1 - \frac{1}{j}.
	\]
	Consequently, by \eqref{eq:41}, there is $c' > 0$ such that for all $i' \in \{0, 1, \ldots, N-1\}$ and
	$j > j_2$,
	\[
		\sup_{x \in K}{|u_{jN+i'}^-(x)|} 
		\leq c \prod_{k = j_2}^{j-1} \bigg(1 - \frac{1}{k} \bigg) \leq \frac{c'}{j},
	\]
	which leads to \eqref{eq:58} and the theorem follows.
\end{proof}
\begin{remark}
In Section~\ref{sec:9} we characterize when $A$ is self-adjoint. In particular, Theorem~\ref{thm:6} settles the problem when $\Lambda_- \neq \emptyset$. If $\Lambda_- = \emptyset$ but $\Lambda_+ \neq \emptyset$, the formula \eqref{eq:72} is a necessary and sufficient condition for self-adjointness of $A$.
\end{remark}

\section{Uniform diagonalization}
\label{sec:5}
Fix a positive integer $N$ and $i \in \{0, 1, \ldots, N-1\}$. Let $(\gamma_n : n \in \NN)$ be a sequence of positive
numbers tending to infinity and satisfying \eqref{eq:87} and \eqref{eq:89}. Let $(a_n : n \in \NN_0)$ and
$(b_n : n \in \NN_0)$ be $\gamma$-tempered $N$-periodically modulated Jacobi parameters such that $\frakX_0(0)$ is
non-diagonalizable and let $\varepsilon = \sign{\tr \frakX_0(0)}$. Suppose that \eqref{eq:88} holds true.
Assume that $\Lambda_- \neq \emptyset$. Let us consider a compact interval in $\Lambda_-$ and a generalized eigenvector 
$(u_n : n \in \NN_0)$ associated to $x \in K$ and corresponding to $\eta \in \sS^1$. We set
\begin{equation}
	\label{eq:46}
	Y_j = Z_{j+1}^{-1} X_{jN+i} Z_j
\end{equation}
and
\begin{equation} \label{eq:46a}
	\vec{v}_j(\eta, x) = Z_j^{-1}(x) \vec{u}_{jN+i}(\eta, x)
\end{equation}
where $Z_j$ is defined in \eqref{eq:3}. In view of Theorem \ref{thm:1}, we have
\[
	Y_j = \varepsilon\bigg(\Id + \sqrt{\frac{\alpha_{i-1}}{\gamma_{(j+1)N+i-1}}} R_j \bigg)
\]
where $(R_j)$ is a sequence from $\calD_1\big(K, \Mat(2, \RR)\big)$ convergent to $\calR_i$ given by \eqref{eq:26}.
Since $\discr \calR_i < 0$ on $K$
\[
	|[\calR_i(x)]_{12}| > 0
\]
and there are $\delta > 0$ and $j_0 \geq 1$ such that for all $j \geq j_0$ and $x \in K$,
\[
	\discr R_j(x) < -\delta
	\qquad\text{and}\qquad
	|[R_j(x)]_{12}| > \delta.
\]
Therefore, $R_j(x)$ has two eigenvalues $\xi_j(x)$ and $\overline{\xi_j(x)}$ where
\begin{equation}
	\label{eq:45}
	\xi_j(x) = \frac{\tr R_j(x) + i \varepsilon \sqrt{-\discr R_j(x)}}{2}.
\end{equation}
Moreover,
\[
	R_j = C_j 
	\begin{pmatrix}
		\xi_j & 0 \\
		0 & \overline{\xi_j}
	\end{pmatrix}
	C_j^{-1}
\]
where
\[
	C_j = 
	\begin{pmatrix}
		1 & 1 \\
		\frac{\xi_j - [R_j]_{11}}{[R_j]_{12}} & \frac{\overline{\xi_j} - [R_j]_{11}}{[R_j]_{12}}
	\end{pmatrix}.
\]
Using \eqref{eq:46}, $Y_j(x)$ has two eigenvalues $\lambda_j(x)$ and $\overline{\lambda_j(x)}$ where
\[
	\lambda_j(x) = \varepsilon\bigg(1  + \sqrt{\frac{\alpha_{i-1}}{\gamma_{(j+1)N+i-1}}} \xi_j(x) \bigg).
\]
Moreover,
\begin{equation}
	\label{eq:48}
	Y_j = C_j D_j C_j^{-1}
\end{equation}
where
\begin{equation}
	\label{eq:47}
	D_j = 
	\begin{pmatrix}
		\lambda_j & 0 \\
		0 & \overline{\lambda_j}
	\end{pmatrix}.
\end{equation}
Theorem \ref{thm:1} implies that $(C_j : j \geq j_0)$ and $(D_j : j \geq j_0)$ belong to
$\calD_1\big(K, \Mat(2, \CC)\big)$. By \eqref{eq:26}, there is a mapping $C_\infty: K \rightarrow \GL(2, \CC)$ 
such that
\[
	\lim_{j \to \infty} C_j = C_\infty
\]
uniformly on $K$.
\begin{claim}
	\label{clm:1}
	There is $c > 0$ such that for all $j \geq L > j_0$,
	\[
		\|\vec{v}_j \| \leq c \bigg(\prod_{k = L}^{j-1} \big\| D_k \big\|\bigg) \|\vec{v}_{L} \|
	\]
	uniformly on $\sS^1 \times K$.
\end{claim}
For the proof, we write
\[
	\vec{v}_j = Y_{j-1} \cdots Y_{j_0} \vec{v}_{L}.
\]
thus
\[
	\|\vec{v}_j\| \leq \big\| Y_{j-1} \cdots Y_{L} \big\| \|\vec{v}_{L}\|.
\]
Next,
\[
	Y_{j-1} \cdots Y_{j_0} = C_{j-1} \big(D_{j-1} C_{j-1}^{-1} C_{j-2}\big)
	\big(D_{j-2} C_{j-2}^{-1} C_{j-3}\big) \cdots
	\big(D_{L} C_{L}^{-1} C_{L-1}\big) C_{L-1}^{-1}
\]
and so
\begin{align*}
	\big\| Y_{j-1} \cdots Y_{L} \big\|
	&\leq 
	\big\|C_{j-1} \big\| 
	\Big\|
	\big(D_{j-1} C_{j-1}^{-1} C_{j-2}\big)
	\big(D_{j-2} C_{j-2}^{-1} C_{j-3}\big)
	\cdots
	\big(D_{L} C_{L}^{-1} C_{L-1}\big)
	\Big\|
	\big\|C_{L-1}^{-1}\big\| \\
	&\leq
	c \prod_{k = L}^{j-1} \|D_k\|
\end{align*}
where the last estimate follows by \cite[Proposition 1]{SwiderskiTrojan2019}, proving Claim \ref{clm:1}.

Next, we show the following statement.
\begin{claim}
	\label{clm:2}
	We have
	\[
		\lim_{j \to \infty} 
		\frac{a_{jN+i-1}}{\sqrt{\gamma_{jN+i-1}}} 
		\prod_{k=j_0}^{j-1} |\lambda_k|^2 =
		\frac{a_{j_0N+i-1} \sinh \vartheta_{j_0}}{\sqrt{\alpha_{i-1} |\tau|}} > 0
	\]
	uniformly on $K \subset \Lambda_-$.
\end{claim}
By \eqref{eq:48} and \eqref{eq:47},
\[
	|\lambda_k(x)|^2 = \det D_k(x) = \det Y_j(x),
\]
which together with \eqref{eq:46} gives
\[
	|\lambda_k(x)|^2 = \frac{\sinh \vartheta_k(x)}{\sinh \vartheta_{k+1}(x)} \cdot 
	\frac{a_{jN+i-1}}{a_{(j+1)N+i-1}}.
\]
Hence,
\[
	\prod_{k=j_0}^{j-1} |\lambda_k(x)|^2
	=
	\frac{\sinh \vartheta_{j_0}(x)}{\sinh \vartheta_j(x)} \cdot
	\frac{a_{j_0N+i-1}}{a_{jN+i-1}},
\]
and since
\[
	\lim_{j \to \infty} \frac{\sinh \vartheta_j(x)}{\vartheta_j(x)} = 1
\]
the claim follows.

\section{Generalized shifted Tur\'an determinants} \label{sec:6}
In this section we study generalized $N$-shifted Tur\'an determinants. Namely, for $\eta \in \RR^2 \setminus \{0\}$ and $x \in \RR$ we consider
\[
	S_n(\eta, x) = a_{n+N-1} \sqrt{\gamma_{n+N-1}} \big\langle E \vec{u}_{n+N}, \vec{u}_n \big\rangle
\]
where $(\vec{u}_n : n \in \NN_0)$ corresponds to a generalized eigenvector associated to $x$ and corresponding to $\eta$,
and
\[
	E = 
	\begin{pmatrix}
		0 & -1 \\
		1 & 0
	\end{pmatrix}.
\]
\begin{theorem} \label{thm:5}
	Let $N$ be a positive integer and $i \in \{0, 1, \ldots N-1\}$. Let $(\gamma_n : n \in \NN)$ be a sequence of positive
    numbers tending to infinity and satisfying \eqref{eq:87} and \eqref{eq:89}. Let $(a_n : n \in \NN_0)$ and
	$(b_n : n \in \NN_0)$ be $\gamma$-tempered $N$-periodically modulated Jacobi parameters such that $\frakX_0(0)$ is
	a non-trivial parabolic element. Suppose that \eqref{eq:88} holds true with $\varepsilon = \sign{\tr \frakX_0(0)}$.
	Then the sequence $(|S_{jN+i}| : j \in \NN)$ converges locally uniformly on $\sS^1 \times \Lambda_-$ to
	a positive continuous function.
\end{theorem}
\begin{proof}
	We use the uniform diagonalization described in Section \ref{sec:5}. Let us define
	\begin{equation}
		\label{eq:66}
		\tilde{S}_j = a_{(j+1)N+i-1} \sqrt{\gamma_{(j+1)N+i-1}} (\det Z_j) \langle E \vec{v}_{j+1}, \vec{v}_j \rangle.
	\end{equation}
	The first step is to show that $(\tilde{S}_j : j \geq j_0)$ is asymptotically close to $(S_{jN+i} : j \geq j_0)$.
	\begin{claim}
		\label{clm:6}
		We have
		\[
			\lim_{j \to \infty} \big| S_{jN+i} - \tilde{S}_j \big|= 0
		\]
		uniformly on $\sS^1 \times K$.
	\end{claim}
	For the proof we write
	\begin{align*}
		S_{jN+i} 
		&= a_{(j+1)N+i-1} \sqrt{\gamma_{(j+1)N+i-1}} \langle E \vec{u}_{(j+1)N+i}, \vec{u}_{jN+i} \rangle \\
		&= a_{(j+1)N+i-1} \sqrt{\gamma_{(j+1)N+i-1}} \langle Z_j^* E Z_{j+1} \vec{v}_{j+1}, \vec{v}_j \rangle \\
		&= a_{(j+1)N+i-1} \sqrt{\gamma_{(j+1)N+i-1}} (\det Z_j) \langle E Z_j^{-1} Z_{j+1} \vec{v}_{j+1}, \vec{v}_j \rangle
	\end{align*}
	where we have used that for any $Y \in \GL(2, \RR)$, 
	\begin{equation}
		\label{eq:53}
		(Y^{-1})^* E = \frac{1}{\det Y} E Y.
	\end{equation}
	Now, by Theorem~\ref{thm:2}
	\begin{align*}
		S_{jN+i} - \tilde{S}_j &=
		a_{(j+1)N+i-1} \sqrt{\gamma_{(j+1)N+i-1}} 
		(\det Z_j) 
		\big\langle E (Z_j^{-1} Z_{j+1} - \Id) \vec{v}_{j+1}, \vec{v}_j \big\rangle \\
		&=
		a_{(j+1)N+i-1} \sqrt{\gamma_{(j+1)N+i-1}}
		(\det Z_j) \sqrt{\frac{\alpha_{i-1}}{\gamma_{(j+1)N+i-1}}}
		\big\langle E Q_j \vec{v}_{j+1}, \vec{v}_j \big\rangle.
	\end{align*}
	Observe that by \eqref{eq:47} and \eqref{eq:48}
	\[
		\|D_k\|^2 = \abs{\lambda_k}^2 = \lambda_k \overline{\lambda_k} = \det Y_k.
	\]
	Therefore, by \eqref{eq:46}, 
	\[
		\prod_{k = j_0}^{j-1} \|D_k\|^2 = \frac{\det Z_{j_0}}{\det Z_j} \frac{a_{j_0 N+i-1}}{a_{jN+i-1}}.
	\]
	Next, in view of Claim \ref{clm:1}, for $j \geq j_0$,
	\[
		\|\vec{v}_j\|^2 \lesssim \prod_{k = j_0}^{j-1} \|D_k\|^2 \lesssim \frac{1} {a_{jN+i-1} |\det Z_j|}.
	\]
	Hence, 
	\begin{align*}
		\big|S_{jN+i} - \tilde{S}_j\big|
		&\lesssim
		a_{(j+1)N+i-1} \sqrt{\gamma_{(j+1)N+i-1}}
		\sqrt{\frac{\alpha_{i-1}}{\gamma_{(j+1)N+i-1}}}
		|\det Z_j| \cdot\|Q_j\| \cdot \|\vec{v}_j\|^2 \\
		&\lesssim
		\|Q_j\|
	\end{align*}
	and the claim follows by Theorem \ref{thm:2}.

	We show next that the sequence $(\tilde{S}_j : j \geq j_0)$ converges uniformly on $\sS^1 \times K$ to a positive
	continuous function. By \eqref{eq:66} and \eqref{eq:53}, we have
	\begin{align*}
		\tilde{S}_j 
		&= a_{(j+1)N+i-1} \sqrt{\gamma_{(j+1)N+i-1}} (\det Z_j) \langle E \vec{v}_{j+1}, Y_j^{-1} \vec{v}_{j+1} \rangle \\
		&= a_{(j+1)N+i-1} \sqrt{\gamma_{(j+1)N+i-1}} (\det Z_j) 
		\langle (Y_j^{-1})^* E \vec{v}_{j+1}, \vec{v}_{j+1} \rangle \\
		&= a_{(j+1)N+i-1} \sqrt{\gamma_{(j+1)N+i-1}} (\det Z_j \cdot \det Y_j^{-1}) \langle E Y_j \vec{v}_{j+1}, \vec{v}_{j+1} \rangle,
	\end{align*}
	and since
	\[
		\det Y_j = \det \big( Z_{j+1}^{-1} X_{jN+i} Z_j \big),
	\]
	we obtain
	\begin{equation}
		\label{eq:54}
		\tilde{S}_j = 
		a_{(j+1)N+i-1} \sqrt{\gamma_{(j+1)N+i-1}}
		(\det Z_{j+1}) \frac{a_{(j+1)N+i-1}}{a_{jN+i-1}} 
		\langle E Y_j \vec{v}_{j+1}, \vec{v}_{j+1} \rangle.
	\end{equation}
	By \eqref{eq:66} we have
	\[
		\tilde{S}_{j+1} = 
		a_{(j+2)N+i-1} \sqrt{\gamma_{(j+2)N+i-1}}
		(\det Z_{j+1}) 
		\langle E Y_{j+1} \vec{v}_{j+1}, \vec{v}_{j+1} \rangle.
	\]
	Therefore, by Theorem \ref{thm:1}
	\begin{align*}
		\tilde{S}_{j+1} - \tilde{S}_j
		=
		\varepsilon
		a_{(j+1)N+i-1} \sqrt{\gamma_{(j+1)N+i-1}}
		(\det Z_{j+1}) 
		\left\langle
			E W_j \vec{v}_{j+1}, \vec{v}_{j+1}
		\right\rangle
	\end{align*}
	where
	\[
		W_j =
		\frac{a_{(j+2)N+i-1}}{a_{(j+1)N+i-1}} \frac{\sqrt{\gamma_{(j+2)N+i-1}}}{\sqrt{\gamma_{(j+1)N+i-1}}}
		\sqrt{\frac{\alpha_{i-1}}{\gamma_{(j+2)N+i-1}}}
		R_{j+1} 
		-
		\frac{a_{(j+1)N+i-1}}{a_{jN+i-1}} \sqrt{\frac{\alpha_{i-1}}{\gamma_{(j+1)N+i-1}}} R_j.
	\]
	Hence,
	\begin{align*}
		W_j = \sqrt{\frac{\alpha_{i-1}}{\gamma_{(j+1)N+i-1}}}
		\bigg(  
		\frac{a_{(j+2)N+i-1}}{a_{(j+1)N+i-1}} R_{j+1} - \frac{a_{(j+1)N+i-1}}{a_{jN+i-1}} R_j \bigg),
	\end{align*}
	and so
	\[
		\| W_j \|
		\lesssim 
		\sqrt{\frac{\alpha_{i-1}}{\gamma_{(j+1)N+i-1}}}
		\Big( 
		\Big| \Delta \Big( \frac{a_{(j+1)N+i-1}}{a_{jN+i-1}} \Big) \Big| +
		\big\| \Delta R_j \big\|
		\Big).
	\]
	Therefore,
	\[
		\big|\tilde{S}_{j+1} - \tilde{S}_j\big|
		\lesssim
		a_{(j+1)N+i-1} 
		|\det Z_{j+1}|
		\Big(
		\Big| \Delta \Big( \frac{a_{(j+1)N+i-1}}{a_{jN+i-1}} \Big) \Big| +
		\big\| \Delta R_j \big\|
		\Big) \|\vec{v}_{j+1}\|^2.
	\]
	On the other hand, by \eqref{eq:54},
	\[
		\tilde{S}_j = 
		\varepsilon
		a_{(j+1)N+i-1} \sqrt{\gamma_{(j+1)N+i-1}}
		(\det Z_{j+1}) 
		\frac{a_{(j+1)N+i-1}}{a_{jN+i-1}}
		\sqrt{\frac{\alpha_{i-1}}{\gamma_{(j+1)N+i-1}}}
		\big\langle 
			E R_j \vec{v}_{j+1}, \vec{v}_{j+1} 
		\big\rangle.
	\]
	Since 
	\begin{equation}
		\label{eq:28}
		\lim_{j\to \infty} \sym (E R_j) =
		\sym(E \calR_i) = 
		\frac{\sqrt{|\tau|}}{2} 
		\begin{pmatrix}
			-1 & 1\\
			1 & -1
		\end{pmatrix}
		-
		\frac{\upsilon}{2 \sqrt{|\tau}|}
		\begin{pmatrix}
			1 & 1 \\
			1 & 1
		\end{pmatrix}
		-
		\frac{\frakS}{2}
		\begin{pmatrix}
			1 & 0 \\
			0 & -1
		\end{pmatrix}
	\end{equation}
	and the matrix on the right-hand side of \eqref{eq:28} has determinant equal to $-\tau > 0$, we obtain
	\[
		|\tilde{S}_j| \gtrsim
		a_{(j+1)N+i-1} 
		|\det Z_{j+1}| \cdot
		\|\vec{v}_{j+1} \|^2.
	\]
	Consequently, we arrive at
	\[
		|\tilde{S}_{j+1} - \tilde{S}_j| \lesssim 
		\Big( 
		\Big| \Delta \Big( \frac{a_{(j+1)N+i-1}}{a_{jN+i-1}} \Big) \Big| +
		\| \Delta R_j \|
		\Big) |\tilde{S}_j|.
	\]
	Since $\tilde{S}_j \neq 0$ on $K$, we get
	\[
		\sum_{j = j_0}^\infty
		\sup_{\eta \in \sS^1} \sup_{x \in K} {
		\bigg|
		\frac{\abs{\tilde{S}_{j+1}(\eta, x)}}{\abs{\tilde{S}_j(\eta, x)}} - 1 
		\bigg|}
		\lesssim
		\sum_{j = j_0}^\infty
		\Big| \Delta \Big( \frac{a_{(j+1)N+i-1}}{a_{jN+i-1}} \Big) \Big| + 
		\sup_{x \in K}{\| \Delta R_j(x) \|},
	\]
	which implies that the product
	\[
		\prod_{k = j_0}^\infty \bigg(1 + \frac{\abs{\tilde{S}_{k+1}} - \abs{\tilde{S}_k}}{\abs{\tilde{S}_k}}\bigg)
	\]
	converges uniformly on $\sS^1 \times K$ to a positive continuous function. Because
	\[
		\bigg| \frac{\tilde{S}_j}{\tilde{S}_{j_0}} \bigg|
		=
		\prod_{k = j_0}^{j-1} \bigg(1 + \frac{\abs{\tilde{S}_{k+1}} - \abs{\tilde{S}_k}}{\abs{\tilde{S}_k}}\bigg),
	\]
	the same holds true for the sequence $(\tilde{S}_j : j \geq j_0)$. In view of Claim \ref{clm:6},
	the proof is completed.
\end{proof}

\begin{corollary} 
	\label{cor:4}
	Suppose that the hypotheses of Theorem~\ref{thm:5} are satisfied. Then for any compact $K \subset \Lambda_-$ there is
	a constant $c>1$ such that for any generalized eigenvector $\vec{u}$ associated with $x\in K$ and corresponding to
	$\eta \in \sS^1$, we have
	\[
		c^{-1} \leq \frac{a_{(j+1)N+i-1}}{\sqrt{\gamma_{(j+1)N+i-1}}} \big\| \vec{v}_j \big\|^2 \leq c
	\]
	where $\vec{v}_j = Z_{j}^{-1} \vec{u}_{jN+i}$.
\end{corollary}
\begin{proof}
	By \eqref{eq:66} and Theorem~\ref{thm:1} we have
	\begin{align*}
		\tilde{S}_j 
		&= 
		a_{(j+1)N+i-1} \sqrt{\gamma_{(j+1)N+i-1}} (\det Z_j) \langle E Y_j \vec{v}_j, \vec{v}_j \rangle \\
		&= \varepsilon a_{(j+1)N+i-1} \sqrt{\alpha_{i-1}} (\det Z_j) \langle E R_j \vec{v}_j, \vec{v}_j \rangle.
	\end{align*}
	Hence, by \eqref{eq:28}, we have
	\[
		|\tilde{S}_j| \asymp a_{(j+1)N+i-1} |\det Z_j| \| \vec{v}_j \|^2.
	\]
	Observe that
	\begin{align*}
		\lim_{j \to \infty}
		\sqrt{\gamma_{(j+1)N+i-1}} \det Z_j 
		&= 
		\lim_{j \to \infty}
		\frac{-2 \sinh \vartheta_j}{\vartheta_j} \vartheta_j \sqrt{\gamma_{(j+1)N+i-1}} \\
		&=
		-2 \sqrt{\alpha_{i-1} |\tau(x)|}
	\end{align*}
	uniformly on $K$. By the fact that $\tilde{S}_j$ is uniformly convergent on $\sS^1 \times K$ to a positive function,
	the conclusion follows.
\end{proof}

\section{Asymptotics of the generalized eigenvectors} \label{sec:7}
In this section we study the asymptotic behavior of generalized eigenvectors. We keep the notation introduced 
in Section \ref{sec:5}.

\begin{theorem} \label{thm:4}
	Let $N$ be a positive integer. Let $(\gamma_n : n \in \NN)$ be a sequence of positive
	numbers tending to infinity and satisfying \eqref{eq:87} and \eqref{eq:89}. Let $(a_n : n \in \NN_0)$ and
	$(b_n : n \in \NN_0)$ be $\gamma$-tempered $N$-periodically modulated Jacobi parameters such that $\frakX_0(0)$ is
	a non-trivial parabolic element. Suppose that \eqref{eq:88} holds true with $\varepsilon = \sign{\tr \frakX_0(0)}$.
	Then for each $i \in \{0, 1, \ldots, N-1\}$ and every compact interval $K \subset \Lambda_-$, there are $j_0 \in \NN$
	and a~continuous function $\vphi: \sS^1 \times K \rightarrow \CC$ such that for every generalized eigenvector
	$(u_n : n \in \NN_0)$,
	\[
		\lim_{j \to \infty}
		\sup_{\eta \in \sS^1} 
		\sup_{x \in K}
		\bigg|
		\frac{\sqrt{\gamma_{(j+1)N+i-1}}}{\prod_{k = j_0}^{j-1} \lambda_k(x)}
		\Big(u_{(j+1)N+i}(\eta, x) - \overline{\lambda_j(x)} u_{jN+i}(\eta, x)\Big)
		-
		\vphi(\eta, x)
		\bigg|
		=0.
	\]
	Moreover, $\vphi(\eta,x) = 0$ if and only if $[\frakX_i(0)]_{21} = 0$. Furthermore,
	\[
		\frac{u_{jN+i}(\eta, x)}{\prod_{k = j_0}^{j-1} |\lambda_k(x)|}
		=
		\frac{|\vphi(\eta, x)|}{\sqrt{\alpha_{i-1} |\tau(x)|}}
		\sin
		\Big(
		\sum_{k = j_0}^{j-1} \theta_k(x) + \arg \vphi(\eta, x)
		\Big)
		+
		E_j(\eta, x)
	\]
	where
	\begin{equation}
		\label{eq:64}
		\theta_k(x) = \arccos\bigg(\frac{\tr Y_k(x)}{2 \sqrt{\det Y_k(x)}} \bigg)
	\end{equation}
	and
	\[
		\lim_{j \to \infty} \sup_{\eta \in \sS^1} \sup_{x \in K} {|E_j(\eta, x)|} = 0.
	\]
\end{theorem}
\begin{proof}
	In the proof we use the uniform diagonalization constructed in Section \ref{sec:5}. For $j > j_0$, we set
	\[
		\phi_j = \frac{u_{(j+1)N+i} - \overline{\lambda_j} u_{jN+i}}{\prod_{k = j_0}^{j-1} \lambda_k}.
	\]
	Let us observe that there is $c > 0$ such that for all $j \in \NN$, and $x \in K$,
	\begin{equation}
		\label{eq:49}
		\left\|
		Z_j^t e_2 - 
		\begin{pmatrix}
			1 & 1 \\
			1 & 1
		\end{pmatrix}
		T_i^t e_2
		\right\|
		\leq
		c \vartheta_j.
	\end{equation}
	We show that the sequence $(\sqrt{\gamma_{(j+1) N + i-1}} \phi_j : j > j_0)$ converges uniformly on $K$. By
	\eqref{eq:47}, $\|D_j\| = |\lambda_j|$, thus by Claim \ref{clm:1} 
	\begin{align*}
		\big|
		u_{(j+1)N+i} - \sprod{\vec{v}_j}{Z_j^t e_2}
		\big|
		&=
		\big|
		\big\langle
		\vec{v}_{j+1}, (Z^t_{j+1} - Z^t_j) e_2
		\big\rangle
		\big| \\
		&\leq
		\|\vec{v}_{j+1}\| \cdot |\vartheta_{j+1} - \vartheta_j| \\
		&\lesssim
		\bigg(\prod_{k = j_0}^{j-1} |\lambda_k| \bigg) |\vartheta_{j+1} - \vartheta_j|.
	\end{align*}
	Hence,
	\[
		\lim_{j \to \infty}
		\sqrt{\gamma_{(j+1)N+i-1}} \frac{\big| u_{(j+1)N+i} - \sprod{\vec{v}_j}{Z_j^t e_2} \big|}
		{\prod_{k = j_0}^{j-1} |\lambda_k|} = 0
	\]
	uniformly on $K$. Next, by \eqref{eq:48} we write
	\[
		(Y_j - \overline{\lambda_j} \Id) \vec{v}_j = C_j
		\begin{pmatrix}
			\lambda_j - \overline{\lambda_j} & 0 \\
			0 & 0
		\end{pmatrix}
		C_j^{-1} \vec{v}_j,
	\]
	thus by \eqref{eq:49} we obtain
	\begin{align}
		\nonumber
		\left|
		\bigg\langle (Y_j - \overline{\lambda_j} \Id) \vec{v}_j, Z_j^t e_2 \bigg\rangle
		-
		\bigg\langle (Y_j - \overline{\lambda_j} \Id) \vec{v}_j, 
		\begin{pmatrix}
			1 & 1\\
			1 & 1
		\end{pmatrix}
		T_i^t e_2
		\bigg\rangle
		\right|
		&\leq \vartheta_j |\lambda_j - \overline{\lambda_j}| \cdot \|\vec{v}_j\| \\
		\label{eq:52}
		&\lesssim
		\vartheta_j^2 \Big(\prod_{k = j_0}^{j-1} |\lambda_k|\Big)
	\end{align}
	where in the last estimate we have used
	\begin{equation}
		\label{eq:51}
		|\lambda_j-\overline{\lambda_j}| = 
		|2i \Im(\lambda_j))| =
		\sqrt{\frac{\alpha_{i-1}}{\gamma_{(j+1)N+i-1}}} \sqrt{-\discr R_j}.
	\end{equation}
	Hence, it is enough to show that the sequence $(\tilde{\phi}_j : j > j_0)$
	where
	\[
		\tilde{\phi}_j = 
		\frac{\sqrt{\gamma_{(j+1)N+i-1}}}{\prod_{k = j_0}^{j-1} \lambda_k} 
		\Big\langle \big( Y_j - \overline{\lambda_j} \Id \big) \vec{v}_j, 
		\begin{pmatrix}
			1 & 1 \\
			1 & 1
		\end{pmatrix}
		T_i^t e_2
		\Big\rangle
	\]
	converges uniformly on $\sS^1 \times K$. To do so, for a given $\epsilon > 0$ there is $L_0 > j_0$ such that
	for all $L \geq L_0$ we have
	\[
		\sum_{k=L-1}^\infty \sup_{K} \| \Delta C_k \| < \epsilon.
	\]
	For $j \geq L$, we set
	\[
		\psi_{j;L} = 
		\frac{\sqrt{\gamma_{(j+1)N+i-1}}}{\prod_{k = j_0}^{j-1} \lambda_k} 
		\Big\langle 
		C_{j} \big( D_j - \overline{\lambda_j} \Id \big) 
		\Big( \prod_{k=L}^{j-1} D_k \Big) C_{L-1}^{-1} \vec{v}_L, 
		\begin{pmatrix}
			1 & 1 \\
			1 & 1
		\end{pmatrix}
		T_i^t e_2
		\Big\rangle.
	\]
	Observe that
	\begin{align*}
		(Y_j - \overline{\lambda_j} \Id) \vec{v}_j &=
		(Y_j - \overline{\lambda_j} \Id) \Big( \prod_{k=L}^{j-1} Y_k \Big) \vec{v}_L \\
		&=
		C_j (D_j - \overline{\lambda_j} \Id) C_j^{-1} \Big( \prod_{k=L}^{j-1} C_k D_k C_k^{-1} \Big) \vec{v}_L \\
		&=
		C_j (D_j - \overline{\lambda_j} \Id) D_j^{-1} \Big( \prod_{k=L}^{j} D_k C_k^{-1} C_{k-1} \Big) C_{L-1}^{-1} \vec{v}_L.
	\end{align*}
	Hence, by \cite[Proposition 1]{SwiderskiTrojan2019}, there is $c > 0$ such that for all $j \geq L \geq j_0$,
	\[
		\bigg\| \prod_{k=L}^j D_k C_k C^{-1}_{k-1} - \prod_{k=L}^j D_k \bigg\| \leq
		c \prod_{k=L}^j |\lambda_k| \cdot \sum_{k=L-1}^\infty \sup_{K} \| \Delta C_k \|.
	\]
	Thus, by Claim~\ref{clm:1}, \eqref{eq:51} and the fact that $\|D_k\| = |\lambda_k|$, we obtain
	\begin{equation} \label{eq:120}
		|\tilde{\phi}_j - \psi_{j;L}| \leq 
		c \sum_{k=L-1}^\infty \sup_{K} \| \Delta C_k \| \leq c\epsilon,
	\end{equation}
	for all $j \geq L$. Hence, for all $n > m \geq L$,
	\[
		|\tilde{\phi}_n - \tilde{\phi}_m| \leq c \epsilon + |\psi_{n; L} - \psi_{m; L}|.
	\]
	Therefore, our task is reduced to showing that the sequence $(\psi_{j;L} : j \geq L)$ converges uniformly on $K$.
	Since, by \eqref{eq:51}
	\begin{align*}
		\frac{\sqrt{\gamma_{(j+1)N+i-1}}}{\prod_{k=L}^{j-1} \lambda_k} 
		\big( D_j - \overline{\lambda_j} \Id \big) 
		\prod_{k=L}^{j-1} D_k
		&=
		i
		\sqrt{\alpha_{i-1}} \sqrt{-\discr R_j} 
		\begin{pmatrix}
			1 & 0 \\
			0 & 0
		\end{pmatrix}
		\prod_{k=L}^{j-1} \frac{D_k}{\lambda_k} \\
		&=
		i \sqrt{\alpha_{i-1}} \sqrt{-\discr R_j}  
		\begin{pmatrix}
			1 & 0 \\
			0 & 0
		\end{pmatrix}
	\end{align*}
	we get uniformly on $K$
	\begin{equation} \label{eq:121}
		\lim_{j \to \infty} \psi_{j; L} =
		\frac{i \sqrt{\alpha_{i-1}} \sqrt{-\discr \calR_i}}{\prod_{k=j_0}^{L-1} \lambda_k}
		\Big\langle
			C_{\infty}
			\begin{pmatrix}
				1 & 0 \\
				0 & 0
			\end{pmatrix}
			C_{L-1}^{-1}
			\vec{v}_{L},
			\begin{pmatrix}
				1 & 1 \\
				1 & 1
			\end{pmatrix}
			T_i^t e_2
		\Big\rangle
	\end{equation}
	where
	\begin{equation} \label{eq:122}
		C_{\infty} = \lim_{j \to \infty} C_j =
		\begin{pmatrix}
			1 & 1\\
			\frac{\xi_{\infty} - [\calR_i]_{11}}{[\calR_i]_{12}} & \frac{\overline{\xi_{\infty}} - [\calR_i]_{11}}{[\calR_i]_{12}}
		\end{pmatrix}.
	\end{equation}
	Thus, we have proved that both sequences $(\tilde{\phi}_j : j > j_0)$ and $(\psi_{j;L} : j \geq L)$ converge uniformly on
	$K$. Let us denote its limits by $\tilde{\phi}_\infty$ and $\psi_{\infty;L}$, respectively. By \eqref{eq:120},
	for all $L \geq L_0$ we have
	\begin{equation} 
		\label{eq:123}
		|\tilde{\phi}_\infty - \psi_{\infty;L}| \leq c \epsilon.
	\end{equation}
	Let us observe that
	\[
		\begin{pmatrix}
			1 & 1 \\
			1 & 1
		\end{pmatrix}
		C_{\infty}
		\begin{pmatrix}
			1 & 0 \\
			0 & 0
		\end{pmatrix}
		= \big( [C_{\infty}]_{11} + [C_\infty]_{21} \big)
		\begin{pmatrix}
			1 & 0 \\
			1 & 0
		\end{pmatrix}.
	\]
	By \eqref{eq:122} the expression on the right-hand side has non-zero imaginary part. Thus from \eqref{eq:121} we can write 
	\begin{align}
		\nonumber
		\psi_{\infty;L} &= 
		\frac{h}{\prod_{k=j_0}^{L-1} \lambda_k} 
		\Big\langle
			C_{L-1}^{-1} \vec{v}_L, 
			\begin{pmatrix}
				1 & 1 \\
				0 & 0
			\end{pmatrix}
			T_i^t e_2
		\Big\rangle \\
		\label{eq:93}
		&=
		\frac{h}{\prod_{k=j_0}^{L-1} \lambda_k} 
		\big( [T_i]_{21} + [T_i]_{22} \big)
		\langle
			C_{L-1}^{-1} \vec{v}_L, e_1
		\rangle
	\end{align}
	for some function $h$ without zeros on $\sS^1 \times K$. Thus, by \eqref{eq:34}, if $[\frakX_{i}(0)]_{21} = 0$, then
	$\psi_{\infty;L} \equiv 0$ for all $L$. Consequently, by \eqref{eq:123}, $\tilde{\phi}_{\infty} \equiv 0$ on $\sS^1 \times K$.
	On the other hand, if $[\frakX_{i}(0)]_{21} \neq 0$, then the following claim holds true.
	\begin{claim}
		\label{clm:4}
		For each $(\eta, x) \in \sS^1 \times K$,
		\[
			\liminf_{L \to \infty} \abs{\psi_{\infty;L}(\eta, x)} > 0.
		\]
	\end{claim}
	On the contrary, let us suppose that there are $\eta \in \sS^1$, $x \in K$ and a sequence $(L_j : j \in \NN)$ such that
	\[
		\lim_{j \to \infty} L_j = \infty,
	\]
	and
	\[
		\lim_{j \to \infty} \psi_{\infty;L_j} (\eta,x) = 0.
	\]
	Setting $\vec{v}_L = v^{L}_1 e_1 + v^L_2 e_2$, we have
	\begin{align*}
		\langle C_{L-1}^{-1} \vec{v}_L, e_1 \rangle
		&=
		\bigg\langle
		\begin{pmatrix}
			\frac{\overline{\xi_{L-1}} - [R_{L-1}]_{11}}{[R_{L-1}]_{12}} & -1 \\
			\frac{\xi_{L-1} - [R_{L-1}]_{11}}{[R_{L-1}]_{12}} & 1
		\end{pmatrix}			
		\vec{v}_L, e_1
		\bigg\rangle \\
		&=
		\frac{\overline{\xi_{L-1}} - [R_{L-1}]_{11}}{[R_{L-1}]_{12}} v^L_1 - v^L_2.
	\end{align*}
	Hence, by \eqref{eq:93}, we obtain
	\[
		\lim_{j \to \infty} 
		\frac{1}{\prod_{k=j_0}^{L_j-1} |\lambda_k(x)|} 
		\bigg( \frac{\overline{\xi_{L_j-1}(x)} - [R_{L_j-1}(x)]_{11}}{[R_{L_j-1}(x)]_{12}} v^{L_j}_1(\eta, x) 
		- v^{L_j}_2(\eta, x) \bigg) = 0.
	\]
	In view of \eqref{eq:46a}, $\vec{v}_{L_j}(\eta, x)$ is a real vector, thus by taking imaginary parts of the bracket, 
	we conclude that
	\[
		\lim_{j \to \infty} \frac{v^{L_j}_1(\eta, x)}{\prod_{k=j_0}^{L_j-1} |\lambda_k(x)|} = 0.
	\]
	Hence,
	\[
		\lim_{j \to \infty} \frac{v^{L_j}_2(\eta, x)}{\prod_{k=j_0}^{L_j-1} |\lambda_k(x)|} = 0,
	\]
	which in view of Claim \ref{clm:2} contradicts to Corollary~\ref{cor:4} proving the claim.

	Next, let us consider $\eta \in \sS^1$ and $x \in K$. By Claim \ref{clm:4},
	\begin{equation} \label{eq:146}
		A = \liminf_{L \to \infty} \abs{\psi_{\infty;L}(\eta, x)} > 0.
	\end{equation}
	Taking $\epsilon = \frac{A}{2c}$, by \eqref{eq:123}, for all $L \geq L_0$,
	\[
		\abs{\tilde{\phi}_{\infty}(\eta, x)}
		\geq
		\abs{\psi_{\infty;L}(\eta, x)} - c \epsilon
		=
		\abs{\psi_{\infty;L}(\eta, x)} - \frac{A}{2}.
	\]
	Thus, in view of \eqref{eq:146}, 
	\[
		\abs{\tilde{\phi}_{\infty}(\eta, x)}
		\geq
		\frac{A}{2}.
	\]
	Consequently, $\tilde{\phi}_{\infty}$ cannot be zero on $\sS^1 \times K$ provided that $[\frakX_i(0)]_{21} \neq 0$.
	
	In view of \eqref{eq:52} there is a function $\vphi: \sS^1 \times K \rightarrow \RR$, such that
	\[
		\vphi = \lim_{j \to \infty} \sqrt{\gamma_{(j+1)N+i-1}} \phi_j
	\]
	uniformly on $\sS^1 \times K$. In fact, one has $\vphi = \tilde{\phi}_{\infty}$. In particular, we obtain
	\begin{align*}
		\lim_{j \to \infty}
		\sup_{\eta \in \sS^1}
		\sup_{x \in K}{
		\bigg|
		\sqrt{\gamma_{(j+1)N+i-1}} 
		\frac{u_{(j+1)N+i}(\eta, x) - \overline{\lambda_j(x)} u_{jN+i}(\eta, x)}{\prod_{k = j_0}^{j-1} |\lambda_k(x)|}
		-
		\vphi(\eta, x) \prod_{k = j_0}^{j-1} \frac{\lambda_k(x)}{|\lambda_k(x)|}
		\bigg|}
		=0.
	\end{align*}
	Since $u_n(\eta, x) \in \RR$, by taking imaginary part we conclude that
	\begin{align*}
		\lim_{j \to \infty}
		\sup_{\eta \in \sS^1}
		\sup_{x \in K}
		\bigg|
		&
		\frac{\sqrt{\alpha_{i-1}}}{2} 
		\sqrt{-\discr R_j(x)}
		\frac{u_{jN+i}(\eta, x)}{\prod_{k = j_0}^{j-1} |\lambda_k(x)|}
		-
		\abs{\vphi(\eta, x)} 
		\sin\Big(\sum_{k = j_0}^{j-1} \theta_k(x) + \arg \vphi(\eta, x) \Big)
		\bigg|
		=0
	\end{align*}
	where we have also used that
	\[
		-
		\sqrt{\gamma_{(j+1)N+i-1}} 
		\Im(\overline{\lambda_j(x)}) = \frac{\sqrt{\alpha_{i-1}}}{2} \sqrt{-\discr R_j(x)}.
	\]
	Lastly, observe that
	\[
		\bigg|
		\frac{1}{\sqrt{-\discr R_j(x)}} - \frac{1}{2 \sqrt{|\tau(x)|}}\bigg| 
		\lesssim
		\sum_{k = j}^\infty \big\|\Delta R_k(x)\big\|
	\]
	which completes the proof.
\end{proof}

\begin{remark}
	\label{rem:3}
	There is $i \in \{0, 1, \ldots, N-1 \}$ such that $\abs{\vphi_i(\eta, x)} > 0$ for all $x \in K$ and $\eta \in \sS^1$. 
	Indeed, by \cite[Proposition 3]{PeriodicIII}, if $[\frakX_{i-1}(0)]_{21} = 0$ and $[\frakX_i(0)]_{21} = 0$, then
	$\frakX_i(0)$ is a multiple of identity which is a \emph{trivial} parabolic element. Contradiction.
\end{remark}

\section{Christoffel--Darboux kernel on the diagonal} 
\label{sec:8}
In this section we study the asymptotic behavior of generalized Christoffel--Darboux kernel on the diagonal. Given Jacobi parameters
$(a_n : n \in \NN_0)$ and $(b_n : n \in \NN_0)$, and $\eta \in \sS^1$, we set
\[
	K_n(x, y; \eta) = \sum_{m = 0}^n u_m(x, \eta) u_m(y, \eta) , \qquad x, y \in \RR,
\]
where $(u_n(x, \eta) : n \in \NN_0)$ is generalized eigenvector associated to $x$ and corresponding to $\eta$. Let
\[
	\rho_n = \sum_{m = 0}^n \frac{\sqrt{\alpha_m \gamma_m}}{a_m}.
\]
For $N$-periodically modulated Jacobi parameters we also study
\[
	 K_{i;j}(x, y; \eta) = \sum_{k=0}^j u_{kN+i}(x, \eta) u_{kN+i}(y, \eta), \qquad x, y \in \RR,
\]
where $i \in \{0, 1, \ldots, N-1\}$. Let
\[
	\rho_{i;j} = \sum_{k=1}^j \frac{\sqrt{\gamma_{kN+i}}}{a_{kN+i}}.
\]
\begin{lemma} 
	\label{lem:1}
	Let $(\gamma_n : n \in \NN)$ and $(a_n : n \in \NN)$ be sequences of positive numbers such that
	\[
		\lim_{n \to \infty} \gamma_n = \infty,
		\qquad
		\sum_{n=0}^\infty \frac{\sqrt{\gamma_n}}{a_n} = \infty.
	\]
	Suppose that $(\xi_n : n \in \NN)$ is a sequence of real functions on some compact set $K \subset \RR^d$ such that
	\[
		\lim_{n \to \infty} \sup_{x \in K}\big|\sqrt{\gamma_n} \xi_n(x) - \psi(x) \big| = 0
	\]
	for certain function $\psi: K \rightarrow (0, \infty)$ satisfying
	\[
		c^{-1} \leq \psi(x) \leq c, \quad\text{for all } x \in K.
	\]
	We set
	\[
		\Xi_n(x) = \sum_{j = 0}^n \xi_j(x), \qquad\text{and}\qquad
		\Delta_n = \sum_{j = 0}^n \frac{\sqrt{\gamma_j}}{a_j}.
	\]
	If
	\begin{equation}
		\label{eq:69}
		\bigg(\frac{\gamma_n}{a_n} : n \in \NN\bigg) \in \calD_1,
	\end{equation}
	then
	\begin{equation}
		\label{eq:68}
		\lim_{n \to \infty}
		\frac{1}{\Delta_n}
		\sum_{k=0}^n 
		\frac{\sqrt{\gamma_k}}{a_k}
		\cos \big( \Xi_k(x) \big) = 0
	\end{equation}
	uniformly with respect to $x \in K$.
\end{lemma}
\begin{proof}
	First, we write
	\begin{align*}
		\bigg|
		\sum_{k = 0}^n \frac{\sqrt{\gamma_k}}{a_k} \cos \big( \Xi_k(x) \big) 
		-
		\sum_{k = 0}^n \frac{\gamma_k}{a_k} \frac{\xi_k(x)}{\psi(x)} \cos \big( \Xi_k(x) \big)
		\bigg|
		\leq
		\sum_{k = 0}^n \frac{\sqrt{\gamma_k}}{a_k} \bigg|1 - \sqrt{\gamma_k} \frac{\xi_k(x)}{\psi(x)}\bigg|.
	\end{align*}
	Since by the Stolz--Ces\`{a}ro theorem
	\[
		\lim_{n \to \infty} \frac{1}{\Delta_n} 
		\sum_{k = 0}^n \frac{\sqrt{\gamma_k}}{a_k} \bigg|1 - \sqrt{\gamma_k} \frac{\xi_k(x)}{\psi(x)}\bigg|
		=
		\lim_{n \to \infty} \bigg|1 - \sqrt{\gamma_k} \frac{\xi_k(x)}{\psi(x)}\bigg| = 0,
	\]
	uniformly with respect to $x \in K$, we obtain
	\[
		\lim_{n \to \infty}
		\frac{1}{\Delta_n} \sum_{k = 0}^n \frac{\sqrt{\gamma_k}}{a_k} \cos \big( \Xi_k(x) \big)
		=
		\lim_{n \to \infty}
		\frac{1}{\Delta_n} \sum_{k = 0}^n \frac{\gamma_k}{a_k} \frac{\xi_k(x)}{\psi(x)} \cos \big( \Xi_k(x) \big).
	\]
	Next, we observe that
	\begin{align*}
		\bigg|
		\sum_{k = 1}^n \frac{\gamma_k}{a_k} \bigg(\xi_k(x) \cos \big( \Xi_k(x) \big) 
		- \int_{\Xi_{k-1}(x)}^{\Xi_k(x)} \cos(t) \ud t \bigg)
		\bigg|
		&\leq
		\sum_{k = 1}^n 
		\frac{\gamma_k}{a_k}
		\int_{\Xi_{k-1}(x)}^{\Xi_k(x)} \big|\cos(t) - \cos \big( \Xi_k(x) \big) \big| \ud t \\
		&\leq
		\frac{1}{2}
		\sum_{k = 1}^n
		\frac{\gamma_k}{a_k} |\xi_k(x)|^2.
	\end{align*}
	In view of the Stolz--Ces\`{a}ro theorem
	\begin{align*}
		\lim_{n \to \infty} \frac{1}{\Delta_n} 
		\sum_{k = 1}^n \frac{\gamma_k}{a_k} |\xi_k(x)|^2 
		&=
		\lim_{n \to \infty} \sqrt{\gamma_n} |\xi_n(x)|^2 = 0,
	\end{align*}
	thus
	\[
		\lim_{n \to \infty} \frac{1}{\Delta_n} 
		\sum_{k = 0}^n \frac{\sqrt{\gamma_k}}{a_k} \cos \big( \Xi_k(x) \big)
		=
		\lim_{n \to \infty}
		\frac{1}{\Delta_n} \sum_{k = 1}^n \frac{\gamma_k}{a_k} \big(\sin \Xi_k(x) - \sin \Xi_{k-1}(x)\big).
	\]
	Now, by the summation by parts we get
	\begin{align*}
		\sum_{k = 1}^n 
		\frac{\gamma_k}{a_k} \big(\sin \Xi_k(x) - \sin \Xi_{k-1}(x)\big)
		&=
		\frac{\gamma_n}{a_n} \sin \Xi_n(x) - \frac{\gamma_1}{a_1} \sin \Xi_0(x) \\
		&+ \sum_{k = 1}^{n-1} \bigg(\frac{\gamma_k}{a_k} - \frac{\gamma_{k+1}}{a_{k+1}}\bigg)
		\sin \Xi_k(x).
	\end{align*}
	Thus, by \eqref{eq:69},
	\begin{align*}
		\sup_{x \in K}
		\bigg|
		\sum_{k = 1}^n
		\frac{\gamma_k}{a_k} \big(\sin \Xi_k(x) - \sin \Xi_{k-1}(x)\big)
		\bigg|
		\leq
		2 \frac{\gamma_1}{a_1} +
		2 \sum_{k = 1}^\infty \bigg|\frac{\gamma_{k+1}}{a_{k+1}} - \frac{\gamma_k}{a_k} \bigg|.
	\end{align*}
	Consequently,
	\[
		\lim_{n \to \infty}
		\frac{1}{\Delta_n} \sum_{k = 1}^n \frac{\gamma_k}{a_k} \big(\sin \Xi_k(x) - \sin \Xi_{k-1}(x)\big)
		=0
	\]
	and the lemma follows.
\end{proof}

\begin{theorem} 
	\label{thm:7}
	Let $N$ be a positive integer. Let $(\gamma_n : n \in \NN)$ be a sequence of positive
	numbers tending to infinity and satisfying \eqref{eq:87} and \eqref{eq:89}. Let $(a_n : n \in \NN_0)$ and
	$(b_n : n \in \NN_0)$ be $\gamma$-tempered $N$-periodically modulated Jacobi parameters such that $\frakX_0(0)$ is
	a non-trivial parabolic element. Suppose that \eqref{eq:88} holds true with $\varepsilon = \sign{\tr \frakX_0(0)}$.
	If
	\begin{equation}
		\label{eq:65}
		\lim_{n \to \infty} \rho_n = \infty,
	\end{equation}
	then there is $n_0 \geq 1$ such that for all $n \geq n_0$,
	\[
		\lim_{n \to \infty} \frac{1}{\rho_{n}} K_{n} (x, x; \eta)
		=
		\frac{1}{2 N}
		\sum_{i = 0}^{N-1}
		\frac{\abs{\vphi_{i}(\eta, x)}^2  a_{j_0N+i-1} \sinh \vartheta_{j_0N+i-1}(x)}
		{\big(\sqrt{\alpha_{i-1} \abs{\tau(x)}}\big)^3}
	\]
	locally uniformly with respect to $(x, \eta) \in \Lambda_- \times \sS^1$.
\end{theorem}
\begin{proof}
	Let $K$ be a compact interval with non-empty interior contained in $\Lambda_-$. By Theorem~\ref{thm:4} and 
	Claim~\ref{clm:2}, there is $j_0 \geq 1$ such that for $x \in K$, $\eta \in \sS^1$, and $k > j_0$,
	\begin{align*}
		\frac{a_{(k+1)N+i-1}}{\sqrt{\alpha_{i-1} \gamma_{(k+1)N+i-1}}}
		u_{kN+i}^2(\eta, x) 
		&=
		\frac{\abs{\vphi_i(\eta, x)}^2 a_{j_0 N+i-1} \sinh \vartheta_{j_0N+i}(x) }
		{\big(\sqrt{\alpha_{i-1} \abs{\tau(x)}}\big)^3}
		\sin^2\Big(\sum_{\ell=j_0}^{k-1} \theta_{\ell N + i}(x) + \arg \varphi_i(\eta,x) \Big) \\
		&\phantom{=}+
		E_{kN+i}(\eta, x)
	\end{align*}
	where
	\[
		\lim_{k \to \infty}
		\sup_{\eta \in \sS^1}
		\sup_{x \in K} |E_{kN+i}(\eta, x)| = 0.
	\]
	Therefore, 
	\begin{align*}
		&\sum_{k = j_0+1}^j
		u_{kN+i}^2(\eta, x) \\
		&\quad=
        \frac{\abs{\vphi_i(\eta, x)}^2 a_{j_0 N+i-1} \sinh \vartheta_{j_0N+i}(x)}
		{2 \big(\sqrt{\alpha_{i-1} \abs{\tau(x)}}\big)^3}
		\sum_{k = j_0+1}^j
		\frac{\sqrt{\alpha_{i-1} \gamma_{(k+1)N+i-1}}}{a_{(k+1)N+i-1}}
		\bigg(1 - \cos\Big(2\sum_{\ell=j_0}^{k-1} \theta_{\ell N + i}(x) + 2\arg \varphi_i(\eta,x) \Big)\bigg) \\
		&\phantom{=}\quad+
		\sum_{k = j_0+1}^j
		\frac{\sqrt{\alpha_{i-1} \gamma_{(k+1)N+i-1}}}{a_{(k+1)N+i-1}} E_{kN+i}(\eta, x).
	\end{align*}
	We claim that
	\begin{equation}
		\label{eq:75}
		\lim_{j \to \infty} \frac{1}{\sqrt{\alpha_{i-1}} \rho_{i-1; j}}
		K_{i; j}(x, x; \eta)
		=
		\frac{\abs{\vphi_i(\eta, x)}^2 a_{j_0 N+i-1} \sinh \vartheta_{j_0N+i-1}(x)}
		{2 \big(\sqrt{\alpha_{i-1} \abs{\tau(x)}} \big)}
	\end{equation}
	uniformly with respect to $(x, \eta) \in K \times \sS^1$. To see this, we observe that by the Stolz--Ces\`{a}ro theorem,
	\[
		\lim_{j \to \infty} \frac{1}{\rho_{i-1; j}} \sum_{k = j_0+1}^j 
		\frac{\sqrt{\gamma_{(k+1)N+i-1}}}{a_{(k+1)N+i-1}} E_{kN+i}(\eta, x)
		=
		\lim_{j \to \infty} \sqrt{\frac{\gamma_{(j+1)N+i-1}}{\gamma_{jN+i-1}}} 
		\cdot \frac{a_{jN+i-1}}{a_{(j+1)N+i-1}} E_{jN+i}(\eta, x) = 0.
	\]
	Since there is $c > 0$ such that
	\[
		\sup_{\eta \in \sS^1} \sup_{x \in K} \sum_{k = 0}^{j_0} u_{kN+i}^2(\eta, x) \leq c
	\]
	to prove \eqref{eq:75} it is enough to show that 
	\begin{equation}
		\label{eq:70}
		\lim_{j \to \infty}
		\frac{1}{\rho_{i-1; j}}
		\sum_{k = j_0+1}^j
		\frac{\sqrt{\alpha_{i-1} \gamma_{(k+1)N+i-1}}}{a_{(k+1)N+i-1}}
		\cos\Big(2 \sum_{\ell=j_0}^{k-1} \theta_{\ell N + i}(x) + 2\arg \varphi_i(\eta,x) \Big)
		=0
	\end{equation}
	uniformly with respect to $(x, \eta) \in K \times \sS^1$. Observe that \eqref{eq:70} is an easy consequence of
	Lemma \ref{lem:1}, provided we show the following statement.
	\begin{claim}
		\label{clm:3}
		For all $i \in \{0, 1, \ldots, N-1\}$,
		\[
			\lim_{j \to \infty} \sqrt{\frac{\gamma_{(j+1)N+i-1}}{\alpha_{i-1}}} \theta_{jN+i}(x)
			=
			\sqrt{\abs{\tau(x)}}
		\]
		uniformly with respect to $x \in K$.
	\end{claim}
	Using the notation introduced in Section \ref{sec:5}, Theorem~\ref{thm:1} gives
	\[
		\lim_{j \to \infty} Y_j = \varepsilon \Id
	\]
	locally uniformly on $\Lambda_-$. In particular,
	\[
		\lim_{j \to \infty} \frac{\tr Y_j(x)}{2 \sqrt{\det Y_j(x)}} = \varepsilon.
	\]
	Since
	\[
		\lim_{t \to 1^-} \frac{\arccos t}{\sqrt{1-t^2}} =1,
	\]
	we obtain
	\[
		\lim_{j \to \infty} \bigg(1-\bigg( \frac{\tr Y_j(x)}{2 \sqrt{\det Y_j(x)}} \bigg)^2  \bigg)^{-1/2} \theta_j(x) 
		= 1.
	\]
	Let us observe that, by Theorem~\ref{thm:1},
	\[
		\sqrt{1-\bigg( \frac{\tr Y_j(x)}{2 \sqrt{\det Y_j(x)}} \bigg)^2} = 
		\frac{\sqrt{-\discr Y_j(x)}}{2 \sqrt{\det Y_j(x)}} =
		\sqrt{\frac{\alpha_{i-1}}{\gamma_{(j+1)N+i-1}}}
		\frac{\sqrt{-\discr R_j(x)}}{2 \sqrt{\det Y_j(x)}}.
	\]
	Hence,
	\begin{equation}
		\label{eq:105}
		\lim_{j \to \infty} \sqrt{\frac{\gamma_{(j+1)N+i-1}}{\alpha_{i-1}}} 
		\sqrt{1-\bigg( \frac{\tr Y_j(x)}{2 \sqrt{\det Y_j(x)}} \bigg)^2} = 
		\sqrt{|\tau(x)|},
	\end{equation}
	proving the claim.
	
	To complete the proof of the theorem we write
	\[
		K_{jN+i}(x, x; \eta) = \sum_{i' = 0}^{N-1} K_{i'; j}(x, x; \eta) 
		+ \sum_{i' = i+1}^{N-1} \big(K_{i'; j-1}(x, x; \eta) - K_{i'; j}(x, x; \eta)\big).
	\]
	Observe that by Theorem~\ref{thm:4},
	\[
		\sup_{\eta \in \sS^1} \sup_{x \in K}{\big|K_{i'; j-1}(x, x; \eta) - K_{i'; j}(x, x; \eta)\big|} =
		\sup_{\eta \in \sS^1} \sup_{x \in K}{|u_{jN+i'}(\eta, x)|^2 } \leq c.
	\]
	Moreover, by \cite[Proposition 3.7]{ChristoffelI} and \eqref{eq:142}, for $m, m' \in \NN_0$,
	\[
		\lim_{j \to \infty} \frac{a_{jN+m'}}{a_{jN+m}} = 
		\lim_{j \to \infty} \frac{\gamma_{jN+m'}}{\gamma_{jN+m}}
		=
		\frac{\alpha_{m'}}{\alpha_m},
	\]
	thus, by the Stolz--Ces\`{a}ro theorem, 
	\begin{align*}
		\lim_{j \to \infty} \frac{\rho_{i-1; j}}{\rho_{jN+i'}} 
		&=
		\lim_{j \to \infty} 
		\frac{\frac{\sqrt{\gamma_{jN+i-1}}}{a_{jN+i-1}}}
		{\sum_{k = 1}^N \frac{\sqrt{\alpha_{i'+k} \gamma_{jN+i'+k}}}{a_{jN+i'+k}}} \\
		&=
		\frac{1}{N \sqrt{\alpha_{i-1}}}.
	\end{align*}
	Hence, by \eqref{eq:75} 
	\begin{align*}
		\lim_{j \to \infty}
		\frac{1}{\rho_{jN+i}} K_{jN+i}(x, x; \eta)
		&=
		\lim_{j \to \infty}
		\sum_{i' = 0}^{N-1} \frac{1}{\rho_{i'-1; j}} K_{jN+i'}(x, x; \eta) \cdot 
		\frac{\rho_{i'-1; j}}{\rho_{jN+i}} \\
		&=
		\sum_{i' = 0}^{N-1}
		\frac{\abs{\vphi_{i'}(\eta, x)}^2  a_{j_0N+i'-1} \sinh \vartheta_{j_0N+i'-1}(x)}
		{2 N \big(\sqrt{\alpha_{i'-1} \abs{\tau(x)}}\big)^3}
	\end{align*}
	and the theorem follows.
\end{proof}

\begin{remark}
	\label{rem:2}
	In view of Remark~\ref{rem:3} by Theorem~\ref{thm:7}, all generalized eigenvectors are not square-summable, 
	hence by \cite[Theorem 6.16]{Schmudgen2017} the operator $A$ is self-adjoint.
	Next, by \cite[Theorem 2.1]{Clark1996}, we conclude that $\mu$ is absolutely continuous on $\Lambda_-$ and its 
	density $\mu'$ has the property that for every compact interval $K \subset \Lambda_-$ with non-empty interior there is $c > 0$ such that
	\[
		c^{-1} \leq \mu'(x) \leq c
	\]
	for almost all $x \in K$ (with respect to the Lebesgue measure). Consequently, we have $\sigmaAC(A) \supset \cl(\Lambda_-)$.
	In view of Theorem~\ref{thm:3} we actually have $\sigmaAC(A) = \sigmaEss(A) = \cl(\Lambda_-)$.
\end{remark}

\section{The self-adjointness of $A$} \label{sec:9}
In this section we study the conditions that guarantee that the operator $A$ is 
self-adjoint. The first theorem covers the case when $\Lambda_- \neq \emptyset$.
\begin{theorem}
	\label{thm:6}
	Let $N$ be a positive integer. Let $(\gamma_n : n \in \NN)$ be a sequence of positive
    numbers tending to infinity and satisfying \eqref{eq:87} and \eqref{eq:89}. Let $(a_n : n \in \NN_0)$ and
	$(b_n : n \in \NN_0)$ be $\gamma$-tempered $N$-periodically modulated Jacobi parameters such that $\frakX_0(0)$ is
	a non-trivial parabolic element. Suppose that \eqref{eq:88} holds true with $\varepsilon = \sign{\tr \frakX_0(0)}$.
	If $\Lambda_- \neq \emptyset$, then the Jacobi operator $A$ associated with $(a_n : n \in \NN_0)$ and $(b_n : n \in \NN_0)$
	is self-adjoint if and only if
	\begin{equation}
		\label{eq:67}
		\sum_{n=0}^\infty \frac{\sqrt{\gamma_n}}{a_n} = \infty.
	\end{equation}
\end{theorem}
\begin{proof}
	The case when \eqref{eq:67} is satisfied is covered by Remark~\ref{rem:2}. 
	Assume now that \eqref{eq:67} is \emph{not} satisfied.
	Let $x \in \Lambda_-$. By Theorem \ref{thm:4} and Claim \ref{clm:1}, there are 
	$j_0 \geq 1$ and $c > 0$ such that for all $j \geq j_0$, $i \in \{0, 1, \ldots, N-1\}$,
	and $\eta \in \sS^1$,
	\begin{align*}
		\big|u_{jN+i}(\eta, x) \big|^2 \leq c \frac{\sqrt{\gamma_{jN+i-1}}}{a_{jN+i-1}}.
	\end{align*}
	Hence, every generalized eigenvector associated to $x$ is square-summable. In view of \cite[Theorem 6.16]{Schmudgen2017}
	the operator $A$ is not self-adjoint. This completes the proof.
\end{proof}

The next theorem covers the case when $\Lambda_- = \emptyset$ but $\Lambda_+ \neq \emptyset$.
\begin{theorem} \label{thm:6a}
	Let $N$ be a positive integer. Let $(\gamma_n : n \in \NN)$ be a sequence of positive
	numbers tending to infinity and satisfying \eqref{eq:87} and \eqref{eq:89}. Let $(a_n : n \in \NN_0)$ and
	$(b_n : n \in \NN_0)$ be $\gamma$-tempered $N$-periodically modulated Jacobi parameters such that $\frakX_0(0)$ is
	a non-trivial parabolic element. Suppose that \eqref{eq:88} holds true with $\varepsilon = \sign{\tr \frakX_0(0)}$.
	If $\Lambda_- = \emptyset$ but $\Lambda_+ \neq \emptyset$, then $\Lambda_+ = \RR$, and
	\begin{enumerate}[(i), leftmargin=2em]
		\item if $-\frakS + \sqrt{\frakS^2 + 4 \frakU} < 0$ then the operator $A$ is not self-adjoint;
		\label{case:1}
		\item if $-\frakS + \sqrt{\frakS^2 + 4 \frakU} > 0$ then the operator $A$ is self-adjoint.
		\label{case:2}
	\end{enumerate}
	Moreover, if the operator $A$ is self-adjoint then $\sigmaEss(A) = \emptyset$.
\end{theorem}
\begin{proof}
	If $\Lambda_- = \emptyset$ then $\frakt = 0$ and so $\Lambda_+ = \RR$. Let $i=0$ and $K = \{0\}$. We can repeat the first part of the proof of Theorem~\ref{thm:3}. Now,
	by \eqref{eq:16} and \eqref{eq:40}, there are $j_1 \geq j_0$ and $c > 0$ such that for all $j \geq j_1$,
	\begin{equation}
		\label{eq:73}
		\big|u^+_{jN}(0)\big|^2 + \big|u^+_{jN-1}(0)\big|^2 =
		\big\|\phi^+_{jN}(0)\big\|^2 \geq c \prod_{k = j_0}^j |\lambda_k^+(0)|^2.
	\end{equation}
	Moreover, for all $j \geq j_1$ and $i' \in \{0, 1, \ldots, N-1\}$,
	\begin{equation}
		\label{eq:74}
		\big\|\phi^-_{jN+i'}(0)\big\|^2 \leq c \prod_{k = j_0}^j |\lambda_k^-(0)|^2
		\quad \text{and} \quad
		\big\|\phi^+_{jN+i'}(0)\big\|^2 \leq c \prod_{k = j_0}^j |\lambda_k^+(0)|^2.
	\end{equation}
	By \eqref{eq:71} we obtain
	\begin{equation}
		\label{eq:76}
		\sum_{j = j_1}^\infty \sum_{i' = 0}^{N-1} |u_{jN+i'}^-(0)|^2
		\leq
		c
		\sum_{j = j_1}^\infty \sum_{i' = 0}^{N-1} |u_{jN+i'}^+(0)|^2.
	\end{equation}
	Hence, the operator $A$ is self-adjoint if and only if there is $j_0 > 0$ such that
	\begin{equation}
		\label{eq:72}
		\sum_{j = j_0}^\infty \prod_{k = j_0}^j |\lambda_k^+(0)|^2 = \infty.
	\end{equation}
	Indeed, if \eqref{eq:72} is satisfied then by \eqref{eq:73} the generalized eigenvector $(u_n^+(0) : n \in \NN_0)$
	is not square-summable, thus by \cite[Theorem 6.16]{Schmudgen2017}, the operator $A$ is self-adjoint.
	On the other hand, if \eqref{eq:72} is not satisfied, then by \eqref{eq:74} and \eqref{eq:76}, all generalized
	eigenvectors associated to $0$ are square-summable, thus by \cite[Theorem 6.16]{Schmudgen2017}, the operator $A$ is not
	self-adjoint. The second part of the theorem follows by Theorem \ref{thm:3}.

	Since $(\gamma_{jN} : j \in \NN)$ approaches infinity, there is $j_0 \geq 1$ such that
	\[	
		|\lambda_j^+(0)| = 1 + \sqrt{\frac{\alpha_{N-1}}{\gamma_{(j+1)N-1}}} \frac{\tr R_j(0) + \sqrt{\discr R_j(0)}}{2}.
	\]
	Next, we observe that
	\[
		\lim_{j \to \infty} \frac{\sqrt{\gamma_{jN}}}{j} = 0.
	\]
	Let us consider the case \ref{case:1}. Because $(R_{jN}(0) : j \in \NN)$ converges to $\calR_0(0)$, there is $j_1 \geq j_0$
	such that for all $j \geq j_1$,
	\[
		j\sqrt{\frac{\alpha_{N-1}}{\gamma_{(j+1)N-1}}} \frac{\tr R_j(0) + \sqrt{\discr R_j(0)}}{2} < -1,
	\]
	hence
	\[
		|\lambda_j^+(0)| \leq 1 - \frac{1}{j}. 
	\]
	Consequently,
	\[
		\prod_{k = j_1}^j |\lambda_k^+(0)| \leq \prod_{k = j_1}^j \bigg(1 - \frac{1}{k}\bigg) = \frac{j_1-1}{j},
	\]
	that is \eqref{eq:72} is not satisfied and so the operator $A$ is not self-adjoint.

	The reasoning in the case \ref{case:2} is analogous. Namely, there is $j_1 \geq j_0$ such that for all $j \geq j_1$,
	\[
		j \sqrt{\frac{\alpha_{N-1}}{\gamma_{(j+1)N-1}}} \frac{\tr R_j(0) + \sqrt{\discr R_j(0)}}{2} > 1,
	\]
	hence
	\[
		|\lambda_j^+(0)| \geq 1 + \frac{1}{j}, 
	\]
	and so
	\[
		\prod_{k = j_1}^j |\lambda_k^+(0)| \geq \prod_{k = j_1}^j \bigg(1 + \frac{1}{k}\bigg) = \frac{j+1}{j_1}. 
	\]
	Therefore, \eqref{eq:72} is satisfied and the operator $A$ is self-adjoint.
\end{proof}

\begin{remark}
	If in Theorem~\ref{thm:6a} one has $\Lambda_- = \emptyset$ and $\Lambda_+ \neq \emptyset$, then $A$ is self-adjoint
	if and only if \eqref{eq:72} holds true. Let us emphasize that we cannot treat the case $\Lambda_- = \Lambda_+ = \emptyset$,
	that is $\tau \equiv 0$.
\end{remark}

\section{The $\ell^1$-type perturbations} 
\label{sec:10}
In this section we show how to get the main results of the paper in the presence of certain size $\ell^1$ perturbations.
Let $(\tilde{a}_n : n \in \NN_0)$ and $(\tilde{b}_n : n \in \NN_0)$ be Jacobi parameters satisfying
\[
	\tilde{a}_n = a_n(1 + \xi_n), \qquad
	\tilde{b}_n = b_n(1 + \zeta_n) 
\]
where $(a_n : n \in \NN_0)$ and $(b_n : n \in \NN_0)$ are $\gamma$-tempered $N$-periodically modulated Jacobi parameters such that
$\frakX_0(0)$ is a non-trivial parabolic element, and $(\xi_n : n \in \NN_0)$ and $(\zeta_n : n \in \NN_0)$ are certain
real sequences satisfying
\[
	\sum_{k = 1}^\infty \sqrt{\gamma_n} (|\xi_n| + |\zeta_n|) < \infty.
\]
We follow the reasoning explained in \cite[Section 9]{jordan}. 

Fix a compact set $K \subset \RR$. Let us denote by $(\Delta_n)$ any sequence of $2\times2$ matrices such that
\[
	\sum_{n = 0}^\infty \sup_K \|\Delta_n\| < \infty.
\]
We notice that
\begin{equation}
	\label{eq:80}
	\tilde{B}_n(x) = B_n(x) + \frac{\sqrt{\gamma_n}}{a_n} \Delta_n(x)
\end{equation}
where
\begin{align*}
	\tilde{B}_0(x) &= 
	\begin{pmatrix}
		0 & 1 \\
		-\frac{1}{\tilde{a}_0} & \frac{x - \tilde{b}_0}{\tilde{a}_0}
	\end{pmatrix}, \\
	\tilde{B}_n(x) &=
	\begin{pmatrix}
		0 & 1 \\
		-\frac{\tilde{a}_{n-1}}{\tilde{a}_n} & \frac{x- \tilde{b}_n}{\tilde{a}_n}
	\end{pmatrix}
	,\quad n \geq 1.
\end{align*}
Moreover, for
\[
	\tilde{X}_n = \tilde{B}_{n+N-1} \tilde{B}_{n+n-2} \cdots \tilde{B}_n
\]
we have
\[
	\tilde{X}_n - X_n = \sum_{k = n}^{n+N-1} \frac{\sqrt{\gamma_n}}{a_n} 
	\tilde{B}_{n+N-1} \cdots \tilde{B}_{k+1} \Delta_k B_{k-1} \cdots B_n, 
\]
which together with
\[
	\sup_{n \in \NN_0} \sup_{x \in K} \big(\|B_n(x)\| + \|\tilde{B}_n(x)\| \big) < \infty,
\]
implies that
\begin{equation}
	\label{eq:78}
	\tilde{X}_n  = X_n + \frac{\sqrt{\gamma_n}}{a_n} \Delta_n.
\end{equation}
Suppose that $K \subset \Lambda_+$. Then, by Theorem \ref{thm:1},
\begin{align*}
	Z_{j+1}^{-1} \tilde{X}_{jN+i} Z_j 
	&=
	\varepsilon\bigg(\Id + \sqrt{\frac{\alpha_{i-1}}{\gamma_{(j+1)N+i-1}}} R_j\bigg)
	+
	\frac{\sqrt{\gamma_{jN+i}}}{a_{jN+i}} Z_{j+1}^{-1} \Delta_{jN+i} Z_j.
\end{align*}
Since there is $c > 0$ such that for all $j \in \NN$,
\[
	\sup_K \big\|Z_{j+1}^{-1} \big\| \leq c \sqrt{\gamma_{jN+i}},
	\qquad\text{and}\qquad
	\sup_K \|Z_j\| \leq c,
\]
by setting 
\[
	V_j = \varepsilon \frac{\sqrt{\gamma_{jN+i}}}{a_{jN+i}} Z_{j+1}^{-1} \Delta_{jN+i} Z_j
\]
we get
\[
	Z_{j+1}^{-1} \tilde{X}_{jN+i} Z_j =
	\varepsilon \bigg(\Id + \sqrt{\frac{\alpha_{i-1}}{\gamma_{(j+1)N+i-1}}} R_j + V_j\bigg)
\]
where $(R_j)$ is a sequence from $\calD_1\big(K, \Mat(2, \RR)\big)$ convergent uniformly on $K$ to $\calR_i$, and
\begin{equation}
	\label{eq:79}
	\sum_{j = 1}^\infty \sup_K \|V_j\| < \infty.
\end{equation}
If $(\sqrt{\gamma_n})$ is sublinear and $(\sup_K \|\Delta_n\|)$ belongs to $\ell^1$, for each subsequence there is a
further subsequence $(L_j : j \in \NN_0)$ such that
\[
	\sup_K \|\Delta_{L_j}\| \leq c \frac{1}{\sqrt{\gamma_{L_j+N-1}}}.
\]
Consequently, we can find subsequence $L_j \equiv i \bmod N$ such that
\[
	 \sup_K \|V_{L_j}\| \leq c \frac{\sqrt{\gamma_{L_j+N-1}}}{a_{L_j+N-1}}.
\]
Since \cite[Theorem 4.4]{Discrete} allows perturbation satisfying \eqref{eq:79} we can repeat the proof of Theorem \ref{thm:3}
to get the following statement.
\begin{theorem} \label{thm:9:1}
	Let $N$ be a positive integer. Let $(\gamma_n : n \in \NN)$ be a sequence of positive
	numbers tending to infinity and satisfying \eqref{eq:87} and \eqref{eq:89}. Let $\tilde{A}$ be the Jacobi operator associated
	with Jacobi parameters $(\tilde{a}_n : n \in \NN_0)$ and $(\tilde{b}_n : n \in \NN_0)$ such that
	\[
		\tilde{a}_n = a_n(1+\xi_n), \qquad
	    \tilde{b}_n = b_n(1+\zeta_n),
	\]
	where $(a_n : n \in \NN_0)$ and $(b_n : n \in \NN_0)$ are $\gamma$-tempered $N$-periodically modulated Jacobi parameters
	such that $\frakX_0(0)$ is a non-trivial parabolic element. Suppose that \eqref{eq:88} holds true with
	$\varepsilon = \sign{\tr \frakX_0(0)}$. If
	\[
		 \sum_{n = 0}^\infty \sqrt{\gamma_n}(|\xi_n| + |\zeta_n|) < \infty
	\]
	for certain real sequences $(\xi_n : n \in \NN_0)$ and $(\zeta_n : n \in \NN_0)$, then	
	\[
		\sigmaEss(\tilde{A}) \cap \Lambda_+ = \emptyset.	
	\]
\end{theorem}

Next, let us consider a compact set $K \subset \Lambda_-$. By Theorem \ref{thm:4} and Claim~\ref{clm:1}, there is $c > 0$
such that for all $n \in \NN_0$,
\begin{equation}
	\label{eq:81}
	\sup_K \big\|B_n B_{n-1} \cdots B_0 \big\| \leq c \gamma_n^{1/4} a_n^{-1/2}, 
\end{equation}
and since $\det B_n = \frac{a_{n-1}}{a_n}$, we get
\begin{equation}
	\label{eq:82}
	\sup_K \big\|\big(B_n B_{n-1} \cdots B_0\big)^{-1}\big\| \leq c \gamma_n^{1/4} a_n^{1/2}.
\end{equation}
Moreover, by \eqref{eq:80}
\begin{align*}
	\tilde{B}_n \cdots \tilde{B}_1 \tilde{B}_0
	&=
	\tilde{B}_n \cdots \tilde{B}_1 B_0 \big(\Id + \gamma_0^{1/2} a_0^{-1} B_0^{-1} \Delta_0\Big) \\
	&=
	\tilde{B}_n \cdots \tilde{B}_2 B_1 B_0 \Big(\Id + \gamma_1^{1/2} a_1^{-1} (B_1 B_0)^{-1} \Delta_1 B_0 \Big) 
	\Big(\Id + \gamma_0^{1/2} a_0^{-1} B_0^{-1} \Delta_0\Big) \\
	&=
	B_n \cdots B_1 B_0
	\prod_{j = 0}^n
	\Big(\Id + \gamma_j^{1/2} a_j^{-1} (B_j \cdots B_1 B_0)^{-1} \Delta_j (B_{j-1} \cdots B_1 B_0) \Big)
\end{align*}
thus by \eqref{eq:81} and \eqref{eq:82}
\begin{align*}
	\|\tilde{B}_n \cdots \tilde{B}_1 \tilde{B}_0\|
	&\leq
	\|B_n \cdots B_1 B_0\|
	\prod_{j = 0}^{n} \Big(1 + \gamma_j^{1/2} a_j^{-1} 
	\|(B_j \cdots B_1 B_0)^{-1} \| \cdot \|B_{j-1} \cdots B_1 B_0\| \cdot \|\Delta_j\|\Big) \\
	&\leq
	\|B_n \cdots B_1 B_0\|
	\prod_{j=0}^{n} \Big(1 + c \gamma_j^{3/4} \gamma_{j-1}^{1/4} a_j^{-3/2} a_{j-1}^{1/2} \|\Delta_j\|\Big) \\
	&\leq
	\|B_n \cdots B_1 B_0\|
	\exp\Big(c \sum_{j = 0}^n \|\Delta_j\|\Big).
\end{align*}
Hence,
\begin{equation}
	\label{eq:98}
	\sup_{K}{\|\tilde{B}_n \cdots \tilde{B}_1 \tilde{B}_0\|}
	\leq c \gamma_n^{1/4} a_n^{-1/2}.
\end{equation}
Next, let us introduce the following sequence of matrices
\begin{equation}
	\label{eq:115}
	M_j = \big(B_j B_{j-1} \cdots B_0 \big)^{-1} \big(\tilde{B}_j \tilde{B}_{j-1} \cdots \tilde{B}_0\big).
\end{equation}
Since
\[
	M_{j+1} - M_j = \big(B_{j+1} B_{j} \cdots B_0 \big)^{-1} \big(\tilde{B}_{j+1} - B_{j+1}\big)
	\big(\tilde{B}_j \tilde{B}_{j-1} \cdots \tilde{B}_0\big),
\]
by \eqref{eq:80}, \eqref{eq:82} and \eqref{eq:98}, we obtain
\begin{align*}
	\sup_K{ \|M_{j+1} - M_j\| }
	&\leq
	c \gamma_{j+1}^{1/4} a_{j+1}^{1/2} \gamma_j^{1/2} a_j^{-1} \gamma_j^{1/4} a_j^{-1/2} 
	\sup_K{\|\Delta_{j+1}\|} \\
	&\leq
	c \sup_K{\|\Delta_{j+1}\|}.
\end{align*}
Therefore, the sequence of matrices $(M_j)$ converges uniformly on $K$ to certain continuous mapping $M$, and
\begin{equation}
	\label{eq:112}
	\sup_K{ \big\|M - M_j \big\| } \leq c \sum_{k = j+1}^\infty \sup_K{\|\Delta_k\|}.
\end{equation}
Observe that for each $x \in K$ the matrix $M(x)$ is non-degenerate. Indeed, we have
\begin{align*}
	\det M(x) 
	&= \lim_{j \to \infty} \det M_j(x) \\
	&= \lim_{j \to \infty} \frac{a_j}{\tilde{a}_j} = 1.
\end{align*}
Given $\eta \in \sS^1$, we set
\[
	\eta_n = \frac{M_{n-1} \eta}{\|M_{n-1} \eta\|}.
\]
Let us denote by $(\tilde{u}_n(\eta, x) : n \in \NN_0)$ generalized eigenvector associated to $x \in \RR$ and $\eta \in \sS^1$
and generated by $(\tilde{a}_n : n \in \NN_0)$ and $(\tilde{b}_n : n \in \NN_0)$. Notice that for all $n \in \NN$ and $x \in K$,
by \eqref{eq:108a} and \eqref{eq:115}, we have
\begin{equation}
	\label{eq:83}
	\vec{u}_n\big(\eta_n(x), x\big) = 
	\frac{1}{\|M_{n-1}(x) \eta\|}
	\vec{\tilde{u}}_n\big(\eta, x\big).
\end{equation}
By Theorem \ref{thm:4} and Claim \ref{clm:1},
\[
	\sup_{n \in \NN} \sup_{x \in K} 
	\frac{a_{n+N-1}}{\sqrt{\gamma_{n+N-1}}} \big\| \vec{u}_{n} \big( \eta_n(x),x \big) \big\|^2 < \infty,
\]
which together with \eqref{eq:83} implies that
\begin{equation}
	\label{eq:119}
	\sup_{n \in \NN}\sup_{x \in K}{
	\frac{\tilde{a}_{n+N-1}}{\sqrt{\gamma_{n+N-1}}}
	\big\| \vec{\tilde{u}}_{n} \big( \eta, x \big) \big\|^2} < \infty.
\end{equation}

\begin{theorem}
	\label{thm:9:2}
	Let $N$ be a positive integer. Let $(\gamma_n : n \in \NN)$ be a sequence of positive
	numbers tending to infinity and satisfying \eqref{eq:87} and \eqref{eq:89}. Let $\tilde{A}$ be the Jacobi operator associated
	with Jacobi parameters $(\tilde{a}_n : n \in \NN_0)$ and $(\tilde{b}_n : n \in \NN_0)$ such that
	\[
		\tilde{a}_n = a_n(1+\xi_n), \qquad
	    \tilde{b}_n = b_n(1+\zeta_n),
	\]
	where $(a_n : n \in \NN_0)$ and $(b_n : n \in \NN_0)$ are $\gamma$-tempered $N$-periodically modulated Jacobi parameters
	such that $\frakX_0(0)$ is a non-trivial parabolic element. Suppose that \eqref{eq:88} holds true with 
	$\varepsilon = \sign{\tr \frakX_0(0)}$. If
	\[
		 \sum_{n = 0}^\infty \sqrt{\gamma_n}(|\xi_n| + |\zeta_n|) < \infty
	\]
	for certain real sequences $(\xi_n : n \in \NN_0)$ and $(\zeta_n : n \in \NN_0)$, then 
	for each compact interval $K \subset \Lambda_-$, there are $j_0 \in \NN$ and a~continuous function
	$\tilde{\vphi}: \sS^1 \times K \rightarrow \RR$ such that
	\begin{equation}
		\label{eq:44}
		\sqrt{\frac{\tilde{a}_{jN+i-1}}{\sqrt{\gamma_{jN+i-1}}}}
		\tilde{u}_{jN+i}(\eta, x)
		=
		|\tilde{\vphi}(\eta, x)| 
		\sin\Big(\sum_{k = j_0}^j \theta_k(x) + \arg \tilde{\vphi}(\eta, x)\Big)
		+
		E_j(\eta, x)
	\end{equation}
	where $\theta_k$ are given in \eqref{eq:64} and 
	\[
		\lim_{j \to \infty}
		\sup_{\eta \in \sS^1} \sup_{x \in K}
		|E_j(\eta, x)| 
		=0.
	\]
	Moreover, $\tilde{\vphi}(\eta, x) = 0$ if and only if $[\frak{X}_i(0)]_{21} = 0$.
\end{theorem}
\begin{proof}
	Fix a compact set $K \subset \Lambda_-$. Since
	\[
		\big\| M_{jN+i-1}(x) \eta \big\| = \big\| M(x) \eta \big\| + o_K(1)
	\]
	and
	\[
		\vphi(\eta_{jN+i}(x), x) = \vphi(\eta(x), x) + o_K(1)
	\]
	where
	\[
		\eta(x) = \frac{M(x) \eta}{\|M(x) \eta\|},
	\]
	by \eqref{eq:83} and Theorem \ref{thm:4}, we obtain
	\begin{align*}
		\frac{\tilde{u}_{jN+i}(\eta, x)}{\prod_{k = j_0}^{j-1} |\lambda_k(x)|}
		&=
		\big\| M_{jN+i-1}(x) \eta \big\| 
		\frac{|\vphi(\eta_{jN+i}(x), x)|}{\sqrt{\alpha_{i-1} |\tau(x)|}}
		\sin\Big(\sum_{k=j_0}^{j-1} \theta_k(x) + \arg \vphi(\eta_{jN+i}(x), x) \Big) + o_K(1) \\
		&=
		\big\| M(x) \eta \big\|
		\frac{|\vphi(\eta(x), x)|}{\sqrt{\alpha_{i-1} |\tau(x)|}}
		\sin\Big(\sum_{k=j_0}^{j-1} \theta_k(x) + \arg \vphi(\eta(x), x) \Big) + o_K(1).
	\end{align*}
	In view of Claim~\ref{clm:2} we conclude the proof.
\end{proof}

Now by repeating the proofs of Theorems \ref{thm:7} and \ref{thm:6}, the asymptotic formula \eqref{eq:44} leads to the following 
statement.
\begin{theorem}
	\label{thm:9:3}
	Let $N$ be a positive integer. Let $(\gamma_n : n \in \NN)$ be a sequence of positive
	numbers tending to infinity and satisfying \eqref{eq:87} and \eqref{eq:89}. Let $\tilde{A}$ be the Jacobi operator associated
	with Jacobi parameters $(\tilde{a}_n : n \in \NN_0)$ and $(\tilde{b}_n : n \in \NN_0)$ such that
	\[
		\tilde{a}_n = a_n(1+\xi_n), \qquad
	    \tilde{b}_n = b_n(1+\zeta_n),
	\]
	where $(a_n : n \in \NN_0)$ and $(b_n : n \in \NN_0)$ are $\gamma$-tempered $N$-periodically modulated Jacobi parameters
	such that $\frakX_0(0)$ is a non-trivial parabolic element. Suppose that \eqref{eq:88} holds true with 
	$\varepsilon = \sign{\tr \frakX_0(0)}$. Assume that
	\[
		 \sum_{n = 0}^\infty \sqrt{\gamma_n}(|\xi_n| + |\zeta_n|) < \infty
	\]
	for certain real sequences $(\xi_n : n \in \NN_0)$ and $(\zeta_n : n \in \NN_0)$. Set
	\[
		\tilde{\rho}_n = \sum_{j = 0}^n \frac{\sqrt{\alpha_j \gamma_j}}{\tilde{a}_j}.
	\]
	If $\Lambda_- \neq \emptyset$ then the Jacobi operator $\tilde{A}$ associated to parameters $\tilde{a}$ and $\tilde{b}$ is
	self-adjoint if and only if $\tilde{\rho}_n \to \infty$. If it is the case then the limit
	\[
		\lim_{n \to \infty} \frac{1}{\tilde{\rho}_n} K_n(x, x; \eta)
	\]
	exists locally uniformly with respect to $(x, \eta) \in \Lambda_- \times \sS^1$ and defines a~continuous positive 
	function.
\end{theorem}

\section{Examples} \label{sec:11}
\subsection{Classes of sequences}
\subsubsection{Kostyuchenko--Mirzoev}
Let $N$ be a positive integer and suppose that $(\alpha_n)$ and $(\beta_n)$ are $N$-periodic Jacobi parameters. We define
\begin{equation}
	\label{eq:51a}
	a_n = \alpha_n \hat{a}_n \Big( 1 + \frac{f_n}{\delta_n} \Big), \qquad
	b_n = \beta_n \hat{a}_n \Big( 1 + \frac{g_n}{\delta_n} \Big),
\end{equation}
where $(f_n),(g_n)$ satisfy
\begin{equation} 
	\label{eq:51b}
	\lim_{n \to \infty} |f_n - \frakf_n| = 0, \quad
	\lim_{n \to \infty} |g_n - \frakg_n| = 0,
\end{equation}
for some $N$-periodic sequences $(\frakf_n), (\frakg_n)$, and $(\hat{a}_n), (\delta_n)$ are positive sequences such that
\begin{equation}
	\label{eq:51c}
	\sum_{n=0}^\infty \frac{1}{\hat{a}_n} < \infty \quad \text{and} \quad
	\lim_{n \to \infty} \delta_n \Big( 1 - \frac{\hat{a}_{n-1}}{\hat{a}_n} \Big) = \kappa > 0, \quad
	\lim_{n \to \infty} \delta_n = \infty.
\end{equation}
The sequences $(a_n)$ and $(b_n)$ satisfying \eqref{eq:51a}--\eqref{eq:51c} are called 
\emph{$N$-periodically modulated Kostyuchenko--Mirzoev} Jacobi parameters. This class has been studied before, see e.g. \cite{JanasMoszynski2003, Silva2007a, Discrete, Yafaev2020a}.

\subsubsection{Symmetric birth--death processes}
In \cite[Section 2]{Kreer1994} it is shown that generators of a birth--death process are unitarily equivalent to Jacobi
matrices of the form
\begin{equation} \label{eq:29a}
	a_n = \sqrt{\lambda_n \mu_{n+1}}, \qquad
	b_n = -\lambda_n - \mu_n,
\end{equation}
where $(\lambda_n : n \in \NN_0)$ and $(\mu_{n+1} : n \in \NN_0)$ are some positive sequences.
When $\lambda_n = \mu_{n+1}$, then we obtain a particularly simple class of Jacobi parameters
\begin{equation} \label{eq:29b}
	b_n = -a_{n-1} - a_n.
\end{equation}
If \eqref{eq:29b} is satisfied, we shall refer to Jacobi parameters $(a_n)$ and $(b_n)$ as 
corresponding to a \emph{symmetric} birth--death process. This class has been studied before, see e.g. \cite{DombrowskiPedersen2002a, Dombrowski1997, DombrowskiPedersen1995, Sahbani2008, Monotonic}. In fact, in view of Proposition~\ref{prop:3} below, instead of \eqref{eq:29a} it is sufficient to consider Jacobi parameters
\[
	a_n = \sqrt{\lambda_n \mu_{n+1}}, \qquad
	b_n = \lambda_n + \mu_n.
\]
\begin{proposition} \label{prop:3}
Let $(a_n : n \in \NN_0)$ and $(b_n : n \in \NN_0)$ be sequences of positive and real numbers respectively. Let $A$ and $\hat{A}$ be Jacobi matrices with Jacobi parameters $(a_n : n \in \NN_0), (b_n : n \in \NN_0)$ and $(a_n : n \in \NN_0), (-b_n : n \in \NN_0)$, respectively. 
Then 
\[
	\hat{A} = U (-A) U^{-1},
\]
where $U : \ell^2(\NN_0) \to \ell^2(\NN_0)$ is a unitary operator defined by $(U x)_n = (-1)^n x_n$.
\end{proposition}
The proof of Proposition~\ref{prop:3} is just a simple computation, see e.g. \cite[Lemma 3.5]{Dombrowski1997} or \cite[Proposition 3.5]{DombrowskiPedersen2002a} for more details.

\subsection{The general $N$}
\subsubsection{Kostyuchenko--Mirzoev's class}

\begin{remark} 
\label{rem:6}
Let $N$ be a positive integer and let $(\alpha_n)$ and $(\beta_n)$ be $N$-periodic Jacobi parameters such that $\frakX_0(0)$
is a non-trivial parabolic element. Consider the sequences $(a_n), (b_n)$ satisfying \eqref{eq:51a}--\eqref{eq:51c}, where
\begin{equation} \label{eq:147a}
	\Big( \frac{\delta_n}{\hat{a}_n} \Big), 
	\Big( \delta_n \Big( 1 - \frac{\hat{a}_{n-1}}{\hat{a}_n} \Big) \Big), 
	(f_n), (g_n) \in \calD_1^N,
\end{equation}
and
\begin{equation} \label{eq:147b}
	(\delta_n - \delta_{n-1}), 
	\Big( \frac{1}{\sqrt{\delta_n}} \Big) \in \calD_1^N.
\end{equation}
Then for $\gamma_n = \alpha_n \delta_n$, the hypotheses of Theorem~\ref{thm:1} are satisfied. Moreover,
\[
	\tau(x) \equiv N \kappa 
	-\varepsilon \sum_{i=0}^{N-1} [\frakX_i(0)]_{11} (\kappa + \frakf_i - \frakf_{i-1})
	-\varepsilon \sum_{i=0}^{N-1} \frac{\beta_i}{\alpha_{i-1}} [\frakX_i(0)]_{21} (\frakf_i - \frakg_i).
\]
To see this, let us first observe that
\begin{equation} \label{eq:149}
	\frac{\gamma_n}{a_n} = \frac{\delta_n}{\hat{a}_n} \frac{1}{1+\tfrac{f_n}{\delta_n}},
\end{equation}
which belongs to $\calD^1_N$.
Next, we write
\begin{align*}
	\frac{\alpha_{n-1}}{\alpha_n} - \frac{a_{n-1}}{a_n} 
	&=
	\frac{\alpha_{n-1}}{\alpha_n} 
	\bigg( 1 - \frac{\hat{a}_{n-1}}{\hat{a}_n} \frac{1+\tfrac{f_{n-1}}{\delta_{n-1}}}{1+\tfrac{f_n}{\delta_n}} \bigg) \\
	&=
	\frac{\alpha_{n-1}}{\alpha_n \delta_n} 
	\bigg( 
	\delta_n \Big( 1 - \frac{\hat{a}_{n-1}}{\hat{a}_n} \Big) +
	\frac{\hat{a}_{n-1}}{\hat{a}_n} \frac{f_n - \tfrac{\delta_n}{\delta_{n-1}} f_{n-1}}{1+\tfrac{f_n}{\delta_n}}
	\bigg).
\end{align*}
Hence,
\[
	\bigg( \gamma_n \Big( \frac{\alpha_{n-1}}{\alpha_n} - \frac{a_{n-1}}{a_n} \Big) \bigg) \in \calD_1^N.
\]
Moreover,
\begin{equation} 
	\label{eq:30a}
	\mathop{\lim_{n \to \infty}}_{n \equiv i \bmod N}
	\gamma_n \Big( \frac{\alpha_{n-1}}{\alpha_n} - \frac{a_{n-1}}{a_n} \Big) = 
	\alpha_{i-1} ( \kappa + \frakf_{i} - \frakf_{i-1} ).
\end{equation}
Analogously, we write
\[
	\frac{\beta_n}{\alpha_n} - \frac{b_n}{a_n} = 
	\frac{\beta_n}{\alpha_n} \frac{1}{\delta_n} 
	\bigg( 1 - \frac{1+\tfrac{g_n}{\delta_n}}{1+\tfrac{f_n}{\delta_n}} \bigg) = 
	\frac{\beta_n}{\alpha_n} \frac{1}{\delta_n} \frac{f_n - g_n}{1 + \tfrac{f_n}{\delta_n}},
\]
thus
\[
	\bigg( \gamma_n \Big( \frac{\beta_n}{\alpha_n} - \frac{b_n}{a_n} \Big) \bigg) \in \calD_1^N
\]
and
\begin{equation}
	\label{eq:30b}
	\mathop{\lim_{n \to \infty}}_{n \equiv i \bmod N}
	\gamma_n \Big( \frac{\beta_n}{\alpha_n} - \frac{b_n}{a_n} \Big) = 
	\beta_i (\frakf_i - \frakg_i).
\end{equation}
Next, we easily compute that
\[
	\sqrt{\frac{\alpha_{n-1}}{\alpha_n}} \sqrt{\gamma_n} - \sqrt{\gamma}_{n-1} =
	\sqrt{\frac{\alpha_{n-1}}{\delta_n}} 
	\frac{\delta_n - \delta_{n-1}}{1+\sqrt{\tfrac{\delta_{n-1}}{\delta_n}}}.
\]
Consequently, all the hypotheses of Theorem~\ref{thm:1} are satisfied. Moreover, by \eqref{eq:30a} and \eqref{eq:30b}, we obtain
\[
	\fraks_n \equiv 0, \quad
	\frakr_n \equiv 0
\]
and
\[
	\fraku_n = 
	\alpha_{n-1} (\kappa + \frakf_n - \frakf_{n-1}) 
	(1 - \varepsilon [\frakX_n(0)]_{11})
	- \varepsilon \beta_n (\frakf_n - \frakg_n) [\frakX_n(0)]_{21}.
\]
To compute the value of $\frakt$, observe that by \eqref{eq:147b} the sequence $(\delta_n - \delta_{n-1} : n \in \NN)$ is bounded and by \eqref{eq:51c} $\delta_n \to \infty$. Thus,
\[
	\lim_{n \to \infty} \frac{\delta_{n-1}}{\delta_n} =
	\lim_{n \to \infty} \bigg( 1 - \frac{\delta_{n} - \delta_{n-1}}{\delta_n} \bigg) =
	1.
\]
Next, by \eqref{eq:51c}
\[
	\lim_{n \to \infty} \frac{\hat{a}_{n-1}}{\hat{a}_n} = 
	1 - \lim_{n \to \infty} \frac{\delta_n \big( 1 - \tfrac{\hat{a}_{n-1}}{\hat{a}_n} \big)}{\delta_n} =
	1 - \kappa \lim_{n \to \infty} \frac{1}{\delta_n} = 1.
\]
This together with \eqref{eq:149} implies that
\[
	\frakt = \lim_{n \to \infty} \frac{\delta_n}{\hat{a}_n}
\]
exists. If we had $\frakt > 0$, then there would exist $n_0 \in \NN$ and a constant $c>0$ such that for all $n \geq n_0$
\[
	\frac{1}{\hat{a}_n} \geq c \frac{1}{\delta_n}.
\]
Consequently, by \eqref{eq:51c}
\begin{equation} \label{eq:148}
	\sum_{n=0}^\infty \frac{1}{\delta_n} < \infty.
\end{equation}
On the other hand, we have
\[
	\delta_n \leq \delta_0 + \sum_{k=0}^{n-1} |\delta_{k+1} - \delta_k|.
\]
Thus by the boundedness of $(\delta_n - \delta_{n-1} : n \in \NN)$ we get that for some $c'>0$ one has $\delta_n \leq c'(n+1)$. It leads to a contradiction with \eqref{eq:148}. Hence, $\frakt=0$, which easily gives the formula for $\tau$.
\end{remark}

\subsubsection{Symmetric birth--death class}

\begin{lemma} 
	\label{lem:2}
	Let $N$ be a positive integer. Suppose that $(\tilde{\alpha}_n)$ is a $2N$ periodic sequence of positive numbers such that
	\begin{equation}
		\label{eq:55a}
		\tilde{\alpha}_{0} \tilde{\alpha}_{2} \ldots \tilde{\alpha}_{2N-2} =
		\tilde{\alpha}_{1} \tilde{\alpha}_{3} \ldots \tilde{\alpha}_{2N-1}.
	\end{equation}
	Set
	\begin{equation}
		\label{eq:55b}
		\alpha_n = \tilde{\alpha}_{2n+1} \tilde{\alpha}_{2n+2}, \qquad
		\beta_n = \tilde{\alpha}_{2n}^2 + \tilde{\alpha}_{2n+1}^2.
	\end{equation}
	Then $\tr \frakX_0(0) = 2 \varepsilon$ where $\varepsilon = (-1)^N$. Moreover,
	\begin{equation}
		\label{eq:55d}
		\tr \frakX_0'(0) = -\varepsilon 
		\sum_{i=0}^{N-1} 
		\frac{1}{\alpha_i} 
		\frac{\tilde{\alpha}_{2i}}{\tilde{\alpha}_{2i-1}}
		\sum_{k=0}^{N-1} 
		\prod_{j=i}^{i+k-1} 
		\bigg( \frac{\tilde{\alpha}_{2j}}{\tilde{\alpha}_{2j+1}} \bigg)^2
	\end{equation}
	and 
	\begin{equation}
		\label{eq:55c}
		\big( 1 - \varepsilon [\frakX_n(0)]_{11} \big) 
		\frac{\alpha_{n-1}}{\alpha_n} - \frac{\tilde{\alpha}_{2n}^2}{\tilde{\alpha}_{2n+1} \tilde{\alpha}_{2n+2}} 
		\varepsilon [\frakX_n(0)]_{21} \equiv 0.
	\end{equation}
\end{lemma}
\begin{proof}
We start with the following Claim, which is inspired by \cite[Lemma 2]{Szwarc1995}.
	\begin{claim}
		Let $\ell \geq 0$ and let $\big( \frakp_n^{[\ell]} : n \geq 0 \big)$ be the sequence of orthogonal polynomials associated with recurrence coefficients $(\alpha_{n+\ell} : n \geq 0)$ and
		$(\beta_{n+\ell} : n \geq 0)$, where $(\alpha_n : n \geq 0)$ and $(\beta_n : n \geq 0)$ satisfy \eqref{eq:55b}. Then
		\begin{equation}
			\label{eq:56}
			\frakp_n^{[\ell]}(0) = 
			\frac{\tilde{\alpha}_{2 \ell}}{\tilde{\alpha}_{2n+2\ell} w_n^{[\ell]}} \sum_{k=0}^n \big( w_k^{[\ell]} \big)^2, 
			\quad \text{where} \quad
			w_k^{[\ell]} = (-1)^k 
			\prod_{j=\ell}^{\ell+k-1} \frac{\tilde{\alpha}_{2j}}{\tilde{\alpha}_{2j+1}}
		\end{equation}
	\end{claim}
	To see this we reason by induction over $n \in \NN_0$. For $n=0$ and $n=1$ the formula \eqref{eq:56} can be checked by
	direct computations. Next, let us observe that
	\begin{equation}
		\label{eq:57a}
		-\tilde{\alpha}_{2\ell+2k-1} w_k^{[\ell]} = 
		\tilde{\alpha}_{2\ell+2k-2} w_{k-1}^{[\ell]}, \quad k \geq 1.
	\end{equation}
	By the recurrence relation we have
	\[
		\alpha_{n+\ell} \frakp_{n+1}^{[\ell]}(0) =  
		-\beta_{n+\ell} \frakp_{n}^{[\ell]}(0) 
		-\alpha_{n+\ell-1} \frakp_{n-1}^{[\ell]}(0), \quad n \geq 1.
	\]
	Hence, by the induction hypothesis, \eqref{eq:57a} and \eqref{eq:55b} we obtain
	\begin{align*}
		\alpha_{n+\ell} \frakp_{n+1}^{[\ell]}(0) 
		&= 
		-\beta_{n+\ell} \frac{\tilde{\alpha}_{2\ell}}{\tilde{\alpha}_{2\ell+2n} w_n^{[\ell]}} 
		\sum_{k=0}^n \big( w_k^{[\ell]} \big)^2
		-\frac{\tilde{\alpha}_{2\ell}}{\tilde{\alpha}_{2\ell+2n-2}} \frac{\alpha_{\ell+n-1}}{w_{n-1}^{[\ell]}} 
		\sum_{k=0}^{n-1} \big( w_k^{[\ell]} \big)^2 \\
		&= 
		\frac{\tilde{\alpha}_{2\ell}}{w_{n+1}^{[\ell]}} 
		\bigg( \frac{\tilde{\alpha}_{2\ell+2n}^2 + 
		\tilde{\alpha}_{2\ell+2n+1}^2}{\tilde{\alpha}_{2\ell+2n+1}} 
		\sum_{k=0}^n \big( w_k^{[\ell]} \big)^2 -
		\frac{\tilde{\alpha}_{2\ell+2n}^2}{\tilde{\alpha}_{2\ell+2n+1}} 
		\sum_{k=0}^{n-1} \big( w_k^{[\ell]} \big)^2 \bigg) \\
		&=
		\frac{\tilde{\alpha}_{2\ell}}{\tilde{\alpha}_{2\ell+2n+1} w_{n+1}^{[\ell]}} 
		\bigg( \tilde{\alpha}_{2\ell+2n}^2 \big( w_n^{[\ell]} \big)^2 + 
		\tilde{\alpha}_{2\ell+2n+1}^2 \sum_{k=0}^{n} \big( w_k^{[\ell]} \big)^2 \bigg) \\
		&= 
		\frac{\tilde{\alpha}_{2\ell+2n+1} \tilde{\alpha}_{2\ell}}{w_{n+1}^{[\ell]}} \sum_{k=0}^{n+1} \big( w_k^{[\ell]} \big)^2
	\end{align*}
	and the conclusion follows by once again using \eqref{eq:55b}.

	Next, in view of \cite[Proposition 3]{PeriodicIII} we have
	\begin{equation} 
		\label{eq:57c}
		\frakX_n(0) =
		\begin{pmatrix}
			-\frac{\alpha_{n-1}}{\alpha_n} \frakp_{N-2}^{[n+1]}(0) & \frakp_{N-1}^{[n]}(0) \\
			-\frac{\alpha_{n-1}}{\alpha_n} \frakp_{N-1}^{[n+1]}(0) & \frakp_{N}^{[n]}(0)
		\end{pmatrix}.
	\end{equation}
	Thus, by \eqref{eq:56},
	\begin{align*}
		\tr \frakX_0(0) 
		&= 
		\frakp_{N}^{[0]}(0) - \frac{\alpha_{N-1}}{\alpha_N} \frakp_{N-2}^{[1]}(0) \\
		&=
		\frac{1}{w_N^{[0]}} 
		\sum_{k=0}^N \big( w_k^{[0]} \big)^2 
		-\frac{\alpha_{N-1}}{\alpha_N} \frac{\tilde{\alpha}_2}{\tilde{\alpha}_{2N-2} w_{N-2}^{[1]}} 
		\sum_{k=0}^{N-2}
		\big( w_k^{[1]} \big)^2.
	\end{align*}
	Observe that
	\begin{equation}
		\label{eq:57b}
		w^{[1]}_k = -\frac{\tilde{\alpha}_1}{\tilde{\alpha}_0} w_{k+1}^{[0]}, 
		\quad k \geq 0.
	\end{equation}
	Therefore, by combining \eqref{eq:57a}, \eqref{eq:57b}, \eqref{eq:55b} and using $2N$-periodicity of $(\tilde{\alpha}_n)$  we arrive at
	\begin{align*}
		\tr \frakX_0(0) 
		&= 
		\frac{1}{w_N^{[0]}} 
		\bigg( 
		\sum_{k=0}^N \big( w_k^{[0]} \big)^2 - 
		\sum_{k=1}^{N-1} \big( w_k^{[0]} \big)^2 
		\bigg) \\
		&= \frac{1}{w_N^{[0]}} \Big( 1 + \big(w_N^{[0]} \big)^2 \Big) = 2 (-1)^N
	\end{align*}
	where the last equality follows from \eqref{eq:56} and \eqref{eq:55a}. 

	In view of \cite[Proposition 2.1]{jordan}, \eqref{eq:57c} gives
	\begin{align*}
		\tr \frakX_0'(0) 
		&= 
		\sum_{i=0}^{N-1} \frac{\frakp_{N-1}^{[i+1]}(0)}{\alpha_i} \\
		&= -\varepsilon 
		\sum_{i=0}^{N-1} \frac{1}{\alpha_i} \frac{\tilde{\alpha}_{2i}}{\tilde{\alpha}_{2i-1}}
		\sum_{k=0}^{N-1} \big( w_k^{[i]} \big)^2
	\end{align*}
	where in the last equality we have used \eqref{eq:56} and \eqref{eq:57a}. Now, \eqref{eq:55d} is an easy consequence 
	of \eqref{eq:56}. Since $|\tr \frakX_0(0)|=2$ and $\tr \frakX'_0(0) \neq 0$, Proposition~\ref{prop:2}
	implies that $\frakX_0(0)$ is a non-trivial parabolic element.

	It remains to prove \eqref{eq:55c}. Observe that by \eqref{eq:56}, \eqref{eq:57a}, \eqref{eq:55b} and $2N$-periodicity of $(\tilde{\alpha}_n)$ we get
	\begin{align*}
		\frac{\tilde{\alpha}_{2n}^2}{\tilde{\alpha}_{2n+1} \tilde{\alpha}_{2n+2}} \frakp_{N-1}^{[n+1]}(0)
		+
		\frac{\alpha_{n-1}}{\alpha_{n}} \frakp_{N-2}^{[n+1]}(0)
		&=
		\frac{\tilde{\alpha}_{2n}}{\tilde{\alpha}_{2n+1} w_{N-1}^{[n+1]}} 
		\bigg( 
		\sum_{k=0}^{N-1} \big( w_k^{[n+1]} \big)^2 - 
		\sum_{k=0}^{N-2} \big( w_k^{[n+1]} \big)^2 
		\bigg) \\
		&= 
		\frac{\tilde{\alpha}_{2n}}{\tilde{\alpha}_{2n+1}} w_{N-1}^{[n+1]} 
		= 
		-w_N^{[n+1]} = -\varepsilon.
	\end{align*}
	Hence, by \eqref{eq:57c},
	\[
		\big( 1 - \varepsilon [\frakX_{n}(0)]_{11} \big) 
		\frac{\alpha_{n-1}}{\alpha_{n}} - 
		\frac{\tilde{\alpha}_{2n}^2}{\tilde{\alpha}_{2n+1} \tilde{\alpha}_{2n+2}}
		\varepsilon [\frakX_{n}(0)]_{21} =
		\frac{\alpha_{n-1}}{\alpha_{n}} (1 - \varepsilon^2) = 0
	\]
	which completes the proof.
\end{proof}

\begin{remark}
\label{rem:4}
Let $N$ be a positive integer. Let $(\alpha_n)$ be a positive $N$-periodic sequence. Suppose that $(\gamma_n)$ is 
a positive sequence satisfying
\[
	\Big( \sqrt{\frac{\alpha_{n-1}}{\alpha_n}} \sqrt{\gamma_n} - \sqrt{\gamma_{n-1}} \Big), 
	\Big( \frac{1}{\sqrt{\gamma_n}} \Big) \in \calD_1^N,
\]
and
\[
	\lim_{n \to \infty} \big( \sqrt{\gamma_{n+N}} - \sqrt{\gamma_n} \big) = 0, \quad
	\lim_{n \to \infty} \gamma_n = \infty.
\]
Let us set
\[
	a_n = \gamma_n, \qquad
	b_n = \gamma_{n-1} + \gamma_n.
\]
Then $\beta_n = \alpha_{n-1} + \alpha_n$ and the hypotheses of Theorem~\ref{thm:1} are satisfied with
\begin{equation} 
	\label{eq:13a}
	\frakr_i = 
	\fraks_i = 
	2 \sqrt{\alpha_{i-1}} 
	\mathop{\lim_{n \to \infty}}_{n \equiv i \bmod N}
	\Big( \sqrt{\frac{\alpha_{n-1}}{\alpha_n}} \sqrt{\gamma_n} - \sqrt{\gamma_{n-1}} \Big),
\end{equation}
and
\begin{equation}
	\label{eq:13b}
	\frakt = 1, \qquad \fraku_i \equiv 0.
\end{equation}
In particular,
\begin{equation} 
	\label{eq:13c}
	\tau(x) = 
	-
	\bigg( \sum_{i=0}^{N-1} \frac{\alpha_{i-1}}{\alpha_i} \bigg) 
	\bigg( \sum_{i=0}^{N-1} \frac{1}{\alpha_i} \bigg)
	\cdot x.
\end{equation}
To see this, let us define
\begin{equation} 
	\label{eq:13}
	\tilde{\alpha}_{2n+1} = \tilde{\alpha}_{2n+2} = \sqrt{\alpha_n}, \qquad n \in \ZZ.
\end{equation}
By Lemma~\ref{lem:2}, $\frakX_0(0)$ is a non-trivial parabolic element with $\tr \frakX_0(0) = 2 \varepsilon$ for 
$\varepsilon = (-1)^N$. Next, we have
\begin{equation} 
	\label{eq:86}
	\begin{aligned}
	\frac{\beta_n}{\alpha_n} - \frac{b_n}{a_n}  
	&= 
	\frac{\alpha_{n-1} + \alpha_n}{\alpha_n} - 
	\frac{a_{n-1} + a_n}{a_n} \\
	&=
	\Big( \frac{\alpha_{n-1}}{\alpha_n} - \frac{a_{n-1}}{a_n} \Big).
	\end{aligned}
\end{equation}
Hence, by \eqref{eq:55c} and \eqref{eq:13},
\begin{align*}
	&\Big(\frac{\alpha_{n-1}}{\alpha_n} - \frac{a_{n-1}}{a_n}\Big) 
	\big( 1-\varepsilon [\frakX_n(0)]_{11} \big) -
	\varepsilon \Big( \frac{\beta_n}{\alpha_n} - \frac{b_n}{a_n} \Big) [\frakX_n(0)]_{21} \\
	&\qquad\qquad= 
	\Big(\frac{\alpha_{n-1}}{\alpha_n} - \frac{a_{n-1}}{a_n}\Big) 
	\Big( 1-\varepsilon [\frakX_n(0)]_{11} - \varepsilon [\frakX_n(0)]_{21}  \Big) \equiv 0.
\end{align*}
In particular, the left-hand side belongs to $\calD_1^N$ and $\fraku \equiv 0$. 

Let us observe that
\begin{align*}
	\frac{\alpha_{n-1}}{\alpha_n} - \frac{a_{n-1}}{a_n} 
	&=
	\Big( \sqrt{\frac{\alpha_{n-1}}{\alpha_n}} - \sqrt{\frac{a_{n-1}}{a_n}} \Big)
	\Big( \sqrt{\frac{\alpha_{n-1}}{\alpha_n}} + \sqrt{\frac{a_{n-1}}{a_n}} \Big) \\
	&=
	\frac{1}{\sqrt{a_n}} \Big( \sqrt{\frac{\alpha_{n-1}}{\alpha_n}} \sqrt{a_n} - \sqrt{a_{n-1}} \Big)
	\Big( \sqrt{\frac{\alpha_{n-1}}{\alpha_n}} + \sqrt{\frac{a_{n-1}}{a_n}} \Big).
\end{align*}
Hence,
\begin{equation}
	\label{eq:22}
	\sqrt{\alpha_n a_n} \Big(\frac{\alpha_{n-1}}{\alpha_n} - \frac{a_{n-1}}{a_n} \Big) =
	\sqrt{\alpha_n} \Big( \sqrt{\frac{\alpha_{n-1}}{\alpha_n}} \sqrt{a_n} - \sqrt{a_{n-1}} \Big)
	\Big( \sqrt{\frac{\alpha_{n-1}}{\alpha_n}} + \sqrt{\frac{a_{n-1}}{a_n}} \Big).
\end{equation}
In particular, the left-hand side of \eqref{eq:22} belongs to $\calD_1^N$. Moreover, we get
\[
	\fraks_i =
	2 \sqrt{\alpha_{i-1}}
	\mathop{\lim_{n \to \infty}}_{n \equiv i \bmod N}
	\Big( \sqrt{\frac{\alpha_{n-1}}{\alpha_n}} \sqrt{\gamma_n} - \sqrt{\gamma_{n-1}} \Big),
\]
which together with \eqref{eq:86} gives \eqref{eq:13a}. Finally, by \eqref{eq:13} and \eqref{eq:55d} we get
\[
	\tr \frakX_0'(0) = -\varepsilon
	\sum_{i=0}^{N-1} \frac{1}{\alpha_i} 
	\sum_{k=0}^{N-1} \frac{\alpha_{i-1}}{\alpha_{i+k-1}} =
	-\varepsilon
	\sum_{i=0}^{N-1} \frac{\alpha_{i-1}}{\alpha_i} 
	\sum_{k=0}^{N-1} \frac{1}{\alpha_{k}}.
\]
By Proposition~\ref{prop:1} we obtain $\frakS =0$. Hence, in view of \eqref{eq:4a} the formula \eqref{eq:13c} follows. 
\end{remark}

\subsection{$N=1$}
In this section we specify our results to $N=1$.
\begin{remark}
\label{rem:5}
Suppose that for some $\varepsilon \in \{-1,1\}$
\[
	\bigg( \sqrt{\gamma_n} \Big( \frac{a_{n-1}}{a_n} - 1 \Big) \bigg), 
	\bigg( \sqrt{\gamma_n} \Big( \frac{b_n}{a_n} + 2 \varepsilon \Big) \bigg), 
	\bigg( \gamma_n \Big( 1 + \frac{a_{n-1}}{a_n} + \varepsilon \frac{b_n}{a_n} \Big) \bigg),
	\bigg( \frac{\gamma_n}{a_n} \bigg) \in \calD_1,
\]
where $(\gamma_n)$ is a positive sequence satisfying
\[
	\big( \sqrt{\gamma_n} - \sqrt{\gamma_{n-1}} \big), 
	\bigg( \frac{1}{\sqrt{\gamma_n}} \bigg) \in \calD_1,
\]
and
\[
	\lim_{n \to \infty} \big( \sqrt{\gamma_n} - \sqrt{\gamma_{n-1}} \big) = 0, \qquad
	\lim_{n \to \infty} \gamma_n = \infty.
\]
Let
\[
	\fraks = \lim_{n \to \infty} \sqrt{\gamma_n} \Big( \frac{a_{n-1}}{a_n} - 1 \Big), \qquad
	\frakr = \lim_{n \to \infty} \sqrt{\gamma_n} \Big( \frac{b_n}{a_n} + 2 \varepsilon \Big), \qquad
	\frakt = \lim_{n \to \infty} \frac{\gamma_n}{a_n}
\]
and
\[
	\fraku = \lim_{n \to \infty} 
	\gamma_n \Big( 1 + \frac{a_{n-1}}{a_n} + \varepsilon \frac{b_n}{a_n} \Big).
\]
Then 
\[
	\tau(x) = x \frakt \varepsilon - \fraku + \frac{1}{4} \fraks^2.
\]
In particular, if $\tau(x)$ is not identically zero, then the hypotheses of Theorems \ref{thm:2} and \ref{thm:1} are satisfied.
\end{remark}

\begin{remark} \label{rem:1}
Suppose that sequences $(\tilde{\xi}_n)$ and $(\tilde{\zeta}_n)$ satisfy $(\sqrt{\delta_n} \tilde{\xi}_n), (\sqrt{\delta_n} \tilde{\zeta}_n) \in \ell^1$. Then
\begin{align*}
	\Big( 1 + \frac{f_n}{\delta_n} \Big) ( 1 + \tilde{\xi}_n) &= 
	1 + \frac{f_n}{\delta_n} + \xi_n, \\
	\Big( 1 + \frac{g_n}{\delta_n} \Big) ( 1 + \tilde{\zeta}_n) &= 
	1 + \frac{g_n}{\delta_n} + \zeta_n,
\end{align*}
where $(\sqrt{\delta_n} \xi_n), (\sqrt{\delta_n} \zeta_n) \in \ell^1$. Thus $\ell^1$-type perturbations of \eqref{eq:51a} cover the Jacobi parameters of 
the form 
\[
	\tilde{a}_n = \hat{a}_n \Big( 1 + \frac{f_n}{\delta_n} + \xi_n \Big), \qquad
	\tilde{b}_n = -2 \varepsilon \hat{a}_n \Big( 1 + \frac{g_n}{\delta_n} + \zeta_n \Big),
\]
where $(\sqrt{\delta_n} \xi_n), (\sqrt{\delta_n} \zeta_n) \in \ell^1$.
\end{remark}

\begin{example}
The case when $\hat{a}_n = (n+1)^\kappa$ for some $\kappa > \tfrac{3}{2}$ and $\delta_n = n+1$ was considered in \cite{Yafaev2020a} when $\tau(x) \neq 0$. More specifically, it was assumed that
\[
	a_n = (n+1)^\kappa \Big( 1 + \frac{\frakf}{n+1} + \calO(n^{-2}) \Big), \qquad
	b_n = -2 \varepsilon (n+1)^\kappa \Big( 1 + \frac{\frakg}{n+1} + \calO(n^{-2}) \Big).
\]
In view of Remarks \ref{rem:1} and \ref{rem:6} the above Jacobi parameters are covered by the present article.
Let us emphasize that we can take any $\kappa>1$ and more general perturbations $(\xi_n)$ and $(\zeta_n)$.
\end{example}

\begin{bibliography}{jacobi}
	\bibliographystyle{amsplain}
\end{bibliography}

\end{document}